\documentclass[a4paper,reqno]{amsart}

\usepackage[plainpages=false,colorlinks,linkcolor=black,citecolor=black,urlcolor=black,breaklinks]{hyperref}

\usepackage{amsmath,amsthm, amssymb}
\usepackage{mathrsfs}
\usepackage[shortlabels]{enumitem}
\usepackage{graphicx}
\usepackage[font=small]{caption}
\usepackage{xcolor}
\usepackage{url}

\newcommand{\N}{\mathbb{N}}
\newcommand{\R}{{\mathbb{R}}}
\newcommand{\C}{{\mathbb{C}}}
\newcommand{\Z}{{\mathbb{Z}}}
\newcommand{\D}{{\mathbb{D}}}
\newcommand{\dd}{{{\rm d}}}
\newcommand{\ii}{{\rm i}}
\newcommand{\e}{{\rm e}}

\newcommand{\ov}{\overline}
\newcommand\wt{\widetilde}
\newcommand\wh{\widehat}

\newcommand{\la}{\lambda}

\newcommand{\eps}{\varepsilon}

\newcommand{\spp}{\sigma_{\rm p}}

\newcommand{\essinf}{\operatorname*{ess \,inf}}

\newcommand{\Dom}{{\operatorname{Dom}}}
\newcommand{\Ker}{{\operatorname{Ker}}}
\newcommand{\Ran}{{\operatorname{Ran}}}

\renewcommand{\Re}{\operatorname{Re}}
\renewcommand{\Im}{\operatorname{Im}}

\newcommand{\supp}{\operatorname{supp}}

\newcommand{\Num}{\operatorname{Num}}

\newcommand{\loc}{\mathrm{loc}}
\newcommand{\BigO}{\mathcal{O}}

\newcommand{\lspan}{{\operatorname{span}}}

\newcommand{\CiR}{{C^{\infty}(\R)}}

\newcommand{\Rd}{\mathbb{R}^d}

\newcommand{\Lt}{{L^2}}

\newcommand{\LtOm}{{L^2(\Omega)}}

\newcommand{\LOm}{{L^2(\Omega)}}

\newcommand{\LolocOm}{{L^1_{\loc}(\Omega)}}

\newcommand{\CcRd}{{C_c^{\infty}(\Rd)}}

\newcommand{\CcOm}{{C_c^{\infty}(\Omega)}}

\newcommand{\Ntime}{\partial_t}

\newcommand{\ls}{\lesssim 	}
\newcommand{\gs}{\gtrsim}

\theoremstyle{plain}

\newtheorem{theorem}{Theorem}[section]
\newtheorem{lemma}[theorem]{Lemma}

\newtheorem{proposition}[theorem]{Proposition}
\newtheorem{corollary}[theorem]{Corollary}

\theoremstyle{definition}

\newtheorem{remark}[theorem]{Remark}
\newtheorem{asm-sec}[theorem]{Assumption}

\newcommand\cA{\mathcal A}
\newcommand\cB{\mathcal B}
\newcommand\cD{\mathcal D}

\newcommand\cF{\mathcal F}

\newcommand\cH{\mathcal H}
\newcommand\cK{\mathcal K}

\newcommand\cW{\mathcal W}

\newcommand\frs{\mathfrak s}
\newcommand\frt{\mathfrak t}

\newcommand\frh{\mathfrak h}

\newcommand{\sfX}{{\mathsf X}}
\newcommand{\sfY}{{\mathsf Y}}

\renewcommand{\a}{\alpha}
\renewcommand{\d}{\delta}
\newcommand{\f}{\varphi}
\newcommand{\s}{\sigma}
\newcommand{\g}{\gamma}
\newcommand{\G}{\Gamma}
\renewcommand{\th}{\theta}
\newcommand{\p}{\psi}
\newcommand{\z}{\zeta}
\renewcommand{\o}{\omega}

\newcommand{\Ac}{{\mathcal A}}
\newcommand{\Bc}{{\mathcal B}}
\newcommand{\Cc}{{\mathcal C}}
\newcommand{\Dc}{{\mathcal D}}
\newcommand{\Gc}{{\mathcal G}}
\newcommand{\Hc}{{\mathcal H}}
\newcommand{\Kc}{{\mathcal K}}
\newcommand{\Oc}{{\mathcal O}}
\newcommand{\Sc}{{\mathcal S}}
\newcommand{\Vc}{{\mathcal V}}
\newcommand{\Wc}{{\mathcal W}}

\newcommand{\nr}[1]{\left\Vert #1\right\Vert}
\newcommand{\abs}[1]{\left\vert #1\right\vert}
\newcommand{\inv}{^{-1}}
\newcommand{\diff}{\, \mathrm d}
\newcommand{\1}{\mathds 1}
\newcommand{\set}[1]{\left\{ #1 \right\}}
\newcommand{\stepp}{\noindent {\bf $\bullet$}\quad }
\newcommand{\innp}[2]{\left< #1 , #2 \right>}
\newcommand{\pppg}[1] {\left< #1 \right>}
\newcommand {\limt}[2]{\xrightarrow[#1 \to #2]{}}
\newcommand{\fonc}[4] { \left\{ \begin{array}{ccc} #1 & \to & #2 \\ #3 & \mapsto & #4 \end{array} \right. }
\newcommand{\high}{{\mathsf{high}}}
\newcommand{\low}{{\mathsf{low}}}

\usepackage{mathtools,dsfont}
\mathtoolsset{showonlyrefs}

\numberwithin{equation}{section}
\numberwithin{figure}{section}

\newcommand{\HH}{{\mathcal H}}
\newcommand{\KK}{{\mathcal K}}

\newcommand{\scK}{\mathscr K}
\newcommand{\cC}{\mathcal C}

\newcommand{\AK}{\Ac_\KK}

\newcommand{\TX}{T_{\mathsf X}}
\newcommand{\TY}{T_{\mathsf Y}}

\newcommand{\kk}{{\mathsf K}}
\newcommand{\cc}{\mathfrak c}

\usepackage{soul}

\begin{document}
\title[Semigroup decay for unbounded damping]{Semigroup decay for the wave equation with unbounded damping}

\author{Antonio Arnal}

\address[PS, AA]{Institute of Applied Mathematics, Graz University of Technology, Steyrergasse 30, 8010 Graz, Austria}

\email{siegl@tugraz.at}
\email{aarnalperez01@qub.ac.uk}

\author{Borbala Gerhat}

\address[BG]{Institute of Science and Technology Austria, Am Campus 1, 3400 Kloster\-neu\-burg, Austria}

\email{borbala.gerhat@ista.ac.at}

\author{Julien Royer}

\address[JR]{Institut de Math\'{e}matiques de Toulouse, Universit\'{e} de Toulouse, 118 route de Narbonne, 31062 Toulouse C\'{e}dex 9, France}

\email{julien.royer@math.univ-toulouse.fr}

\author{Petr Siegl}

\thanks{
A.~Arnal acknowledges the support of NAWI Graz for his postdoctoral stay at TU Graz in 2023-2024. This research was partially funded by the Austrian Science Fund (FWF) 10.55776/P 33568-N.
B.~Gerhat has received funding from the European Union’s Horizon 2020
research and innovation programme under the Marie Sk\l odowska-Curie Grant
Agreement No.~101034413, and the EXPRO grant No.~20-17749X
of the Czech Science Foundation.
J.~Royer is supported by the Labex CIMI, Toulouse, France, under grant ANR-11-LABX-0040-CIMI.
We are grateful to Perry Kleinhenz for numerous inspiring discussions.
}

\subjclass[2010]{35L05, 35P05, 34G10, 34L40, 47A10, 47D06, 34D05, 35B40, 26A12}

\keywords{damped wave equation, unbounded damping, resolvent bounds, low frequencies, semi-uniform stability}

\date{\today}

\begin{abstract}
We study the damped wave equation with a damping coefficient which is possibly singular and unbounded at infinity. In general, zero belongs to the spectrum of the corresponding generator, which prevents a uniform (exponential) decay for the energy. However, for initial conditions in a suitable subspace, a detailed analysis of the resolvent norm for low frequencies leads to sharp polynomial time-decay rates for the solution and its energy. 
\end{abstract}

\maketitle

\section{Introduction and main result}
\label{sec:intro}

Consider the damped wave equation (DWE)
\begin{equation}
\label{dwe.2ndorder}
\left\{
\begin{aligned}
\partial_{tt} u(t,x) + a(x) \Ntime u(t,x) &= (\Delta - q(x)) u(t,x), & \quad t & > 0,&  \quad x  & \in \Omega,
\\
u(t,x) & =  0, & \quad t & > 0, & \quad x & \in \partial \Omega,\\
(u(0), \partial_t u(0)) & = (f,g),
\end{aligned}
\\
\right.
\end{equation}
where $\emptyset \neq \Omega \subset \Rd$ is open and the non-negative damping coefficient (absorption index) $a$ and non-negative potential $q$ satisfy a minimal regularity assumption (i.e.~only local integrability):
\begin{equation}\label{asm:1}
0\leq a,q \in \LolocOm.
\end{equation}
 Throughout this paper, we will consider the unique mild solution $u(t)$ for the problem \eqref{dwe.2ndorder} (in the sense given by Proposition \ref{prop:m-dissipative} and Remark \ref{rem:mild.sol} below). Our goal is to obtain the decay rates of several quantities, in particular the energy
\[
E(u;t) = \|\nabla u\|_{L^2}^2 + \|q^\frac12 \nabla u\|_{L^2}^2 + \|\partial_t u\|_{L^2}^2
\]
as well as (weighted) $L^2$-norms of the solution $u(t)$, depending on the coefficients $a$, $q$ and the initial conditions. We are mainly interested in the case where $\Omega$ is unbounded and the coefficients $a$ and $q$ may be large at infinity. We focus here on the case of uniformly positive $a$, which simplifies the analysis of high frequencies (but is not essential for our key low frequency estimates; see Section~\ref{sec:non.up.damping} where further results and comments on the non-uniformly positive case are discussed). An example of (regular) coefficients illustrating the setting is
\begin{equation}\label{intro.ex}
	\Omega = \Rd, \quad a(x) = |x|^2 +1, \quad q(x) \equiv q_0 \geq 0, \quad x \in \Omega.
\end{equation}

\subsection*{Main results}
In order to state our results, we introduce the energy Hilbert space
\begin{equation}\label{H.def.intro}
	\HH := \cW \oplus L^2(\Omega), \quad \|F\|_{\HH}^2 := \|f\|_{\cW}^2 + \|g\|_{L^2}^2, \quad F = (f,g) \in \HH.
\end{equation}
Here $\cW$ is the Hilbert completion  of $\CcOm$ with respect to the inner product $\langle \cdot, \cdot \rangle_{\cW}$ inducing the norm
\begin{equation} \label{W.norm}
	\|u \|_{\cW}^2 := \| \nabla u \|_{L^2}^2 + \| q^{\frac12} u \|_{L^2}^2, \qquad u \in \cW.
\end{equation}
In general, $\cW$ might not be contained in $L^2(\Omega)$ but, in any case, $\nabla u$ and $q^{\frac 12} u$ are (identified with) functions in $L^2(\Omega)$ (see Appendix~\ref{app:W0} for more details).

The essential first observation is that a uniform (and hence exponential) energy decay of solutions is possible only in the case when the damping coefficient $a$ is dominated by $-\Delta+q$ in the sense of \eqref{a.q.rel.bdd.intro} below. 

\begin{proposition} \label{prop:exp-decay}
	Let $\Omega \subset \Rd$ be non-empty and open, and let $0\leq a,q \in \LolocOm$. If there exist $\nu > 0$ and $C > 0$ such that for all $F=(f,g) \in \HH$ and $t \geq 0$ the solution $u(t)$ of \eqref{dwe.2ndorder} satisfies
	\begin{equation} \label{eq:exp-decay}
		\|\nabla u(t)\|_{L^2} + \|q^\frac12 u(t)\|_{L^2} + \|\partial_t u(t)\|_{L^2} \leq C \e^{-\nu t} \|F\|_{\HH},
	\end{equation}
	then there exists $M > 0$ such that
	\begin{equation}\label{a.q.rel.bdd.intro}
		\|a^\frac 12 \varphi\|_{L^2}^2 + \|\varphi\|_{L^2}^2 \leq M \left(\|\nabla \varphi\|_{L^2}^2 + \|q^\frac 12 \varphi\|_{L^2}^2 \right), \qquad \varphi \in \CcOm.
	\end{equation}
	Moreover, if $a(x) \geq a_0 > 0$ a.e.~in $\Omega$, then the reverse implication holds, so that \eqref{eq:exp-decay} and~\eqref{a.q.rel.bdd.intro} are equivalent in this case.
\end{proposition}

Here we are interested in cases like \eqref{intro.ex}, where \eqref{a.q.rel.bdd.intro} does not hold and uniform energy decay cannot prevail.
In our main theorem below, we hence consider initial conditions in a subspace of $\cH$, namely in the Hilbert space
\begin{equation}\label{KK.def}
	\KK  := \big\{ (f,g)  \in \big(H^1_0(\Omega) \cap \Dom(q^\frac 12) \big)\times L^2(\Omega) \, : \,a f \in \LolocOm, \ af+g \in \cW^*\big\},
\end{equation}
endowed with the norm
\begin{equation} \label{KK.norm}
	\|F\|_{\KK}^2 = \|(f,g)\|_{\KK}^2   := \|(f,g)\|^2_{\HH} + \|f\|_{L^2}^2 + \|af + g\|_{\cW^*}^2,
\end{equation}
(see Section~\ref{ssec:Ran.A.K} below).
Since for $F \in \KK$ we have $a f + g  \in L^1_\loc(\Omega) \hookrightarrow \cD'(\Omega)$, the condition $a f + g \in \cW^*$ means that $a f + g$ has an extension to a bounded functional on $\cW$ (which is unique due to the density of $\CcOm$ in $\cW$). 

Our main result reads as follows.

\begin{theorem}
	\label{thm:decay}
	Let $\Omega \subset \Rd$ be non-empty and open, let $0\leq a,q \in \LolocOm$ and let $a(x) \geq a_0 > 0$ a.e.~in $\Omega$.
	Then there exists $C>0$ such that, for any initial data $F \in \KK$ and all $t \ge 0$, the solution $u(t)$ of \eqref{dwe.2ndorder} decays in time as
	\begin{align}
		%\|\partial_t u(t)\|_{L^2(\Omega)} +
		\|\nabla u(t)\|_{L^2} + \|q^\frac12 u(t)\|_{L^2}  &\leq C \|F\|_{\KK} \langle t \rangle^{-1},
		\label{u.energy.est}
		\\
		\|\partial_t u(t)\|_{L^2}  &\leq C \|F\|_{\KK} \langle t \rangle^{-\frac 32},
		\label{u.partialt.est}
		\\
		\|a^\frac12 u(t)\|_{L^2}+\|u(t)\|_{L^2} & \leq C \|F\|_{\KK}  \langle t \rangle^{- \frac 12}. \label{u.L2.est}
	\end{align}
	Moreover, if $\Omega$ is an exterior domain and, for some $\beta>0$ and $c > 0$,
	\begin{equation}\label{a.unbdd.intro}
		a(x) \geq c \langle x \rangle^\beta \quad \text{a.e.~in } \Omega,
	\end{equation}
	then, for all $t \geq 0$,
	\begin{align}\label{dtu.L2.est.gamma}
		\|\partial_t u(t)\|_{L^2} & \leq C \|F\|_{\KK} \langle t \rangle^{-\frac 32  - \frac{\beta}{2(2+\beta)}},
		\\
		\label{u.L2.est.gamma}
		\|u(t)\|_{L^2} & \leq C \|F\|_{\KK} \langle t \rangle^{-\frac12  - \frac{\beta}{2(2+\beta)}}.
	\end{align}
\end{theorem}

Notice that, combining \eqref{u.energy.est} and \eqref{u.partialt.est},  the energy  decays as
\begin{equation}
E(u;t) \leq C \| F \|_\KK^2 \langle t \rangle^{-2}, \qquad t \geq 0.
\end{equation}
Moreover, we get a similar estimate for the $L^2$-norm of $u(t)$ whenever it is controlled by the left-hand side of \eqref{u.energy.est} (for instance if $q$ is uniformly positive or if $\Omega$ is such that the Poincar\'e inequality holds, see also Remark~\ref{rem:super-Poincare} below).

We show in Section \ref{sec:optimal} that the estimates \eqref{u.energy.est}--\eqref{u.L2.est} are sharp in the model case $\Omega = \Rd$, $q (x) \equiv 0$, $a (x) \equiv 1$ and can thus not be improved in the general setting.

\subsection*{Strategy} 
The strategy of our proof relies on the semigroup point of view and the related spectral and resolvent analysis. More precisely, the energy decay is obtained via resolvent estimates for the damped wave operator in $\KK$ (see Section~\ref{sec:dwe}) specifically for the spectral parameter $\la \in \ii \R \setminus \{0\}$ (see Section~\ref{sec:decay} and in particular Theorem~\ref{thm:exp.res} below). From this spectral point of view, we can deal separately with the contributions of high ($\la \to \pm \ii \infty$) and low ($\la \to 0$) frequencies.

The contribution of high frequencies is not our primary concern in this paper. For uniformly positive $a$ (even with minimal regularity), the resolvent norm of the generator is bounded for high frequencies (see Proposition~\ref{prop:high.pos}). Thus the uniform exponential energy decay would follow if the resolvent were also bounded at low frequencies, which is the main issue here (see below).

In fact, our results generalize in a straightforward way if the resolvent remains bounded for high frequencies, which can hold also without a uniformly positive $a$. Namely, this is known in $d=1$ for a regular $a$ which is unbounded at $\pm \infty$ and possibly zero on a bounded subset of $\R$ (see \cite[Thm.~3.5]{arnal2026resolvent-dwe} or Theorem~\ref{thm:antonio} below), as well as for the example $a(x)= |x|^\beta$, $\beta>0$, $x \in \Rd$, discussed in Section~\ref{sec:non.up.damping}. These conclusions are natural as they suggest that the resolvent at $\pm \ii \infty$ remains bounded for unbounded damping also in higher dimensions as long as the usual Geometric Control Condition (GCC) is satisfied. We recall that the GCC says that all the rays of light (or classical trajectories), along which high frequency waves propagate, go through the damping region $\{a>0\}$ (see references below).

On the other hand, the presence of an undamped trajectory in a waveguide (see the example in Section~\ref{ssec:noGCC}) yields various rates of resolvent growth at $\pm \ii \infty$ depending on the behavior of $a$ in the neighborhood of this trajectory (in line with the conclusions in \cite{Leautaud-2017-26}). Nonetheless, the singularity of the resolvent for high frequencies in these examples is always milder than the one around zero (see below) and so the resulting energy decay rate originates in the low frequencies.

The main analysis in this paper concerns low frequencies ($\la \to 0$). Note that if \eqref{a.q.rel.bdd.intro} is not satisfied, then zero is in the essential spectrum of the generator (see Corollary~\ref{cor:0.sp} for details and Theorem~\ref{thm:A.basic} for further claims on the real essential spectrum). This is the effect responsible for non-exponential rates and was noticed first in \cite{Freitas-2018-264}.

The polynomial rates in Theorem \ref{thm:decay} are obtained from resolvent estimates for $\la \to 0$, which are proved without assuming the uniform positivity of $a$; for instance~$\Omega = \Rd$ and $a(x) \geq 1$ for a.e.~$|x| \geq 1$ is covered (see Assumption~\ref{asm:a.1}, Theorem~\ref{thm:low} and also Remark~\ref{rem:super-Poincare}). Although the damped wave operator is highly non-self-adjoint for unbounded damping (see Figure~\ref{fig:x2} and \cite{Arifoski-2020-52}), the key observation is that the resolvent estimates around zero in Theorem~\ref{thm:low} can be reduced to a self-adjoint spectral problem. This allows for classical general tools (Neumann bracketing \cite[Chap.~XIII.15]{Reed4} and asymptotic perturbation theory \cite[Chap.~VIII]{Kato-1966}). Moreover, the recent new functional-analytic understanding of the arising operators in \cite{gerhat2024schur} enabled us to work with the minimal $L^1_{\rm loc}$ assumptions on $a$ and $q$.

\subsection*{Literature}

Most existing results deal with the case where $a$ (and $q$) are bounded. When $\Omega$ is bounded, then \eqref{a.q.rel.bdd.intro} holds by the Poincar\'e inequality and low frequencies are not an issue. It is proved in \cite{RauchTay74} (see also \cite{Ralston69}) that the GCC is essentially necessary and sufficient for the uniform (exponential) energy decay (see also \cite{BardosLebRau92} for damping at the boundary). On the other hand, when the GCC is violated, the decay cannot be uniform and some regularity is required for the initial condition. Under minimal assumptions on the damping the decay is at least logarithmic in time (see \cite{Lebeau96}), and for intermediate situations, depending on the geometry of the undamped rays of light, various polynomial decay rates have been obtained (see for instance \cite{BurqHit07,AnantharamanLea14,Leautaud-2017-26,Kleinhenz19,KleinhenzWan} and references therein). These papers also raise the question of the link between the regularity of the damping where it vanishes and the decay rate of the wave. See also \cite{BurqGer20} for a refined version of the GCC for rough dampings.

If $q (x) \equiv 1$ (then we refer to \eqref{dwe.2ndorder} as the Klein-Gordon equation) and $a$ is bounded, then \eqref{a.q.rel.bdd.intro} always holds, even if $\Omega$ is unbounded, and the results are similar to the bounded case above. See \cite{BurqJol16} for the minimal logarithmic decay with loss of regularity under some weak assumption on $a$ and the exponential uniform decay under a suitable version of the GCC.

When $q (x)\equiv 0$ and $\Omega$ is unbounded and such that the Poincar\'e inequality does not hold, then \eqref{a.q.rel.bdd.intro} fails and the decay cannot be uniform because of the lack of decay in the contribution of low frequencies. Of course, depending on the geometry and the damping, we can still have independently a lack of decay in the contribution of high frequencies.

In particular, when $a (x)\equiv 0$, there is no decay at all. Nonetheless, one can consider a very closely related problem of the local energy decay (see \cite{Burq98,BoucletBur21} without damping, \cite{FahsRoy} when the damping is small at infinity, and references therein).

As already mentioned, the simplest model for a damped wave equation without \eqref{a.q.rel.bdd.intro} is the case with $\Omega = \Rd$ and $a (x)\equiv 1$, which will be discussed in Section \ref{sec:optimal}. Estimates for the solution $u (t)$ and its derivatives in Lebesgue spaces have first been proved in \cite{Matsumura76}. Similar results, as well as the diffusive phenomenon (the damped wave behaves for large times like a solution of a heat equation, see Section \ref{sec:optimal}) have then been proved in various settings.
See for instance \cite{Ikehata02,AlouiIbrKhe15} in an exterior domain, \cite{TodorovaYor09,IkehataTodYor13,wakasugi14,SobajimaWak16} for slowly decaying dampings, \cite{Royer18,MallougRoy18} in a waveguide and \cite{JolyRoy18} in a periodic setting. We also refer to \cite{ChillHar04,RaduTodYor11,RaduTodYor16,Batty-2016-270,Nishiyama16} for results in abstract settings.

In this paper, we are interested in the wave equation with unbounded dampings. We refer to \cite{Freitas-2018-264,Arifoski-2020-52} for earlier studied spectral properties of the wave operator. Closer to our problem, Ikehata and Takeda have proved in \cite{ikehata2020uniform} energy decay as in Theorem \ref{thm:decay} in the case $\Omega=\Rd$ with $d \geq 3$, $q (x) \equiv 0$ and uniformly positive $a \in C(\Rd)$. More precisely, by employing a modified Morawetz (multiplier) method, they established that for initial data satisfying
\begin{equation}\label{v.IT.intro}
	F=(f,g) \in \big(H^1(\Rd) \cap L^1(\Rd)\big) \times \big(L^2(\Rd) \cap L^1(\Rd) \big), \quad a f \in L^1(\Rd) \cap L^2(\Rd),
\end{equation}
weak solutions of \eqref{dwe.2ndorder} decay both in energy and $L^2$-norm as
\begin{equation}\label{IT:rates.intro}
	\begin{aligned}
		E(u;t)= \|\partial_t u(t)\|_{L^2}^2 + \|\nabla u(t)\|_{L^2}^2 &\leq C (F)^2 \langle t \rangle^{-2},
		\\
		\|u(t)\|_{L^2} &\leq C(F) \langle t \rangle^{- \frac 12}, \qquad t\ge0,
	\end{aligned}
\end{equation}
(see \cite[Thm.~1.2]{ikehata2020uniform}). The used method does not apply in dimensions $d=1,2$ and the decay in this case remained open (see \cite[Sec.~3]{ikehata2020uniform} for a detailed explanation). Our result covers all dimensions $d$, thus it solves the open problem, and it also allows for a minimal regularity of $a$ and $q$ as well as a general open $\Omega$ instead of $\Rd$. Moreover, in this generality, the same rates as in \eqref{IT:rates.intro} are obtained for the $L^2$-norm of $u(t)$ and its gradient, and we have a better decay for $\partial_t u(t)$. A detailed comparison of Theorem~\ref{thm:decay} and the result in \cite{ikehata2020uniform}, indicating also the origin of the restriction $d \geq 3$ coming from the Sobolev inequality, as well as the relation of $C(F)$ and $\|F\|_{\KK}$, can be found in Section~\ref{sec:comp.IT} below.

In \cite{sobajima2018diffusion}, Sobajima and Wakasugi obtained an improved estimate when the damping is indeed unbounded at infinity. On exterior domains $\Omega \subset \Rd$ with $d \geq2$ and $\partial \Omega$ smooth, $q(x)\equiv0$ and the damping $a \in C^2(\ov \Omega)$ satisfying $a >0$ on $\ov \Omega$ and
\begin{equation}
	\lim_{|x| \to \infty} \frac{a(x)}{|x|^\beta} = a_0>0
\end{equation}
with $\beta > 0$, they prove the following weighted $L^2$-norm decay for a compactly supported initial condition $F \in ( H^2(\Omega) \cap H^1_0(\Omega)) \times H^1_0(\Omega)$ (with an arbitrary $\delta>0$ when $d >2$ and $\delta=0$ in case that $d=2$)
\begin{equation}\label{Sob.Wak.rates}
	\| a^\frac 12 u(t)\|_{L^2} \leq \wt C(F) \langle t \rangle^{- \frac{d}{2(2+\beta)} - \frac{\beta}{2(2+\beta)} + \delta}, \qquad t>0,
\end{equation}
(see \cite[Cor.~1.2]{sobajima2018diffusion}). Their method relies on the comparison with the solutions of the associated heat equation and leads to a decay rate depending on the spatial dimension $d$. In comparison, our estimates \eqref{dtu.L2.est.gamma}--\eqref{u.L2.est.gamma} seem to be the first results exhibiting the effect of the unboundedness of $a$ at infinity \emph{without} restricting to initial data with compact support (unlike \eqref{Sob.Wak.rates}, our rates are $d$-independent).

\subsection*{Plan of the paper}
In Section \ref{sec:dwe}, we recall how the generator of the damped wave equation is defined in the presence of unbounded coefficients, and we give its basic properties in the spaces $\HH$ and $\KK$. In Section \ref{sec:low}, we prove the resolvent estimates for low frequencies and the associated coercivity for the Schur complement. Section~\ref{sec:high} is concerned with the high frequency resolvent estimates for uniformly positive damping. In Section \ref{sec:decay}, we show how we can deduce the time decay estimates from the resolvent estimates. In Section \ref{sec:optimal}, we prove that our rates are sharp in the model case with constant coefficients. The optimality in this case is based on the diffusive phenomenon and decay estimates for the heat equation.
Section~\ref{sec:non.up.damping} is devoted to the case of non-uniformly positive damping; we discuss some cases where GCC holds and the high frequency estimates remain valid. We also present an example where the GCC is violated (but nevertheless the appearing singularity of the resolvent does not spoil the decay rate of the semigroup).
Finally, we compare in Section \ref{sec:comp.IT} our approach with the previous result by Ikehata and Takeda which motivated our study, and some technical details about the space $\Wc$ are collected Appendix \ref{app:W0} at the end of the paper.

\subsection*{Notation}
\label{ssec:notation}
We write $\C_+ = \{\la \in \C: \Re\la  > 0\}$, $\C_-= - \C_+$, $\D = \{\la \in \C \, : \, |\la| < 1\}$, $\D_+ = \D \cap \C_+$. The resolvent set, spectrum, point spectrum and continuous spectrum of a linear operator $T$ are denoted by $\rho(T)$, $\sigma(T)$, $\spp(T)$ and $\sigma_{\rm c}(T)$, respectively. The essential spectra $\sigma_{\mathrm{e}j}(T)$ of $T$, $j=1,\dotsc,5$, are defined as in~\cite[Chap.~IX]{EE}.
To avoid introducing multiple constants whose exact values are inessential for our purposes, we write $a \lesssim b$ to indicate that, given $a,b \ge 0$, there exists a constant $C>0$, independent of any relevant variable or parameter, such that $a \le Cb$. The relation $a \gtrsim b$ is defined analogously whereas $a \approx b$ means that $a \lesssim b$ \textit{and} $a \gtrsim b$. Finally, we write $\langle t \rangle = \sqrt{1+t^2}$ for $t \in \R$.

\section{Generator of the damped wave equation for unbounded damping}
\label{sec:dwe}

\subsection{Definition of the generator and solution of the damped wave equation} \label{sec:generator-mild-solution}

As described in the introduction, we study the time dependent problem \eqref{dwe.2ndorder} from a spectral point of view. Formally, we can rewrite \eqref{dwe.2ndorder} as
\begin{equation}\label{Cauchy}
	\left\{ \begin{aligned}
		\partial_t U(t)
		& = \cA U, \\
		U(0)  & = F,
	\end{aligned}
	\right.
\end{equation}
where $ U(t) = (u(t),\partial_t u(t))$, $F = (f,g)$ and
\begin{equation}\label{A.action}
	\Ac = \begin{pmatrix}
		0 & I \\
		\Delta - q & -a
	\end{pmatrix}.
\end{equation}
It is natural to see $\Ac$ as an operator in the space $\HH = \cW \oplus L^2(\Omega)$ defined as in~\eqref{H.def.intro}, so that $\nr{U(t)}^2_{\HH}$ is then precisely the energy of $u(t)$. Note that if $\Omega$ and $q$ are such that
\begin{equation}\label{Om.q.Poincare}
		\|\nabla u\|_{L^2}^2 + \|q^\frac12 u\|_{L^2}^2 \gs \|u\|_{L^2}^2, \qquad u \in \CcOm,
\end{equation}
e.g.~when the Poincar\'e inequality applies or when $q(x)\geq q_0 >0$ a.e.~in $\Omega$, then $\cW$ is a subspace of $L^2(\Omega)$ and coincides (with equivalent norms) with $H^1_0(\Omega) \cap \Dom(q^\frac 12)$ equipped with the natural norm	%
\begin{equation}\label{W.form.dom}
		( \| \nabla u \|_{L^2}^2 +  \| q^{\frac12} u \|_{L^2}^2 + \|u\|_{L^2}^2)^\frac12,	
\end{equation}
i.e.~the form domain of the self-adjoint Dirichlet realization of $-\Delta+q$ in $L^2(\Omega)$.

Due to the minimal assumptions on the coefficients \eqref{asm:1}, where no relative boundedness (in any sense) of $a$ with respect to $-\Delta+q$ is available, the recent method of dominant Schur complements~\cite{gerhat2024schur} is employed to find a realization of $\cA$ which generates a $C_0$-contraction semigroup. To this end, we introduce the Hilbert space
\begin{equation}
	\label{Dt.def}
	\cD_\frt := H_0^1(\Omega) \cap \Dom (q^\frac12) \cap \Dom (a^\frac12),
\end{equation}
equipped with the inner product arising from
\begin{equation}
	\label{Dt.norm}
	\|u\|^2_{\cD_\frt} := \|\nabla u \|_{L^2}^2 + \| q^{\frac12} u \|_{L^2}^2 + \| a^{\frac12} u\|_{L^2}^2 + \| u\|_{L^2}^2. 
\end{equation}
We define $\cA$ as an operator in $\HH$ with the domain
\begin{equation} \label{A.dom}
\Dom(\cA) = \left\{F= (f, g) \in \cW \times \cD_\frt \, : \, (\Delta-q) f -ag \in L^2(\Omega) \right\};
\end{equation}
(see Appendix~\ref{app:W0} for details about $\cW$ and the meaning of $(-\Delta + q)f$ when $f \in \cW$).

\begin{proposition}[{\cite[Thm.~4.2]{gerhat2024schur}}]  \label{prop:m-dissipative}
The operator $\cA$ defined by \eqref{A.action} and \eqref{A.dom} is densely defined and m-dissipative in $\HH$, hence it generates a strongly continuous contraction semigroup $\e^{t \cA}$ on $\HH$. 
\end{proposition}

\begin{remark}
	\label{rem:mild.sol}
	Recall that for $F \in \HH$, the function $U(t) = \e^{t\Ac} F$ is the unique mild solution of the Cauchy problem \eqref{Cauchy} in $\HH$ (and note that $U(t)$ is even the unique classical solution of~\eqref{Cauchy} in $\HH$ if $F \in \Dom (\cA)$).
		
	In detail, being a mild solution in $\HH$ means that $U : [0,\infty) \to \HH$ is continuous, its primitive satisfies 
	\begin{equation}
		\int_0^t U(s)\, \dd s \in \Dom(\cA), \qquad t \geq 0,
	\end{equation}
	and it solves the integrated Cauchy problem
	\begin{equation}
		\label{Cauchy.int}
		U(t)
		= \cA
		\left(\int_0^t U(s)\, \dd s\right)
		+ F, \qquad t \geq 0.
	\end{equation}
	The first component of the vector $U(t) = (u(t), v(t))$ is our solution for the DWE \eqref{dwe.2ndorder}. 	
	Considering the action of $\cA$ in~\eqref{A.action}, it reads
	\begin{equation}
		u(t) = \int_0^t v(s) \, \dd s +f , \qquad t \ge 0.
	\end{equation}
	This in particular implies
	\begin{equation} 
		\label{der-u-f2}
		u \in C^1 ([0,\infty); \LOm), \qquad U(t) = (u(t), \partial_t u(t)), \qquad t \ge 0.
	\end{equation}
	In Section~\ref{sec:comp.IT}, we discuss further properties of $u(t)$ for special initial conditions $F \in \cK$, as well as its relation to the weak solutions used in \cite{ikehata2020uniform}.
\end{remark}

\subsection{The Schur complement}

The space $\cD_{\frt}$ introduced in \eqref{Dt.def} and \eqref{Dt.norm} arises as the domain of the forms
\begin{equation}\label{tla.def}
\frt_\la [u] := \| \nabla u \|_{L^2}^2 + \| q^{\frac12} u \|_{L^2}^2 + \la \| a^{\frac12} u \|_{L^2}^2 + \la ^2 \|u\|_{L^2}^2, 
\quad 
\Dom(\frt_\la)  := \cD_\frt, 
\end{equation}
which are coercive on $\cD_{\frt}$ (after a shift and rotation) for $\la \in \C \setminus (-\infty, 0]$ (see~\cite[Lem.~4.13]{gerhat2024schur}). Thus the $\frt_\la$  define Schr\"odinger operators
\begin{equation}\label{Tla.def}
\begin{aligned}
T_\la &=  -\Delta + q + \la a + \la^2, 
\\
\Dom(T_\la) &= \left\{ u \in \cD_\frt \, : \, (-\Delta + q + \la a) u \in L^2(\Omega) \right\}, \quad \la \in \C \setminus (-\infty, 0],
\end{aligned}
\end{equation}
which are after a shift and rotation m-sectorial. Notice that $T_\la$ is essentially the Schur complement of $\cA$.
We will deduce the spectral properties of $\Ac$ from the analysis of $T_\la$. We recall the following spectral properties for $\Ac$ and $T_\la$.
\begin{theorem}[{\cite[Thm.~4.2]{gerhat2024schur}, \cite[Sec.~2--4]{Freitas-2018-264}}]\label{thm:A.basic}
Let $\cA$ and $T_\la$ be as in \eqref{A.action}, \eqref{A.dom} and \eqref{Tla.def}, respectively. Then the following claims hold.
\begin{enumerate}[\upshape (i), wide]
	\item 
	\label{item:spec.equiv} We have the spectral equivalence for $* \in \{ \ , {\rm p}, {\rm e2}\}$
	
	\begin{equation}\label{spec.equiv}
		\forall \, \la \in \C \setminus (-\infty,0] \,\,\, : \,\,\, \la \in \sigma_{*} (\cA) \,\,  \iff \,\, 0 \in \sigma_* (T_\la).
	\end{equation}
	\item If $ \la \in \sigma(\cA) \setminus (-\infty,0]$, then $\ov \la \in \sigma(\cA)$, 
	\begin{equation}
		\Re \la \leq - \frac12 \essinf_{x\in \Omega} a(x) \quad \text{and} \quad |\la|^2 \geq \inf( \sigma ( (-\Delta+q)_{\rm D}) ),
	\end{equation}
	where $(-\Delta+q)_{\rm D}$ is the self-adjoint Dirichlet realization of $-\Delta+q$ in $L^2(\Omega)$.

	\item If $T_\la$ has compact resolvent for some $\la \in \C \setminus (-\infty,0]$, then $\sigma(\cA) \setminus (-\infty,0]$ is purely discrete, i.e.~it consists of isolated eigenvalues with finite algebraic multiplicities. In particular, this holds if $\Omega$ is bounded or if
	\begin{equation}\label{a.unbd}
		\lim_{R \to \infty} \essinf_{|x|>R, \, x \in\Omega} a(x) = +\infty.
	\end{equation}
	\item Suppose in addition that $a,q \in C^1(\Omega)$ and $\Omega$ contains a sector
	\begin{equation}
		S_\delta:=\left\{(x_1,x') \in \R \times \R^{d-1} \, : \,  x_1>0, |x'| < \delta x_1\right\} \subset \Omega
	\end{equation}
	for some $\delta>0$, and that $a$ can be decomposed as $a(x) = a_{\rm r}(|x|) + a_{\rm p}(x)$ such that
	\begin{equation}
		\begin{aligned}
			\lim_{r \to \infty} a_{\rm r}(r) &=  \infty, \quad  \lim_{r \to \infty }\frac{a_{\rm r}'(r)}{a_{\rm r}(r)} = 0, \quad \lim_{x \to \infty, \,x \in S_\delta }\frac{a_{\rm p}^2(x)+q(x)}{a_{\rm r}(|x|)} = 0.
		\end{aligned}
	\end{equation}
	Then
	\begin{equation}\label{sp.ess}
		(-\infty,0] \subset \sigma_{e2} (\cA).
	\end{equation}
\end{enumerate}
\end{theorem}

As an illustration of the generic spectral (and pseudospectral) properties for unbounded damping which is not controlled by $-\Delta+q$, we recall an example with $\Omega =\R$, $a(x)=2x^2$ and $q=0$ from~\cite{Freitas-2018-264,Arifoski-2020-52,arnal2026resolvent-dwe}. Here
\begin{equation}\label{sp.x^2}
	\sigma(\cA) = (-\infty,0] \, \, \dot{\cup} \, \left\{2^\frac 13 \e^{\pm \ii \frac 23 \pi} (2k+1)^\frac 23\right\}_{k \in \N_0}
\end{equation}
is illustrated in Figure~\ref{fig:x2}; further examples with explicit eigenvalues can be found in \cite[Sec.~6]{Freitas-2018-264} and other sufficient conditions on $a$ and $q$ so that \eqref{sp.ess} holds in~\cite[Sec.~4]{Freitas-2018-264}.
\begin{figure}[htb!]
	\includegraphics[width= 0.6 \textwidth]{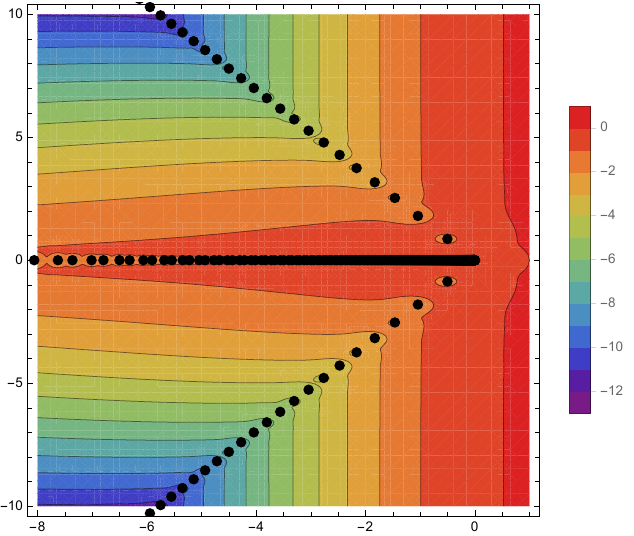}
	\caption{Figure reproduced from \cite{Arifoski-2020-52}. Numerical computation of spectrum (in black, see \eqref{sp.x^2}) and pseudospectra ($\log_{10}$ scale, approximation by  $800\times800$ matrix) of $\Ac$ (see Section \ref{sec:dwe}) with $a(x)=2x^2$ and $q(x)\equiv0$, $x\in \R$.}
	\label{fig:x2}
\end{figure}

Note that when the damping is unbounded at infinity, the resolvent behavior of the damped wave operator in the left complex half-plane is quite non-trivial (see Figure~\ref{fig:x2} for illustration), and qualitatively resembles Schr\"odinger operators with unbounded complex potentials (see e.g.~\cite{Davies-1999-200}).
Nevertheless, unlike for Schr\"odinger operators, where the resolvent behavior at $\pm \ii \infty$ depends on the growth rate of the potential at infinity (see e.g.~\cite{Boulton-2002-47,Pravda-Starov-2006-73,Dencker-2004-57,Helffer-2013-book,arnal2023resolvent}), the resolvent norm of the damped wave operator is bounded (and in general non-decaying) at $\pm \ii \infty$, irrespective of the growth rate of $a$ (see \cite{arnal2026resolvent-dwe} and Section~\ref{ssec.GCC}).

\subsection{The resolvent of \texorpdfstring{$\cA$}{A}}
\label{ssec:A.res}
We recall some steps in~\cite[Sec.~4]{gerhat2024schur} and in particular a representation of the resolvent of $\cA$ in terms of its Schur complement.

For $\lambda \in \C \setminus (-\infty,0]$ the form $\frt_\la$ defines a bounded distribution-valued operator
\begin{equation}
	\label{T.hat.def}
	\widehat T_\la \in \cB (\cD_\frt, \cD_\frt^*), \qquad \widehat T_\la u := \frt_\la[u,\cdot] \in \cD_\frt^*, \qquad u \in \cD_\frt.
\end{equation}
The operator $T_\la$ in \eqref{Tla.def} is in fact the restriction of $\widehat T_\la$ to the maximal domain in $\LOm$, namely, $T_\la := {\widehat T_\la}\vert_{\Dom(T_\la)}$. Due to the coercivity of $\frt_\la$ (without any shift or rotation for $\Re \la >0$, see~\cite[Lem.~4.13]{gerhat2024schur}), we have
\begin{equation}\label{T.hat.la.coer}
\wh T_{\la}^{-1} \in \cB(\cD_{\frt}^*,\cD_{\frt}), \qquad \Re \la >0.	
\end{equation}
Moreover, it is helpful to observe that the following bounded extension property holds (for further generalizations see \cite{gerhat2024schur}).

\begin{lemma}\label{lem:T.la.ext}
	Let $\la \in \C \setminus (-\infty,0]$ and let $T_\la$ and $\wh T_\la$ be as in \eqref{Tla.def} and \eqref{T.hat.def}, respectively. If $T_{\la}^{-1} \in \cB(L^2(\Omega))$, then $\wh T_{\la}^{-1} \in \cB(\cD_{\frt}^*,\cD_{\frt})$. 
\end{lemma}
\begin{proof}
	Indeed, it is proven in~\cite[Lem.~4.13]{gerhat2024schur} that there exists a shift $\mu_\la \in \C$ such that $(\widehat T_{\la}-\mu_\la)^{-1} \in\cB(\cD_\frt^*,\cD_\frt)$. Then $(T_{\la}- \mu_\la)^{-1} \in \cB(L^2(\Omega))$ and, using the first resolvent identity, we infer that
	\begin{equation}
		\begin{aligned}
			T_{\la}^{-1} & = (T_{\la}- \mu_\la)^{-1} - \mu_\la T_{\la}^{-1} (T_{\la}- \mu_\la)^{-1} \\
			& \subset (\widehat T_{\la}- \mu_\la)^{-1} - \mu_\la T_{\la}^{-1} (\widehat T_{\la}- \mu_\la)^{-1} =: R_{\la}  \in \cB(\cD_\frt^*, \cD_\frt),
		\end{aligned}
	\end{equation}
	where for the boundedness we have used that $\cD_\frt \subset L^2(\Omega)$ is boundedly embedded and  $T_{\la}^{-1} \in \cB (L^2(\Omega) , \cD_\frt)$ (see~\cite[Lem.~2.12]{gerhat2024schur}). We then have
	\begin{equation}
		R_{\la} \widehat T_{\la} \vert_{\Dom (T_{\la})} = I_{\Dom (T_{\la})}, \qquad \widehat T_{\la} R_{\la}\vert_{L^2(\Omega)} = I_{\LtOm}.
	\end{equation}
	By continuity and density (see~\cite[Lem.~4.13]{gerhat2024schur} for the density of $\Dom (T_{\la})$ in $\cD_\frt$), these identities extend from $\Dom (T_{\la})$ and $\LtOm$ to $\cD_\frt$ and $\cD_\frt^*$, respectively. This implies $\widehat T_{\la}^{-1} = R_{\la}$. 
\end{proof}

The operator $\cA$ in $\HH$ is implemented as a restriction of a bounded distribution-valued operator as well. To this end, in the second component of the product space $\HH$, we consider the triple (see \eqref{Dt.def} and \eqref{Dt.norm})
\begin{equation}
	\cD_\frt \subset L^2(\Omega)~\widehat=~L^2(\Omega)^* \subset \cD_\frt^*
\end{equation}
and the distribution-valued operator matrix is introduced as
\begin{equation}\label{A.hat.def}
	{\widehat \cA} := \begin{pmatrix}
		0 & I \\
		\Delta - q & -a
	\end{pmatrix} \in \cB (\cW \oplus \cD_\frt, \cW \oplus \cD_\frt^*).
\end{equation}
While the entries in the first row are clearly well-defined and bounded between the respective spaces (note that 
\begin{equation}\label{Dt.W.incl}
\Dc_\frt =  H_0^1(\Omega) \cap \Dom (q^\frac12) \cap \Dom(a^{\frac 12}) \subset \cW	
\end{equation}
is boundedly embedded, see~\cite[Prop.~4.6]{gerhat2024schur}), the ones in the second row are defined weakly. More precisely, the operators
\begin{equation}
	\Delta - q  \in \cB (\cW , \cW^*), \qquad a  \in \cB (\cD_\frt, \cD_\frt^*),
\end{equation}
are defined by
\begin{equation}\label{CD.def}
	\begin{aligned}
		((\Delta-q)u,v)_{\cW^*\times \cW} & := -\int_\Omega \nabla u (x) \overline{ \nabla v (x)} \dd x - \int_\Omega q(x) u(x) \overline{v(x)} \dd x, \\
		(au,v)_{\cD_\frt^*\times \cD_\frt} & := \int_\Omega a (x)u (x)\overline {v(x)} \dd x,
	\end{aligned}
\end{equation}
for all $u$ and $v$ in $\cW$ or $\Dc_\frt$, respectively (see Appendix~\ref{app:W0} for details on $\cW$).
In fact, the above defined operator $-(\Delta-q)$ is nothing but the Riesz isomorphism $J_{\cW} : \cW \to \cW^*$. Hence, in particular, for $u \in \cW$ and $u^* \in \cW^*$,
\begin{equation}\label{Dq.Riesz}
	\|(\Delta-q)u \|_{\cW^*} = \|u\|_{\cW}, \qquad \|(\Delta-q)^{-1} u^* \|_{\cW} = \|u^*\|_{\cW^*}.
\end{equation}

The operator $\cA$ is defined as the maximal restriction of $\wh \cA$ to $\HH$, in detail,
\begin{equation}\label{A.T.op}
\cA := {\widehat \cA}\vert_{\Dom(\cA)},
\end{equation}
with $\Dom(\cA)$ introduced in \eqref{A.dom}. Note that the first row in \eqref{A.hat.def} does not contribute any restriction to the domain of $\cA$ since $\Dc_\frt \subset \cW$ and we employed a triple $\cD_\frt \subset L^2(\Omega) \subset \cD_\frt^*$ in the second space component only. 

The adjoint of $\cA$ has the expected structure. 
\begin{proposition}\label{prop:A*}
Let $\cA$ be as in \eqref{A.action} and \eqref{A.dom}. Then the adjoint of $\cA$ reads
\begin{equation}\label{A*.def}
\begin{aligned}
\Ac^* &= \begin{pmatrix}
	0 & -I \\
	-(\Delta - q)  & -a
\end{pmatrix},
\\
\Dom(\cA^*) & = \left\{ F =(f, g) \in \Wc \times \cD_t \, : \, (\Delta-q) f  +ag \in L^2(\Omega) \right\}.
\end{aligned}
\end{equation}
\end{proposition}
\begin{proof}
We prove that the adjoint of $\Ac$ can be obtained as the restriction of the (adjoint) distributional matrix
\begin{equation}\label{C.hat.def}
	{\wh \cC}:= \begin{pmatrix}
		0 & -I \\
		-(\Delta - q)  & -a
	\end{pmatrix} \in \cB (\Wc \oplus \cD_t, \Wc \oplus \cD_t^*).
\end{equation}
We denote the restriction of $\wh \cC$ to the domain in~\eqref{A*.def} by $\cC$ and we prove that  $\cC = \Ac^*$. Analogously as for $\Ac$, it follows that $\cC$ is m-dissipative and since also $\Ac^*$
is m-dissipative, it is sufficient to show only the inclusion $\cC \subset \Ac^*$. To this end, fix $(f,g) \in \Dom (\cC)$ and let $(u,v)\in\Dom(\Ac)$ be arbitrary. With \eqref{CD.def} we compute
\begin{equation} \label{C.adj}
\begin{aligned}
\langle \cC (f,g), (u,v) \rangle_{\HH}
& = - \langle g, u \rangle_{\cW} - \langle (\Delta - q) f + a g,v\rangle_{L^2} \\
& = - \langle g, u \rangle_{\cW} - ( (\Delta - q) f, v)_{\Wc^* \times \Wc} - (ag,v)_{\cD_t^* \times \cD_t}\\
& = \overline{((\Delta-q) u,g)}_{\Wc^*\times \Wc}  + \overline{\langle v,f \rangle}_{\cW} -  \overline{(av,g)}_{\cD_t^* \times \cD_t}  \\
& = \overline{\langle (\Delta-q) u -av,g\rangle}_{L^2}  + \overline{\langle v,f \rangle}_{\cW}   \\
&  = \langle (f,g), \Ac (u,v) \rangle_{\HH}.
\end{aligned}
\end{equation}
It follows that $(f, g) \in \Dom (\Ac^*)$ and $\Ac^* (f, g) = \cC (f,g)$, i.e.~that $\cC \subset \Ac^*$.
\end{proof}

Finally, the resolvent of $\cA$ is constructed in terms of $\wh T_\la^{-1}$ as
\begin{equation}
	\label{A.res}
		(\cA-\la)^{-1} =
		- \begin{pmatrix}
			\la^{-1} \left(I +  \widehat T_{\la}^{-1} (\Delta-q) \right) &    \widehat T_{\la}^{-1}  \\[2mm]
			\widehat T_{\la}^{-1} (\Delta-q) &  \la    \widehat T_{\la}^{-1}
		\end{pmatrix},
\end{equation}
see~\cite[proof of Thm.~2.8]{gerhat2024schur}. 

\subsection{Range of \texorpdfstring{$\cA$}{A} and the space \texorpdfstring{$\KK$}{K}}
\label{ssec:Ran.A.K}

We can characterize the range of $\cA$ as
\begin{align}
%\nonumber
\lefteqn{\Ran(\cA)}\\
%\nonumber
& \quad = \big\{(f,g) \in \cW \times L^2(\Omega) \, : \, \exists~ (u,v) \in \Dom(\cA),\,
f = v,  \, g = (\Delta-q)u -av \big\}
\\
%\nonumber
& \quad = \left\{(f,g) \in \cD_\frt \times L^2(\Omega) \, : \, \exists~ u \in \cW, \, g = (\Delta-q)u -af \right\}
\\
\label{A.Ran}
& \quad = \left\{(f,g) \in \cD_\frt \times L^2(\Omega) \, : \, af + g \in \cW^* \right\},
\end{align}
where the bijectivity of $\Delta-q = - J_{\cW} : \cW \to \cW^*$ was used in the last step. It turns out that the range of $\cA$ is precisely the space $\KK$ introduced in \eqref{KK.def} and which appears in our main results.

\begin{proposition} \label{prop:Ran.A}
Let $\cA$ be as in \eqref{A.action} and \eqref{A.dom}. Then the following claims hold.
\begin{enumerate}[\upshape (i)]
	\item \label{item:A.inj} Both $\cA$ and $\cA^*$ are injective.
	\item \label{item:ran.A} $\Ran(\cA) = \KK$ and we have
	\begin{equation}\label{K.norm.equiv}
		\nr{F}_{\HH}^2 + \nr{\cA\inv F}_{\HH}^2 = \nr{F}_{\KK}^2, \qquad F \in \Ran(\cA).
	\end{equation}
	Moreover, the norm 
	\begin{equation}
		F = (f,g) \mapsto (\|f\|_{\cD_{\frt}}^2 + \|g\|_{L^2}^2 + \|af + g\|_{\cW^*}^2)^\frac12
	\end{equation}
	is equivalent to $\nr{\cdot}_{\KK}$. In particular,
	\begin{equation}\label{K.in.Dt+L2}
	\KK \subset \cD_{\frt} \oplus L^2(\Omega).
	\end{equation}
	
	\item $\KK$ is dense in $\HH$.
\end{enumerate}
\end{proposition}

\begin{proof}
\begin{enumerate}[\upshape (i), wide]
	\item The claims follow from the injectivity of $\Delta-q$ on $\cW$ (see~\eqref{Dq.Riesz} and Proposition~\ref{prop:A*}).
	\item From \eqref{A.Ran}, we  have $\Ran(\cA) \subset \KK$. Since for $F = (f,g) \in \Ran(\Ac)$
	\begin{equation}\label{A.inv}
		\Ac \inv F = \big((\Delta-q)\inv (g +af), f \big),
	\end{equation}
	we obtain by \eqref{Dq.Riesz} that
	\begin{equation}
		\nr{F}_{\HH}^2 + \nr{\Ac\inv F}_{\HH}^2
		= \nr{f}_{\Wc}^2 + \nr{g}_{L^2}^2 + \nr{g +af}_{\Wc^*}^2 + \nr{f}_{L^2}^2 = \nr{F}_{\KK}^2.
	\end{equation}
	Moreover, using $\widehat T_1\inv \in \Bc(\Dc_\frt^*,\Dc_\frt)$ (see \eqref{T.hat.la.coer}), and the continuity of the embeddings $\Wc^* \subset \Dc_\frt^*$ and $L^2(\Omega) \subset \Dc_\frt^*$, we arrive at
	\begin{equation}
		\| a^{\frac 12} f \|_{L^2} \lesssim \|\widehat T_1 f\|_{\Dc_\frt^*} \lesssim \| (\Delta - q) f \|_{\Wc^*} + \|af + g\|_{\Wc^*} + \nr{f}_{L^2} + \nr{g}_{L^2} \lesssim \nr{F}_{\KK}.
	\end{equation}
	This shows the equivalence of the norms for $F \in \Ran(\cA)$.
	
	It remains to justify that $\KK \subset \Ran(\cA)$. Let $(f,g) \in \KK$. We have $f \in H^1_0(\Omega) \cap \Dom(q^\frac 12)$, $g \in \LOm$, $af \in \LolocOm$ and $af+g \in \cW^{*} \subset \cD_{\frt}^*$, so it is enough to show that $a^\frac12 f \in \LOm$. To this end, we write
	\begin{equation}
		\cD'(\Omega) \ni (-\Delta + q + a +1) f = (-\Delta + q)f + (a f +g) + f -g \in \cD_{\frt}^*.
	\end{equation}
	Since $\wh T_1$ is bijective between $\cD_\frt$ and $\cD_{\frt}^*$, there exists $h \in \cD_{\frt}$ such that the identity 
	\begin{equation}
		(-\Delta + q + a +1) f = \wh T_1 h = (-\Delta + q + a +1 ) h
	\end{equation}
	holds in $\cD'(\Omega)$. We show below that $f=h \in \cD_\frt$ and hence $a^\frac12 f \in \LOm$ as claimed. 
	
	Let $u := f-h$, then $u \in H^1_0(\Omega) \cap \Dom(q^\frac 12)$, $a u \in \LolocOm$ and
	\begin{equation}
		\Delta u = (q+a +1) u \in \LolocOm.
	\end{equation}
	Employing Kato's inequality (see e.g.~\cite[Sec.~VII.2.1]{EE} for details), we obtain
	\begin{equation}
		\Delta |u| \geq \Re \left(
		\frac{\ov u}{|u|} \Delta u
		\right)= \Re \left(
		(q+a +1) |u|
		\right) \geq |u| 
	\end{equation}
	in $\cD'(\Omega)$, i.e.~for all $0\leq \varphi \in \CcOm$
	\begin{equation}
		((1-\Delta)|u|, \varphi)_{\cD' \times \cD} \leq 0.
	\end{equation}
	Since $u \in H^1_0(\Omega) \cap \Dom(q^\frac 12) \subset H_0^1(\Omega)$, also $|u| \in H_0^1(\Omega)$. Moreover, there exists $\{\varphi_k\} \subset \CcOm$ with $\varphi_k \geq 0$, $k \in \N$, such that $\varphi_k \to |u|$ in $H^1_0(\Omega)$ (see e.g.~\cite[Cor.~VI.2.4, Thm.~VI.3.6]{EE}). Hence
	\begin{equation}
		\||u|\|_{H^1}^2 = \lim_{k \to \infty} \langle |u|, \varphi_k \rangle_{H^1} = \lim_{k \to \infty} ((1-\Delta)|u|, \varphi_k)_{\cD' \times \cD} \leq 0,
	\end{equation}
	and we conclude that $|u| = |f-h| = 0$ a.e.~in $\Omega$.
\item By \ref{item:ran.A} and \ref{item:A.inj}, we have $\overline{\Kc} = \overline{\Ran(\Ac)} = \Ker(\Ac^*)^\perp = \Hc$.
\qedhere
\end{enumerate}

\end{proof}

\begin{corollary}\label{cor:0.sp}
Let $\cA$ be as in \eqref{A.action} and \eqref{A.dom}. Then $0 \in \rho(\cA)$ if and only if there exists $M>0$ such that
\begin{equation}\label{a.q.rel.bdd}
\|a^\frac 12 f\|_{L^2}^2 + \|f\|_{L^2}^2 \leq M \big(\|\nabla f\|_{L^2}^2 + \|q^\frac 12 f\|_{L^2}^2 \big), \qquad f \in \CcOm;
\end{equation}
in this case $(-1/M,0] \subset \rho(\cA)$. Moreover, if $0 \in \sigma(\cA)$, then $0 \in \sigma_{\rm c} (\cA) \subset \sigma_{\rm e1} (\cA)$.
\end{corollary}
\begin{proof}
Assume that $0 \in \rho(\cA)$. It follows that $\Ran(\cA) = \KK = \cH$, thus the inclusions \eqref{Dt.W.incl} and \eqref{K.in.Dt+L2} yield that $\cW = \cD_{\frt}$. Since the identity map $I_{\cD_{\frt} \to \cW}: \cD_{\frt} \to \cW : f \mapsto f$ is everywhere defined and bounded, the bounded inverse (or closed graph) theorem  shows that $I_{\cD_{\frt} \to \cW}^{-1} =   I_{\cW \to \cD_{\frt}}$ is bounded. Hence there exists $M>0$ such that for all $f \in \CcOm \subset \cD_{\frt}$
\begin{equation}
\|\nabla f\|_{L^2}^2 + \|q^\frac 12 f\|_{L^2}^2 + \|a^\frac 12 f\|_{L^2}^2 + \|f\|_{L^2}^2 \leq M \left( \|\nabla f\|_{L^2}^2 + \|q^\frac 12 f\|_{L^2}^2  \right),
\end{equation}
and \eqref{a.q.rel.bdd} follows.

To show the reverse implication, suppose that \eqref{a.q.rel.bdd} holds. It follows that for all $\la \in (-1/M,0)$
\begin{equation}
\frt_{\la}[f] \geq \big(1-|\la| M \big)\big(\|\nabla f\|_{L^2}^2 + \|q^\frac 12 f\|_{L^2}^2\big) + |\la|^2\|f\|_{L^2}^2 \gs \|f\|_{\cD_{\frt}}^2, \quad f \in \CcOm.
\end{equation}
Since $\CcOm$ is dense in $\cD_{\frt}$, the forms $\frt_{\la}$, $\la \in (-1/M,0)$, are coercive on $\cD_{\frt}$ and so we have $\wh T_\la^{-1} \in \cB(\cD_{\frt}^*,\cD_{\frt})$. It follows that $\la \in \rho(\cA)$ (see~\cite[Cor.~3.7, Sec.~4]{gerhat2024schur} for details). For $\la =0$, notice that \eqref{a.q.rel.bdd} implies that (with equivalent norms)
\begin{equation}
\cW = H^1_0(\Omega) \cap \Dom(q^\frac12) = \cD_\frt,	% = \Dom((-\Delta+q)_{\rm D}^\frac12);
\end{equation}
(see \eqref{W.form.dom}), and thus also $\cW^* = \cD_{\frt}^*$. It is then straightforward to see from~\eqref{A.Ran} that $\Ran(\cA) = \cH$, hence $\cA$ is bijective and $0 \in \rho(\cA)$ by the closed graph theorem.

Finally, since $\Ran (\cA)$ is dense in $\cH$ and $\cA$ is injective by Proposition~\ref{prop:Ran.A}, it follows that if $0 \in \sigma(\cA)$, then $0 \in \sigma_{\rm c}(\cA)$. It is immediate from the definitions that $\sigma_{\rm c} (\cA) \subset \sigma_{\mathrm e1} (\cA)$. 
\end{proof}

\subsection{Restriction of \texorpdfstring{$\e^{t\cA}$}{the} semigroup to \texorpdfstring{$\KK$}{K}}

For some estimates of Theorem~\ref{thm:decay}, we will use the restriction of the semigroup $\e^{t\cA}$ to the space $\KK$. Define the operator $\AK$ in $\KK$ by
\begin{equation}
	\begin{aligned} \label{AK.def}
		\AK & \subset  \Ac, \\
		\Dom(\AK) &= \{ U \in \Dom(\Ac) \cap \KK \, : \, \Ac U \in \KK\}
		= 	\Dom(\Ac) \cap \Ran(\cA).
	\end{aligned}
\end{equation}

\begin{proposition}
Let $\cA$ be as in \eqref{A.action} and \eqref{A.dom} and $\AK$ as in \eqref{AK.def}.
\begin{enumerate}[\upshape (i)]
\item\label{prop.Ak.i} We have $\rho(\Ac) \subset \rho(\AK)$ and
\begin{equation}\label{res.H.to.K}
	\nr{(\AK-\la)\inv}_{\cB(\KK)} \leq 	\nr{(\Ac-\la)\inv}_{\cB(\HH)}, \qquad \la \in \rho(\cA).
\end{equation}
\item For $t \geq 0$ we have $\e^{t\Ac} \KK \subset \KK$ and the restriction $(\e^{t\Ac}|_\KK)_{t \geq 0}$ defines a strongly continuous contraction semigroup $(\e^{t \cA_\KK})_{t\ge 0}$ on $\KK$ with the densely defined generator $\AK$.
\end{enumerate}
\label{prop:AK}
\end{proposition}

\begin{proof}
\begin{enumerate}[\upshape (i), wide]
\item 
Let $\la \in \rho(\Ac)$. If $\la = 0$ then $\HH = \KK$ and the statement is clear. Now assume that $\la \neq 0$. Since $(\Ac-\la)$ is injective, then so is $(\Ac_\KK-\la)$. Let $F \in \KK \subset \HH$. Since $(\Ac-\la)$ is surjective, there exists $U \in \Dom(\cA)$ such that $(\cA-\la) U = F$. It follows that $U= \la^{-1} (\cA U - F) \in \Ran(\cA)$. Thus $U \in \Dom(\AK)$ and $(\AK-\la) U = F$. This proves that $(\Ac_\KK-\la)$ is surjective. Finally, for the boundedness of the inverse we estimate
	\begin{equation}
		\begin{aligned}
			\nr{(\AK-\la)\inv F}_\KK^2
			& = \nr{\Ac\inv (\Ac-\la)\inv F}_\HH^2 + \nr{(\Ac-\la)\inv F}_\HH^2 \\
			& = \nr{ (\Ac-\la)\inv \Ac\inv F}_\HH^2 + \nr{(\Ac-\la)\inv F}_\HH^2 \\
			& \leq  \nr{ (\Ac-\la)\inv}_{\cB(\HH)}^2
			\left(\nr{\Ac\inv F}_\HH^2 + \nr{F}_\HH^2 \right)
			\\
			&
			= \nr{ (\Ac-\la)\inv}_{\cB(\HH)}^2 \nr{F}_\KK^2,
		\end{aligned}
	\end{equation}
	where \eqref{K.norm.equiv} was used in the first and last steps.

\item
First, we prove that $(\e^{t\Ac}|_{\KK})_{t\geq 0}$ defines a contraction semigroup on $\KK$.
	To this end, let $F \in \Ran(\cA) = \KK$ and let $U \in \Dom(\Ac)$ be such that $F = \Ac U$. By \cite[Lem.~II.1.3~(ii)]{Engel-Nagel-book}, we have
	\[
	\e^{t\Ac} F = \e^{t\Ac} \Ac U = \Ac \e^{t\Ac} U \in \Kc,
	\]
	so $(\e^{t\Ac}|_{\KK})_{t\ge0}$ defines a semigroup on $\KK$. Similarly,
	\begin{equation}
		\Ac\inv \e^{t\Ac} F = \Ac\inv \e^{t\Ac} \Ac U = \e^{t\Ac} U = \e^{t\Ac}\Ac\inv F.
	\end{equation}
	Using \eqref{K.norm.equiv} and that $(\e^{t\cA})_{t\ge0}$ is contractive on $\HH$, we arrive at
	\[
	\begin{aligned}
		\nr{\e^{t\Ac} F}_\KK^2
		&= \nr{\Ac\inv \e^{t\Ac} F}_\HH^2 + \nr{ \e^{t\Ac} F}_\HH^2
		= \nr{\e^{t\Ac} \Ac\inv F}_\HH^2 + \nr{ \e^{t\Ac} F}_\HH^2
		\\ &\leq \nr{\Ac\inv F}_\HH^2 + \nr{F}_\HH^2 = \nr{F}_\KK^2.
	\end{aligned}
	\]
	We similarly check that $\nr{\e^{t\Ac}F-F}_\KK \to 0$ as $t\to 0$, and it follows that $(\e^{t\Ac}|_{\KK})_{t\geq 0}$ is a strongly continuous contraction semigroup on $\KK$. We denote its generator by $\Gc$ and we show that $\Gc= \AK$.

	Notice first that by the Hille--Yosida theorem \cite[Thm. II.3.5]{Engel-Nagel-book}, $\Gc$ is a densely defined maximal dissipative operator in $\KK$ (the dissipativity follows from the resolvent bound therein, see \cite[Sec.~V.3.10]{Kato-1966}). Since $\cA_\KK$ is also m-dissipative by \ref{prop.Ak.i}, it is sufficient to prove one inclusion $\AK \subset \Gc$. For $U \in \Dom(\AK)$ and $F =\cA U$, we have by~\cite[Lem.~II.1.3~(iv)]{Engel-Nagel-book}
	\[
	\e^{t\Gc}U -U = \e^{t\Ac} U - U = \cA \int_0^t \e^{s\Ac} U \diff s = \int_0^t \e^{s\Ac} F \diff s, \qquad t \geq 0.
	\]
	It follows that $U \in \Dom(\Gc)$ and $\Gc U = F = \Ac U = \AK U$, and hence $\AK \subset \Gc$.
	\qedhere
\end{enumerate}
\end{proof}

\section{Resolvent estimates for low frequencies}
\label{sec:low}

\subsection{Statements for the low frequency resolvent estimates}

The resolvent estimates for $\la\to 0$ are obtained under the following assumption (weaker than the assumption used in Theorem \ref{thm:decay}).

\begin{asm-sec}
	\label{asm:a.1}
	Let $\Omega \subset \R^d$ be open and non-empty and let $0 \leq a, q \in \LolocOm$. Suppose that there are open non-empty sets $\Omega_1, \Omega_2 \subset \Omega$ such that
	\begin{enumerate}[\upshape (i)]
		\item \label{itm:Om1.bded} $\Omega_1$ is a bounded region with Lipschitz continuous boundary (see e.g.~\cite[Def.~D.2.3.1,~p.~815]{bhattacharyya2012distributions}),
		\item \label{itm:Om.dun} $\Omega_1 \cap \Omega_2 = \emptyset$, $\left(\overline{\Omega_1 \cup \Omega_2}\right)^{\circ} = \Omega$,  $|\Omega \setminus (\Omega_1 \cup \Omega_2)| = 0$,
		\item \label{itm:Om12.a} $\|\1_{\Omega_1} a \|_{L^1} > 0$ and there exists $a_0 > 0$ such that $a(x) \ge a_0$ a.e.~in $\Omega_2$.
	\end{enumerate}
\end{asm-sec}

We remark that Assumption \ref{asm:a.1} is satisfied in the case when there exists a (non-empty) ball $B(x_0,r_0) \subset \Omega$ with $\ov{B(x_0,r_0)} \subset \Omega$ such that $\|\1_{B(x_0,r_0)} a\|_{L^1} >0$ and $a(x) \geq a_0 >0$ for a.e.~$x \in \Omega \setminus B(x_0,r_0)$.

For an exterior domain $\Omega$, we shall also explicitly consider damping coefficients $a$ which are unbounded at infinity.

\begin{asm-sec}
	\label{asm:a.2}
	Let $\Omega \subset \Rd$ be an exterior domain in $\Rd$ with a Lipschitz boundary, let $0 \leq a, q \in \LolocOm$ and assume that there exist $\beta > 0$ and $r_0 > 0$ such that $\Rd \setminus B(0,r_0) \subset \Omega$ and
	\begin{equation}\label{a.unbd.asm}
		a(x) \gs |x|^\beta, \qquad \text{a.e.} \,\,\, |x|>r_0.
	\end{equation}
\end{asm-sec}

Notice that Assumption~\ref{asm:a.2} is stronger than Assumption~\ref{asm:a.1}, since under Assumption~\ref{asm:a.2} we get Assumption~\ref{asm:a.1} by considering (for any $r>r_0>0$)
\begin{equation}\label{omega.a.unbd.def}
	\Omega_1 = B(0, r) \cap \Omega, \qquad \Omega_2 = \Omega \setminus \overline{B(0, r)}.
\end{equation}

We set
\begin{equation}\label{gamma.def}
	\gamma=	\begin{cases}
		1,& \text{under Assumption \ref{asm:a.1}},
		\\[1mm]
		\frac{2}{2+\beta}, & \text{if Assumption \ref{asm:a.2} holds in addition}.
	\end{cases}
\end{equation}

Our main spectral result is the analysis of the resolvent of $\cA$ near 0. 
To formulate it we define the bounded operators (see~\eqref{K.in.Dt+L2})
\begin{equation} \label{def:PI2}
\Pi_1: \KK \to \cD_{\frt}:  (f,g) \mapsto f, \qquad
\Pi_2: \KK \to L^2(\Omega):  (f,g) \mapsto g.
\end{equation}

\begin{theorem}
\label{thm:low}
Let $\cA$ and $\cA_\KK$ be as in \eqref{A.action}, \eqref{A.dom} and \eqref{AK.def}, respectively, and let Assumption~\ref{asm:a.1} (or the stronger Assumption \ref{asm:a.2}) hold. Let $\gamma$ be as in \eqref{gamma.def}. Then there exist $\tau_0>0$ and $C > 0$ such that for $\la \in \ov{\tau_0\D_+} \setminus \{0\}$, we have $\lambda \in \rho(\Ac) \cap \rho(\AK)$ and
\begin{align}
\|(\Ac - \la)^{-1} \|_{\cB(\HH)} + \|(\Ac_\Kc - \la)^{-1} \|_{\cB(\KK)}   & \leq C   |\la|^{-1},
\label{r.A.small.b.H-H}
\\
\|(\Ac_\KK - \la)^{-1} \|_{\cB(\KK,\HH)}  & \leq C ,
\label{r.A.small.b.K-H}
\\
\|\Pi_1(\Ac_\KK - \la)^{-1} \|_{\cB(\KK,L^2)}  & \leq C |\la|^{- \frac {\gamma}2 },
\label{r.A.small.b.K-L}\\
\|\Pi_1 (\Ac_\KK - \la)^{-1} \|_{\cB(\KK,\Dc_\frt)}  & \leq C |\la|^{-\frac 12},
\label{r.A.small.b.K-D}\\
\| \Pi_2(\Ac_\KK -\la)^{-1}\|_{\cB(\KK,L^2)} & \leq C,
\label{r.A.small-dt1}\\
\| \Pi_2(\Ac_\KK -\la)^{-2}\|_{\cB(\KK,L^2)} & \leq C |\la|^{-\frac \gamma 2}.
\label{r.A.small-dt2}
\end{align}
(In \eqref{r.A.small.b.K-H} it is understood that $(\Ac_\KK-\la)\inv$ is composed on the left with the embedding of $\KK$ into $\HH$, and in \eqref{r.A.small.b.K-L} that $\cD_{\frt}$ is embedded in $L^2(\Omega)$.)
\end{theorem}

In the case of our main interest, i.e.~a strong damping for which \eqref{a.q.rel.bdd.intro} is not satisfied, zero belongs to the spectrum of $\cA$ (see Corollary~\ref{cor:0.sp}). Thus the rate in \eqref{r.A.small.b.H-H} above cannot be improved and we cannot have uniform decay for $\e^{t\cA}$ (see Proposition \ref{prop:exp-decay}). However, we still have some (polynomial) decay in some suitable topology and the rates in Theorem \ref{thm:decay} reflect the nature of the singularity of the resolvent $(\cA - \la)\inv$ around zero. Details are given in Section \ref{sec:decay}.

Theorem \ref{thm:low} is proved below. The main step is the following estimate on the form $\frt_\la$ for $\la \in \overline{\D_+}$ near zero.

\begin{proposition}
\label{prop:tb.lbound}
Let Assumption~\ref{asm:a.1} (or the stronger Assumption \ref{asm:a.2}) hold and let the form $\frt_\la$ be as in \eqref{tla.def}. Then as $\la \to 0$ in $\ov{\D_+} \setminus \{0\}$
\begin{equation}
\label{tb.lbound}
|\frt_\la[u]| \gs \| \nabla u \|_{L^2}^2 + \| q^{\frac12} u \|_{L^2}^2 + |\la| \| a^{\frac12} u \|_{L^2}^2 + |\la|^\gamma \| u \|_{L^2}^2, \qquad u \in \cD_\frt,
\end{equation}
where $\gamma$ is as in \eqref{gamma.def}.
\end{proposition}

In Section \ref{ssec:res.A.zero} we deduce Theorem \ref{thm:low} from Proposition \ref{prop:tb.lbound}. Then we provide a proof of Proposition \ref{prop:tb.lbound} in Section \ref{ssec:t.nearzero}.

\begin{remark}
\begin{enumerate}[\upshape (i), wide ]
\item
If $\Omega$ and $q$ are such that \eqref{Om.q.Poincare} holds, one can easily improve \eqref{tb.lbound} to $\gamma = 0$, formally corresponding to $\beta = \infty$ (see \eqref{t.la.rot}). This will also improve \eqref{dtu.L2.est.gamma} and \eqref{u.L2.est.gamma} accordingly.
\item For $q=0$ and $a(x) = |x|^\beta$, $x \in \Rd$, a scaling argument shows that the rate $\gamma={2/(2+\beta)}$ in \eqref{tb.lbound} is in fact optimal.
\item The conditions in Assumption~\ref{asm:a.1} can be further relaxed. Consider the example $a(x,y)=x^2y^2$, $(x,y) \in \R^2$, for which the self-adjoint operator $-\Delta + x^2 y^2$ in $L^2(\R^2)$ has compact resolvent and the lowest eigenvalue is positive (see e.g.~\cite{Simon-1983-146}). By a scaling argument applied on the right-hand side of \eqref{t.la.rot}, one obtains that \eqref{tb.lbound} holds with $\gamma = 1/3$ (like for $a(x,y) = |(x,y)|^4$). Note that $a(0,y) = a(x,0) = 0$, so Assumption \ref{asm:a.1} is not satisfied in this example.
\end{enumerate}
\label{rem:super-Poincare}
\end{remark}

\subsection{The resolvent norm of \texorpdfstring{$\Ac$}{A} near zero}
\label{ssec:res.A.zero}
Assume that Proposition \ref{prop:tb.lbound} holds. The inequality~\eqref{tb.lbound} allows us to estimate the norm of the distributional Schur complement $\wh T_\la$ between suitable spaces, which leads to an estimate of the Frobenius--Schur factorization~\eqref{A.res} of the resolvent.
\begin{lemma}\label{lem:T.ineq}
Let $\widehat T_{\la}$ be defined as in~\eqref{T.hat.def} and let~Assumption~\ref{asm:a.1} or the stronger Assumption \ref{asm:a.2} hold. Then there exists $\tau_0>0$ such that $\wh T_\la^ {-1}: \cD_\frt^* \to \cD_\frt$ is bounded for $\lambda \in \overline{\tau_0\D_+} \setminus \set 0$ and, as $\la \to 0$,
\begin{equation}\label{T.la.inv.norm}
\begin{aligned}
\|\wh T_\la^{-1}\|_{\cB(\cW^*,\cW)} &
\ls
1,
&
\quad \|\wh T_\la^{-1}\|_{\cB(\Wc^*,L^2)} &\ls  |\la|^{- \frac \gamma 2 },
\\
\|\wh T_\la^{-1}\|_{\cB(L^2,\cW)} &\ls  |\la|^{- \frac \gamma 2 } ,
&
\|\wh T_\la^{-1}\|_{\cB(L^2,L^2)} &\ls
|\la|^{- \gamma},
\\
\|\wh T_\la^{-1}\|_{\cB(\Wc^*,\Dc_\frt)} & \ls  |\la|^{- \frac 1 2 },
&
\|\wh T_\la^{-1}\|_{\cB(L^2,\Dc_\frt)} & \ls  |\la|^{- \frac {\gamma+1} 2 },
\end{aligned}
\end{equation}
where $\gamma$ is as in \eqref{gamma.def}.
\end{lemma}
\begin{proof}
Note first that by~\eqref{tb.lbound} there exists $\tau_0>0$ such that for all $0 \neq \la \in \overline{\tau_0\D_+}$ we have $|\frt_{\la}[u]|\gtrsim |\la|\|u\|_{\cD_{\frt}}^2$, and thus $\wh T_\la^ {-1}: \cD_\frt^* \to \cD_\frt$ is bounded. Since $\cD_\frt$ is a (dense and boundedly embedded) subspace of $\Wc$ and $L^2(\Omega)$, it follows that $L^2(\Omega)$ and $\Wc^*$ are (dense and boundedly embedded) subspaces of $\cD_\frt^*$. Thus one can indeed view $\wh T_\la^ {-1}$ as an operator between the various spaces as in \eqref{T.la.inv.norm}, where the respective restrictions are not written explicitly as they are clear from the context.

We show the second inequality in the first line of \eqref{T.la.inv.norm}, the remaining ones are justified analogously.
By Proposition~\ref{prop:tb.lbound} we have, as $\la \to 0$ in $\overline{\D_+} \setminus \{0\}$,
\begin{equation}
\label{eq_coer}
|(\widehat T_{\la} u, u)_{\cD_\frt^* \times \cD_\frt}|
\gtrsim |\la|^{\frac \gamma2} \|u\|_{\cW} \|u\|_{L^2}, \qquad u \in \cD_\frt.
\end{equation}
Hence, using \eqref{eq_coer} for $u = \wh T_{\la}^{-1} \varphi$ with $\varphi \in \cW^*$ 
\begin{equation}
\|\wh T_\la^{-1}\|_{\cB(\cW^*, L^2)}  = \sup_{0\neq \varphi \in \cW^*} \frac{\|\widehat T_{\la}^{-1} \varphi \|_{L^2}}{\|\varphi\|_{\cW^*}}
%\\
%&
\ls |\la|^{-\frac \gamma2} \sup_{0\neq \varphi \in \cW^*}
\frac{|(\varphi, \wh T_{\la}^{-1} \varphi)_{\cW^* \times \cW}|}{\|\varphi\|_{\cW^*} \|\wh T_{\la}^{-1} \varphi\|_{\cW}}
\ls |\la|^{-\frac \gamma2};
\end{equation}
notice that the pairing in $\cD_\frt^* \times \cD_\frt$ can be written in the duality $\cW^* \times \cW$ since $\f \in \Wc^*$ and $\wh T_{\la}^{-1} \varphi \in \cD_{\frt} \subset \cW$.
\end{proof}

\begin{proof}[Proof of Theorem~\ref{thm:low}]
We prove first \eqref{r.A.small.b.H-H}, based on the estimates in Lemma \ref{lem:T.ineq} and the representation of the resolvent in \eqref{A.res}.
Notice that $\overline{\tau_0 \D_+} \setminus \{0\}\subset \rho(\cA) \subset \rho (\cA_\cK)$ by the spectral equivalence~\eqref{spec.equiv} and the bounded invertibility of $T_\la$ for $\la \in \overline{\tau_0\D_+} \setminus \{0\}$ in Lemma~\ref{lem:T.ineq} (see also Proposition~\ref{prop:AK}).
For any $F = (f,g)\in \HH$ we have
\begin{equation}
\begin{aligned}
 \|(\Ac-\la)^{-1} F\|_{\HH} 
&   \ls |\la|^{-1} \left(\|f\|_{\cW} + \|\wh T_{\la}^{-1}(\Delta-q)f\|_{\Wc} \right)
+ \|\wh T_{\la}^{-1}g\|_{\Wc}
\\
& \qquad \qquad 
+ \|\wh T_{\la}^{-1}(\Delta-q)f\|_{L^2} + |\la| \|\wh T_{\la}^{-1}g\|_{L^2}
\\
& \ \ls
|\la|^{-1} \nr{f}_{\Wc}  + \abs\la\inv \|\wh T_{\la}^{-1}\|_{\cB(\Wc^*,\Wc)}\|(\Delta-q)\|_{\cB(\Wc,\Wc^*)}  \|f\|_{\Wc}
\\
&  
\qquad \qquad  + \|\wh T_{\la}^{-1}\|_{\cB(\Wc^*,L^2)}\|(\Delta-q)\|_{\cB(\Wc,\Wc^*)}  \|f\|_{\Wc}  \\
& \qquad \qquad 
+\|\wh T_{\la}^{-1}\|_{\cB(L^2,\Wc)} \|g\|_{L^2} 
+ |\la| \|\wh T_{\la}^{-1}\|_{\cB(L^2)} \|g\|_{L^2}.
\end{aligned}
\end{equation}
Employing Lemma \ref{lem:T.ineq}, \eqref{gamma.def} and \eqref{Dq.Riesz}, we obtain
\begin{equation} \label{res.A.H-H}
\|(\Ac-\la)^{-1} F\|_{\HH} \lesssim  |\la|^{-1} \nr{F}_{\HH},
\end{equation}
i.e.~\eqref{r.A.small.b.H-H} holds for $\Ac$. Then the estimate for $\Ac_\KK$ in \eqref{r.A.small.b.H-H} follows from \eqref{res.H.to.K}. 

Notice that if we only consider the second row in \eqref{A.res} we get
\begin{equation} \label{res.A.H-0L2}
\|\Pi_2 (\Ac-\la)^{-1} F\|_{L^2} \lesssim  |\la|^{-\frac \gamma 2} \nr{F}_{\HH}.
\end{equation}

To justify the remaining inequalities, we first rearrange the formula \eqref{A.res} for $(\Ac-\la)^{-1} F$ when $F = (f,g) \in \KK = \Ran (\cA)$ (see~\eqref{A.Ran} and Proposition~\ref{prop:Ran.A}). Since $f \in \cD_\frt$ implies $a f \in \cD_\frt^*$, we have
\begin{equation}
\wh T_\la^{-1}(\Delta-q) f = \wh T_\la^{-1}\left(
(\Delta-q) f - \la a f - \la^2 f +  \la a f + \la^2 f
\right)
= -f + \la \wh T_\la^{-1} ( af + \la f).
\end{equation}
Hence,
\begin{equation}
(\Ac-\la)^{-1} F=
-
\begin{pmatrix}
\wh T_\la^{-1} (a f +g) + \la \wh T_\la^{-1}f
\\[1mm]
\la	\wh T_\la^{-1} (a f +g) + \la^2 \wh T_\la^{-1}f - f
\end{pmatrix} .
\end{equation}
We set $U = (u,v) = (\Ac-\la)^{-1} F$. Since $ af+g \in \cW^*$ and $f \in \LOm$, by \eqref{T.la.inv.norm} we arrive at (with $\iota, \kappa \in \{0,1\}$)
\begin{equation} 
\begin{aligned}
& \|u\|_{\cW} + \iota \|u\|_{L^2} + \kappa \|a^\frac12 u\|_{L^2}
\\
& \quad \ls \left(\|\wh T_\la^{-1}\|_{\cB(\cW^*,\cW)} + \iota \|\wh T_\la^{-1}\|_{\cB(\cW^*,L^2)} + \kappa \|\wh T_\la^{-1}\|_{\cB(\cW^*,\cD_{\frt})} \right)\|af + g\|_{\cW^*} 
\\
& \qquad + |\la| \left( \|\wh T_\la^{-1}\|_{\cB(L^2,\cW)} + \iota \|\wh T_\la^{-1}\|_{\cB(L^2,L^2)} + \kappa \|\wh T_\la^{-1}\|_{\cB(L^2,\cD_{\frt})} \right)  \|f\|_{L^2}
\\ & \quad \ls 
\left(1 + \iota |\la|^{- \frac \gamma 2} + \kappa |\la|^{- \frac 1 2} \right) \|af + g\|_{\cW^*}  
+ \left(|\la|^{1-\frac{\gamma}{2}} + \iota |\la|^{1- \gamma} + \kappa |\la|^{1- \frac{\gamma+1} 2} \right) \|f\|_{L^2}
\\
& \quad \ls \left(1 + \iota |\la|^{- \frac \gamma 2} + \kappa |\la|^{- \frac 1 2} \right) \|F\|_\KK
\end{aligned}
\end{equation}
and 
\begin{equation}
\begin{aligned}
\|v\|
& \ls |\la| \|\wh T_\la^{-1}\|_{\cB(\cW^*,L^2)} \|af + g\|_{\cW^*}  +
|\la|^2 \|\wh T_\la^{-1}\|_{\cB(L^2)} \|f\|_{L^2} + \|f\|_{L^2}
\\[1mm]
& \ls
\|f\|_{L^2}+ \|af + g\|_{\cW^*} \ls \|F\|_\KK,
\end{aligned}
\end{equation}
(see~\eqref{KK.norm}). The inequalities~\eqref{r.A.small.b.K-H}, \eqref{r.A.small.b.K-L} and \eqref{r.A.small.b.K-D} follow by taking $\iota=\kappa=0$, $\iota=1$ and $\kappa=0$, and $\iota=\kappa=1$, respectively. Finally, \eqref{r.A.small-dt1} follows from \eqref{r.A.small.b.K-H} since $\|\Pi_2\|_{\cB(\HH,L^2)} = 1$ and \eqref{r.A.small-dt2} follows from \eqref{r.A.small.b.K-H} and \eqref{res.A.H-0L2}.
\end{proof}

\subsection{Low frequency estimate for the associated quadratic form}
\label{ssec:t.nearzero}

In the rest of the section we prove Proposition \ref{prop:tb.lbound}. The proof is based on a classical Neumann-bracketing argument (see \cite[Chap.~XIII.15]{Reed4}) and asymptotic perturbation theory for a family of self-adjoint operators in the sense of quadratic forms (see \cite[Chap.~VIII]{Kato-1966}). The main ingredient is the following lower bound.

\begin{lemma}\label{lem:hb.lbound}
	Let Assumption~\ref{asm:a.1} or Assumption~\ref{asm:a.2} hold.
	Then as $b \to 0^+$
	\begin{equation}
		\label{eq:Hb.lbound}
			\|\nabla u\|^2_{L^2} + b \|a^\frac12 u \|^2_{L^2}  \gs b^\gamma \|u \|^2_{L^2}, \qquad u \in H^1_0(\Omega) \cap \Dom(a^\frac12), 
	\end{equation}
	where $\gamma$ is as in~\eqref{gamma.def}.
\end{lemma}

\begin{proof}
Let $u\in H^1_0(\Omega) \cap \Dom(a^\frac12)$ be arbitrary. We split the problem into its parts on $\Omega_1$ and $\Omega_2$ as in Assumption~\ref{asm:a.1}. More precisely,
\begin{equation}
	\label{eq:hb.lbound.split}
		\|\nabla u \|^2_{L^2(\Omega)} + b\|a^\frac12 u\|^2_{L^2(\Omega)}  = \sum_{j=1}^2 \left( \|\nabla u_j\|^2_{L^2(\Omega_j)} + b \|a^\frac12 u_j \|^2_{L^2(\Omega_j)} \right),
\end{equation}
where we write $u_j := u\vert_{\Omega_j}$ for $j=1,2$. We thus may prove~\eqref{eq:Hb.lbound} for each summand separately.

\stepp \emph{Assumption~\ref{asm:a.1}:} Let us first consider the case when only Assumption~\ref{asm:a.1} holds. For the $j=2$ summand, we clearly have by Assumption~\ref{asm:a.1}~\ref{itm:Om12.a}
	\begin{equation}\label{eq:hb.lbound.1}
		\|\nabla u_2\|^2_{L^2(\Omega_2)} + b \|a^\frac12 u_2 \|^2_{L^2(\Omega_2)} \ge a_0 b \| u_2 \|_{L^2(\Omega_2)}^2.
	\end{equation}

	For $j=1$, we use a perturbative argument in $L^2 (\Omega_1)$. Let $H_0$ denote the Neumann Laplacian on $\Omega_1$, i.e. the self-adjoint operator associated with the quadratic form
	\begin{equation*}
		\frh_0 [v] := \| \nabla v \|_{L^2(\Omega_1)}^2, \qquad \Dom(\frh_0) := H^{1}(\Omega_1),
	\end{equation*}
	and consider the form perturbation
	\begin{equation*}
		\mathfrak a [v] := \| a^\frac12 v \|_{L^2(\Omega_1)}^2, \qquad \Dom(\mathfrak a) :=  \Dom(a^\frac12\vert_{\Omega_1}).
	\end{equation*}
	Using~\cite[Thm.~VIII.4.9]{Kato-1966} we derive an expansion for the lowest eigenvalue $\la_0(b)$ of the operator $H_b$ associated with the form $\frh_b := \frh_0 + b \mathfrak a$ in the regime $b \to 0^+$. 
	
	Note that both $\frh_0$ and $\mathfrak a$ are densely defined, closed and non-negative. Since $\Omega_1$ has Lipschitz boundary, the restrictions of functions in $\CcRd$ to $\Omega_1$ are dense in $H^1(\Omega_1)$ (see~\cite[Thm.~8.10.7]{bhattacharyya2012distributions}). Since these functions are also contained in  
	\begin{equation}
		\Dom (\frh_b) = H^1(\Omega_1) \cap \Dom(a^\frac12\vert_{\Omega_1}),
	\end{equation}
	the latter is dense in $\Dom (\frh_0)$. Hence, the conditions in \cite[p.~464]{Kato-1966} hold with $\frt = \frh_0$ and $\frt^{\bf (1)} = \mathfrak a$. Moreover, by \cite[Thms.~VIII.3.11,~VIII.3.15]{Kato-1966}, the operators $H_b$ converge to $H_0$ strongly in the generalized sense as $b \to 0^+$. Since $\Omega_1$ has Lipschitz boundary, the limit $H_0$ has compact resolvent (see e.g.~\cite[Thm.~8.11.4, ~p.~614]{bhattacharyya2012distributions}), while its lowest eigenvalue $\la_0(0) = 0$ is simple and stable (in the sense of the definition in \cite[p.~437]{Kato-1966}) with constant eigenfunction. With $P_0$ denoting the orthogonal projection on the corresponding eigenspace $\lspan\{1\}$, recalling that $a \in \LolocOm$ by assumption, we have
	\begin{equation}
		\Ran(P_0) \subset \Dom(a^\frac12\vert_{\Omega_1}) = \Dom (\frt^{\bf (1)}).
	\end{equation}
	
	By the previous paragraph, the assumptions of \cite[Thm.~VIII.4.9]{Kato-1966} are satisfied. Hence, the lowest eigenvalue of $H_b$ admits the expansion
	\begin{equation}
		\la_0(b) = b |\Omega_1|^{-1} \|a\vert_{\Omega_1}\|_{L^1(\Omega_1)} + o(b), \qquad b \to 0^+.
	\end{equation}
	Considering Assumption~\ref{asm:a.1}~(iii), it thus follows that
	\begin{equation}\label{eq:hb.lbound.2}
		\frh_b[v] \gs b \|v\|_{L^2(\Omega_2)}^2, \qquad v \in \Dom (\frh_b).
	\end{equation}
	Now~\eqref{eq:Hb.lbound} with $\gamma=1$ follows from~\eqref{eq:hb.lbound.split},~\eqref{eq:hb.lbound.1} and~\eqref{eq:hb.lbound.2} with $v = u_1 \in \Dom (\frh_b)$.

\stepp \emph{Assumption~\ref{asm:a.2}:} Let Assumption~\ref{asm:a.2} hold in addition, where $\Omega_1$ and $\Omega_2$ are as in \eqref{omega.a.unbd.def} with 
\begin{equation}\label{eq:r.a.unbd.def}
	r \equiv r(b) :=  b^{-\frac{1}{2 + \beta}}.
\end{equation}
For the summand $j=2$ in~\eqref{eq:hb.lbound.split}, we apply \eqref{a.unbd.asm} to find
\begin{equation}\label{eq:hb.lbound.unbdd.1}
	\|\nabla u_2\|^2_{L^2(\Omega_2)} + b \|a^\frac12 u_2 \|^2_{L^2(\Omega_2)} \ge  b r^{\beta} \| u_2 \|_{L^2(\Omega_2)}^2 = b^\frac{2}{2+\beta}  \| u_2 \|_{L^2(\Omega_2)}^2.
\end{equation}

To treat the summand $j=1$, let $v$ denote the zero extension of $u$ to $\Rd$, i.e.~$v \in H^1(\Rd)$ with $v(x) = 0$ on $\Rd \setminus \Omega$. Note that, for $b$ small enough, we have $r\ge 2r_0$ and thus
\begin{equation}
\mathbb B := B(0,1) \setminus \overline {B(0,\tfrac12)} \subset B(0,1) \setminus \overline {B(0,\tfrac{r_0}r)} \subset r^{-1} \Omega_1.
\end{equation}
Setting $w (y) = r^\frac d2 v(ry)$ and using~\eqref{a.unbd.asm}, we obtain the following lower estimate
\begin{equation}
	\begin{aligned}
		& \|\nabla u_1 \|_{L^2(\Omega_1)}^2 + b \|a^\frac12u_1\|_{L^2(\Omega_1)}^2 \\
		& \qquad  \qquad = \int_{B(0,r)} |\nabla v (x) |^2 \dd x + b \int_{\Omega_1} a(x) |v(x)|^2 \dd x \\
		& \qquad \qquad = r^{-2} \int_{B(0,1)} |\nabla w (y)|^2 \dd y + b \int_{r^{-1} \Omega_1} a (r y) |w(y)|^2 \dd y \\
		& \qquad \qquad \gs  r^{-2} \left( \int_{B(0,1)} | \nabla w(y)|^2 \dd y + \int_{\mathbb B} |y|^\beta |w(y)|^2 \dd y \right).
	\end{aligned}
\end{equation}
It is easy to see that the Neumann realization of $-\Delta + \1_{\mathbb B} |y|^\beta$ in $L^2 (B(0,1))$, i.e.~the self-adjoint operator associated with the form
\begin{equation}
	\frs [w]  : = \|\nabla w\|_{L^2(B(1,0))}^2 + \|\1_{\mathbb B} |y|^\frac\beta 2 w\|_{L^2(B(1,0))}^2, \qquad \Dom (\frs) := H^1(B(1,0)),
\end{equation}
has compact resolvent and a positive lowest eigenvalue. This further implies 
\begin{equation}\label{eq:hb.lbound.unbdd.2}
	\|\nabla u_1 \|_{L^2(\Omega_1)}^2 + b \|a^\frac12u_1\|_{L^2(\Omega_1)}^2  \gs  r^{-2} \|w\|^2_{L^2(B(0,1))} = b^{\frac{2}{2+\beta}} \|u_1\|_{L^2(\Omega_1)}^{2}.
\end{equation}
Putting together~\eqref{eq:hb.lbound.split}, \eqref{eq:hb.lbound.unbdd.1} and~\eqref{eq:hb.lbound.unbdd.2}, the claim~\eqref{eq:Hb.lbound} follows with $\gamma=2/(2+\beta)$.
\end{proof}

\begin{proof}[Proof of Proposition~\ref{prop:tb.lbound}]
For any $u \in \cD_\frt$ and $\la \in \ov{\D_+} \setminus \{0\}$, setting $b:=|\la|>0$ it can be readily verified that
\begin{equation}\label{t.la.rot}
\begin{aligned}
|\frt_\la[u]|
&\ge |\Re(\e^{-\frac{\ii}{2} \arg{\la}} \frt_\la[u])|
\\ &
\geq \frac{\sqrt 2}{2}
\big(
\|\nabla u\|_{L^2}^2 + \|q^\frac12 u\|_{L^2}^2 + b \|a^\frac12 u\|_{L^2}^2 \big) - b^2 \|u\|_{L^2}^2
.
\end{aligned}
\end{equation}
The claim now follows from Lemma~\ref{lem:hb.lbound}.
\end{proof}

\section{Resolvent estimates for high frequencies}
\label{sec:high}

It is for high frequencies that we add the assumption $a(x) \geq a_0 > 0$ a.e. in $\Omega$ in the statement of Theorem \ref{thm:decay}. Under this assumption, implying in particular the GCC, we prove uniform boundedness of the resolvent of $\cA$ on the imaginary axis away from zero. The challenge in the proof is the minimal regularity of the coefficients.  Some cases with non-uniformly positive damping are discussed in Section~\ref{sec:non.up.damping}.

\begin{proposition}
\label{prop:high.pos}
Let $\Ac$ and $\AK$ be as in \eqref{A.action}, \eqref{A.dom} and \eqref{AK.def}, respectively. Assume that there exists $a_0 > 0$ such that $a(x) \ge a_0$ a.e.~in $\Omega$. Then $\overline{\C_+} \setminus \set 0 \subset \rho(\cA)$ and for all $\eps >0$ it holds that
	\begin{equation}
		\label{res.n.A.large.b}
		\sup_{\la \in \ov {\C_+ \setminus \eps\D}} \|(\Ac - \la)^{-1} \|_{\cB(\HH)} < \infty, \quad
		\sup_{\la \in \ov {\C_+ \setminus \eps\D}} \|(\Ac_\KK - \la)^{-1} \|_{\cB(\KK)} < \infty.
	\end{equation}
\end{proposition}

The proof of Proposition \ref{prop:high.pos} relies on a simple resolvent estimate for $T_{\la}$ with $\la \in \ii \R$  together with the following lemma. The latter is an extension of the argument in \cite[Proof of Thm.~3.5]{arnal2026resolvent-dwe} and establishes the respective resolvent estimates for $\Ac$ (without assuming uniform positivity of the damping).
\begin{lemma}\label{lem:T.to.A}
Let $\Ac$ and $T_\la$ be as in \eqref{A.action}, \eqref{A.dom} and \eqref{Tla.def}, respectively.
Assume that for some $b_0 > 0$ we have $\s(\Ac) \cap \ii\R \subset \ii(-b_0,b_0)$.
Then there exists $C_0>0$ such that
\begin{equation}\label{eq:G.T.rate}
\|(\Ac - \ii b)^{-1} \|_{\cB(\HH)} \leq C_0 \left(\frac{1}{|b|}+|b| \|T_{\ii b}^{-1}\|_{\cB(L^2)} \right), \qquad |b| \geq b_0.
\end{equation}
\end{lemma}

\begin{proof}
We recall that by \eqref{spec.equiv} and Lemma~\ref{lem:T.la.ext} the operator $\widehat T_{\ii b}$ defined as in~\eqref{T.hat.def} is boundedly invertible for $\abs b \geq b_0$ and $(\Ac-\ii b)\inv$ is given by \eqref{A.res}.
Let $F = (f,g) \in \HH$. By \eqref{Dq.Riesz} we have
\begin{equation}\label{A.large.b.gen}
\begin{aligned}
& \|(\Ac-\ii b)^{-1} F \|_{\HH} \\
& \quad \ls |b|^{-1}\left(\|f\|_{\Wc}+\|\wh T_{\ii b}^{-1}(\Delta-q)f\|_{\cW} \right)+ \|\wh T_{\ii b}^{-1} g\|_{\Wc}
\\
& \quad \quad 
+ \|\wh T_{\ii b}^{-1}(\Delta-q)f\|_{L^2} + |b| \|\wh T_{\ii b}^{-1}g\|_{L^2}
\\
& \quad \ls |b|^{-1}\left(1+\|\wh T_{\ii b}^{-1}\|_{\cB(\cW^*,\cW)}\right) \|f\|_{\Wc} + \|\wh T_{\ii b}^{-1}\|_{\cB(\cW^*,L^2)} \|f\|_{\Wc}
\\
& \quad \quad + \|T_{\ii b}^{-1}\|_{\cB(L^2,\Wc)} \|g\|_{L^2} + |b| \| T_{\ii b}^{-1}\|_{\cB(L^2)} \|g\|_{L^2}.
\end{aligned}
\end{equation}

For $u \in \cD_{\frt}$ and $b \in \R$ we have
\begin{equation}
\|u\|_{\cW}^2 = \Re(\frt_{\ii b}[u]) + b^2 \|u\|_{L^2}^2 \leq |\frt_{\ii b}[u]| + b^2 \|u\|_{L^2}^2.
\end{equation}
For $v \in L^2(\Omega)$ we can use this with $u= T_{\ii b}^{-1} v \in \Dom (T_{\ii b})$ to arrive at
\begin{equation}
		\|T_{\ii b}^{-1} v\|_{\cW}^2
		\ls \left(\|T_{\ii b}^{-1}\|_{\cB(L^2)}+b^2 \|T_{\ii b}^{-1}\|_{\cB(L^2)}^2\right) \|v\|_{L^2}^2.
\end{equation}
Hence, as $\abs b \to \infty$,
\begin{equation}\label{T.ib.est}
\|T_{\ii b}^{-1}\|_{\cB(L^2,\cW)} \ls \sqrt{\|T_{\ii b}^{-1}\|_{\cB(L^2)}}+|b| \|T_{\ii b}^{-1}\|_{\cB(L^2)} \lesssim \frac 1 {\abs b} + |b| \|T_{\ii b}^{-1}\|_{\cB(L^2)}.
\end{equation}

Next we estimate for arbitrary $\varphi \in \cW^*$ that
\begin{equation*}
	\begin{aligned}
		\|\wh T_{\ii b}^{-1} \varphi\|_{L^2} & = \sup_{0\neq f \in L^2(\Omega)} \frac{|\langle f, \wh T_{\ii b}^{-1} \varphi \rangle_{L^2}|}{\|f\|_{L^2}} && = \sup_{0\neq f \in L^2(\Omega)} \frac{|\frt_{-\ii b}[\wh T_{-\ii b}^{-1} f, \wh T_{\ii b}^{-1} \varphi ]|}{\|f\|_{L^2}} \\
		& = \sup_{0\neq f \in L^2(\Omega)} \frac{|\frt_{\ii b}[\wh T_{\ii b}^{-1} \varphi , \wh T_{-\ii b}^{-1} f ]|}{\|f\|_{L^2}} && = \sup_{0\neq f \in L^2(\Omega)} \frac{|(\varphi , \wh T_{-\ii b}^{-1} f)_{\cW^*\times \cW}|}{\|f\|_{L^2}} \\
		& \le \|\varphi\|_{\cW^*} \sup_{0\neq f \in L^2(\Omega)} \frac{\|\wh T_{-\ii b}^{-1} f\|_{\cW}}{\|f\|_{L^2}} && \le \|\varphi\|_{\cW^*} \|T_{-\ii b}^{-1}\|_{\cB(L^2,\cW)}
	\end{aligned}
\end{equation*}
and hence
\begin{equation}\label{T.ib.-ib}
\|\wh T_{\ii b}^{-1}\|_{\cB(\cW^*,L^2)} \le  \| T_{-\ii b}^{-1}\|_{\cB(L^2,\cW)}\ls \frac 1 {|b|} + |b| \|T_{\ii b}^{-1}\|_{\cB(L^2)}, \qquad |b| \to \infty.
\end{equation}

Finally, we estimate the remaining term $\|\wh T_{\ii b}^{-1}\|_{\cB(\cW^*,\cW)}$. For distinct spectral parameters $\la, \mu \in \C \setminus (-\infty,0]$ with $0 \in \rho(T_\la)\cap \rho(T_\mu)$, the second resolvent identity leads to
\begin{equation}
\begin{aligned}
\wh T_{\la}^{-1} &= \wh T_{\mu}^{-1}  + (\mu-\la) \wh T_{\mu}^{-1} (a+\la+\mu) \wh T_{\la}^{-1}
\\
& = \frac{\mu}{\la} \wh T_{\mu}^{-1} + \mu (\mu-\la) \wh T_{\mu}^{-1} \wh T_{\la}^{-1} + \frac{\mu-\la}{\la} \wh T_{\mu}^{-1}(\Delta-q) \wh T_{\la}^{-1}.
\end{aligned}
\end{equation}
Note that there are no issues with the domains of the involved operators since $\widehat T_\la^{-1},\widehat T_\mu^{-1} \in \Bc(\Dc_\frt^*,\Dc_\frt)$ and $a,(\Delta-q) \in \Bc(\Dc_\frt,\Dc_\frt^*)$. Further manipulations yield
\begin{equation}\label{T.la.mu.id}
\left(
\wh T_{\mu}^{-1}(\Delta-q) +  \frac{\la}{ \la - \mu }
\right)  \wh T_{\la}^{-1}
= \frac{\mu}{\la - \mu} \wh T_{\mu}^{-1} - \la \mu \wh T_{\mu}^{-1} \wh T_{\la}^{-1}.
\end{equation}
We restrict this identity to $\cW^*$ and select $\la = \ii b$, $\mu = b$. The next step is to invert the operator in parentheses on the left hand side, i.e.
\begin{equation}
B -  \frac{\ii}{\ii-1}	
\end{equation}
where
\begin{equation}\label{B.def.TG}
B := I_{\cD_\frt \to \cW} \wh T_{b}^{-1} I_{\cW^* \to \cD_\frt^*} J_{\cW} \in \cB(\cW);
\end{equation}
recall that $J_{\cW}=-(\Delta-q) :\cW \to \cW^*$ is the Riesz isomorphism (see~\eqref{CD.def}). For $u,v \in \Wc$ we derive
	\begin{equation*}
		\begin{aligned}
				\langle u, Bv \rangle_{\cW} & = (J_{\cW} u, I_{\cD_\frt \to \cW} \wh T_{b}^{-1} I_{\cW^* \to \cD_\frt^*} J_{\cW} v )_{\cW^*\times \cW} \\
				& = (I_{\cW^* \to \cD_\frt^*} J_{\cW} u,  \wh T_{b}^{-1} I_{\cW^* \to \cD_\frt^*} J_{\cW} v )_{\cD_\frt^* \times \cD_\frt} \\
				& = \frt_{b} [\wh T_{b}^{-1} I_{\cW^* \to \cD_\frt^*} J_{\cW} u,  \wh T_{b}^{-1} I_{\cW^* \to \cD_\frt^*} J_{\cW} v]
		\end{aligned}
	\end{equation*}
Due to the symmetry of $\frt_{b}$, it follows that $B$ is symmetric and hence self-adjoint.
Then $\big(B-\ii/(\ii-1)\big)$ is boundedly invertible and
\begin{equation}\label{B.inv}
\left \|
\left( B -  \frac{\ii}{\ii-1} \right)^{-1}
\right \|_{\cB(\cW)} \leq 2.
\end{equation}

Observe that it follows from
\begin{equation}
	\frt_b[u] = \|\nabla u\|_{L^2}^2 + \|q^\frac12 u\|_{L^2}^2 + b \|a^\frac 12 u\|_{L^2}^2 + b^2 \|u\|_{L^2}^2, \qquad u \in \cD_{\frt},
\end{equation}
that (see the proof of Lemma~\ref{lem:T.ineq} for details on analogous arguments)
\begin{equation}\label{T.b.est}
	\|\wh T_{b}^{-1}\|_{\cB(\cW^*,\cW)} \leq 1, \qquad \| T_{b}^{-1}\|_{\cB(L^2,\cW)} \ls \frac{1}{|b|}, \qquad |b| \to  \infty.
\end{equation}

In summary, we have from \eqref{T.la.mu.id}, \eqref{B.inv}, \eqref{T.b.est} and \eqref{T.ib.-ib} that
\begin{equation}\label{T.ib.W0}
\begin{aligned}
\|\wh T_{\ii b}^{-1}\|_{\cB(\cW^*,\cW)} & \leq 2
\left(
\|\wh T_{b}^{-1}\|_{\cB(\cW^*,\cW)} + b^2 \|T_{b}^{-1}\|_{\cB(L^2,\cW)} \|\wh T_{\ii b}^{-1}\|_{\cB(\cW^*,L^2)}
\right)
\\
& \ls
1 + b^2 \frac{1}{|b|} \left(\frac 1 {|b|} + |b| \|T_{\ii b}^{-1}\|_{\cB(L^2)}\right), \qquad b \to \infty.
\end{aligned}
\end{equation}
Thus the lemma follows from \eqref{A.large.b.gen}, \eqref{T.ib.est}, \eqref{T.ib.-ib} and \eqref{T.ib.W0}.
\end{proof}

\begin{proof}[Proof of Proposition~\ref{prop:high.pos}]
For the first claim 
$
\overline\C_+ \setminus \{0\} \subset \rho(\Ac),
$
see Theorem~\ref{thm:A.basic} (ii).
For $u \in \Dc_\frt$ we have
\begin{equation*}
|\frt_{\ii b} [u]| \ge | \Im \frt_{\ii b} [u] |= |b| \|a^\frac12 u\|_{L^2}^2 \ge |b| a_0  \|u\|_{L^2}^2,
\end{equation*}
so $\|T_{\ii b}^{-1}\|_{\cB(L^2)}\ls |b|^{-1}$ for $0 \neq b\in \R$ and the first part of \eqref{res.n.A.large.b} is obtained by Lemma \ref{lem:T.to.A}. The second estimate then follows from \eqref{res.H.to.K}.
\end{proof}

\section{Time decay}
\label{sec:decay}

In this section we recall how in general one can convert resolvent estimates into time decay for the corresponding semigroup, and in particular we prove that Theorem \ref{thm:low} and Proposition \ref{prop:high.pos} imply Theorem \ref{thm:decay}.

\subsection{Proof of the main results}

We begin with Proposition \ref{prop:exp-decay}.

\begin{proof}[Proof of Proposition \ref{prop:exp-decay}]
By Corollary \ref{cor:0.sp}, \eqref{a.q.rel.bdd.intro} holds if and only if $0 \in \rho(\Ac)$.
First, if~\eqref{eq:exp-decay} holds, then $0 \in \rho(\Ac)$ by \cite[Thm.~II.3.8]{Engel-Nagel-book}. Conversely, assuming uniform positivity of $a$, it follows from $0 \in \rho(\cA)$ and Proposition \ref{prop:high.pos} that the resolvent of $\Ac$ is well-defined and uniformly bounded on the imaginary axis. Hence,~\eqref{eq:exp-decay} holds by the Gearhart--Pr\"uss theorem \cite[Thm.~V.1.11]{Engel-Nagel-book}.
\end{proof}

\begin{remark}
\label{rem:decay.mild}
If the operator norm of the semigroup decays in time, the decay is necessarily exponential (see~\cite[Prop.~V.1.2]{Engel-Nagel-book}). Hence, when $0 \in \sigma(\Ac)$ we cannot have a norm decay of the solutions $U (t) = \e^{t \Ac} F$ which is uniform with respect to the initial value $F \in \HH$, i.e.~an estimate of the type
\begin{equation}
\forall~t \ge 0, \,\, \forall~F \in \HH \,\,\, : \,\,\, \| U(t) \|_{\HH} \le \rho(t) \|F\|_{\HH}
\end{equation}
with a positive function $\rho$ decaying at infinity. To obtain the uniform estimates in Theorem~\ref{thm:decay}, one thus restricts to a subspace of initial values $F \in\KK$. Nevertheless, the energy of the solution decays to 0 for any initial value. This follows from the boundedness of the semigroup $(\e^{t\Ac})_{t\ge0}$, the estimates \eqref{u.energy.est} and~\eqref{u.partialt.est} in Theorem~\ref{thm:decay} and the density of $\KK$ in $\HH$.
\end{remark}

Now we turn to the proof of the main Theorem \ref{thm:decay}. For these estimates, we use the following abstract result.

\begin{theorem}\label{thm:exp.res}
Let $\scK$ be a Hilbert space and let $A$ be the generator of a strongly continuous semigroup on $\scK$. Let $\sfX$ and $\sfY$ be normed vector spaces. Let $\TX \in \Bc(\sfX,\scK)$ and $\TY \in \Bc (\scK , \sfY)$. Suppose that the following two assumptions hold.
\begin{enumerate}[\upshape (i)]
\item There exists $\tau_0 \geq 1$ such that
\begin{equation}
\overline{\C_+} \setminus \ii (-\tau_0, \tau_0) \subset \rho(A),
\end{equation}
and we have
\begin{equation} \label{hyp:high-freq}
\sup_{\la \in \overline{\C_+ \setminus \tau_0 \D}} \nr{(A-\la)\inv}_{\cB(\scK)} < \infty.
\end{equation}

\item There exist $m \in \N_0$, $\kappa \in [0,1)$ and $C_0 > 0$ such that for all
\begin{equation}
j \in
\begin{cases}
\{ m, m+1, m+2 \} & \text{if } \kappa = 0,\\
\{ m, m+1 \} & \text{if } \kappa \neq 0,
\end{cases}
\end{equation}
and $\la \in \tau_0 \D_+$, we have
\begin{equation}\label{hyp:low-freq}
\nr{\TY (A-\la)^{-1-j} \TX}_{\cB(\sfX,\sfY)}  \leq
C_0 \abs \la^{m-\kappa-j} .
\end{equation}
\end{enumerate}
Then there exists $C > 0$ such that for all $t \ge 0$ we have
\begin{equation} \label{eq:abstract-time-decay}
\nr{\TY \e^{tA} \TX}_{\cB(\sfX,\sfY)} \leq C
\pppg t^{-m-1+\kappa} .
\end{equation}
\end{theorem}
The proof of Theorem~\ref{thm:exp.res} is postponed to Section \ref{sec:from-resolvent-to-time-decay}. Now we use Theorem \ref{thm:exp.res} to prove Theorem \ref{thm:decay}.

\begin{proof}[Proof of Theorem~\ref{thm:decay}]
We apply Theorem \ref{thm:exp.res} with $A = \Ac_\KK$, $\scK = \sfX = \KK$ and $\TX = I_{\KK}$.
The assumption \eqref{hyp:high-freq} holds by Proposition~\ref{prop:high.pos}.

We begin with \eqref{u.energy.est}. By \eqref{r.A.small.b.H-H} and \eqref{r.A.small.b.K-H} in Theorem \ref{thm:low}, the estimates \eqref{hyp:low-freq} hold with $\sfY = \HH$, $\TY = I_{\KK \to \HH}$ and $m=\kappa = 0$. Then Theorem~\ref{thm:exp.res} gives
\[
\nr{\e^{t\Ac}}_{\Bc(\KK,\HH)} \lesssim \langle t \rangle^{-1}.
\]
This proves in particular \eqref{u.energy.est} (as well as an estimate for $\|\partial_t u(t)\|_{L^2}$ which is not good enough).

Next we prove \eqref{u.partialt.est} and \eqref{dtu.L2.est.gamma} simultaneously. Indeed, setting $\sfY = L^2(\Omega)$, $\TY = \Pi_2$, $m=1$ and $\kappa = \frac \gamma2$ (see \eqref{gamma.def} and \eqref{def:PI2}), it follows from \eqref{r.A.small.b.H-H}, \eqref{r.A.small-dt1} and \eqref{r.A.small-dt2} in Theorem~\ref{thm:low} that
\begin{equation}
	\nr{\Pi_2 \e^{t\Ac}}_{\Bc(\KK,L^2)} \lesssim \langle t \rangle^{\frac\gamma 2-2}.
\end{equation}

For \eqref{u.L2.est} we choose $\sfY = \cD_t$, $\TY= \Pi_1 : \KK \to\cD_t$ (see \eqref{def:PI2}), $m = 0$ and $\kappa = \frac 12$. The claim then follows since \eqref{r.A.small.b.K-D} and \eqref{r.A.small.b.H-H} give
\begin{equation}
	\nr{\Pi_1 \e^{t\Ac}}_{\Bc(\KK,\cD_{\frt})} \lesssim \langle t \rangle^{-\frac12}.
\end{equation}

Finally, for \eqref{u.L2.est.gamma} we use the estimates \eqref{r.A.small.b.K-L}, \eqref{r.A.small.b.H-H} to apply Theorem \ref{thm:exp.res} with $\sfY=L^2(\Omega)$, $\TY= I_{\cD_t \to L^2} \Pi_1 : \KK \to L^2(\Omega)$, $m=0$ and $\kappa = \frac \gamma 2$. This yields the final claim by
\begin{equation*}
	\nr{\Pi_1 \e^{t\Ac}}_{\Bc(\KK,L^2)} \lesssim \langle t \rangle^{\frac\gamma 2-1}. \qedhere
\end{equation*}
\end{proof}

\begin{remark} \label{rem:Batty}
Estimates of the solution for initial conditions in the range of the generator have been proved in general settings. With the resolvent estimates \eqref{r.A.small.b.H-H} and \eqref{res.n.A.large.b}, we can apply \cite[Thm.~7.6]{batty2016fine} to our bounded (contractive) $C_0$-semigroup $(\e^{t\Ac})_{t\ge0}$ in $\HH$, which gives
\begin{equation}
\label{sg.decay.rate}
\| \e^{t\Ac} \Ac (\Ac - 1)^{-1} \|_{\cB(\HH)} = \BigO\left(t^{-1}\right), \qquad t \to \infty.
\end{equation}
Then \eqref{u.energy.est} follows with Proposition \ref{prop:Ran.A}, since for $F \in \KK$ we get
\begin{equation} \label{semigr}
\begin{aligned}
\| \e^{t\Ac} F \|_{\HH}
\lesssim \langle t \rangle^{-1} \|(\Ac-1) \Ac^{-1} F \|_{\HH}
 \lesssim \langle t \rangle^{-1} {\nr{F}_{\KK}}.
\end{aligned}
\end{equation}
This however does not give \eqref{u.partialt.est} for $\|\partial_t u(t)\|_{L^2}$.
\end{remark}

\subsection{From resolvent estimates to time decay} \label{sec:from-resolvent-to-time-decay}

In this paragraph we prove Theorem \ref{thm:exp.res}. The proof is inspired by standard ideas (see the proof of the Gearhart--Pr\"uss theorem for the contribution of high frequencies and \cite[Prop.~5.3]{BoucletBur21} for the contribution of low frequencies).

\begin{proof}[Proof of Theorem~\ref{thm:exp.res}]
\stepp \emph{Approximating $\e^{tA}$ in terms of resolvent:}
Let $M \ge 1 $ and $\o \geq 0$ be such that
\[
\forall~t \geq 0 \,\,\, : \,\,\,  \|\e^{tA}\|_{\cB(\scK)} \leq M \e^{t\o},
\]
(see \cite[Prop.~I.5.5]{Engel-Nagel-book} and notice that if we can take $\o < 0$ then the result is clear). We choose some fixed $\mu > \o$. Let $\f \in \sfX$ and recall (see \cite[Thm.~II.1.10~(ii)]{Engel-Nagel-book}) that for all $\tau \in \R$ we have in $\scK$
\begin{equation} \label{eq:res-prop}
(A-(\mu+\ii\tau))\inv \TX \f = - \int_0^\infty  \e^{t(A-(\mu+\ii\tau))} \TX \f \diff t.
\end{equation}
Thus $\tau \mapsto (A-(\mu+\ii\tau))\inv \TX \f$ is the Fourier transform of $t \mapsto -\1_{\R_+} (t) \e^{-t\mu} \e^{tA} \TX \f$ (which belongs to $L^1(\R;\scK) \cap L^2(\R;\scK)$). For $R > 0$ we set
\[
\widetilde U_R(t) = -\frac 1 {2\pi} \int_{-R}^R \e^{\ii t\tau} (A-(\mu+\ii\tau))\inv \diff \tau  \quad \in \cB(\scK).
\]
Then it follows from \cite[Rem.~5.1.5]{Seifert-2022-} that
\begin{equation} \label{eq:etA-UR-1}
\int_0^{\infty} \|\TY \big( \e^{t(A-\mu)} - \widetilde U_R(t) \big) \TX \f\|_{\sfY}^2 \diff t \limt R \infty 0.
\end{equation}
\stepp \emph{Deriving sufficient condition for \eqref{eq:abstract-time-decay}:} We set
\[
U_{R}(t) = \e^{t\mu} \widetilde U_R(t) = -\frac 1 {2\ii \pi} \int_{\G_{R}} \e^{t\la} (A-\la)\inv \diff \la,
\]
where $\G_{R}$ is the line segment joining $\mu-\ii R$ to $\mu+\ii R$.
We show below that if there exists a constant $C > 0$ such that
\begin{equation} \label{eq:estim-UR}
	\forall~\f \in \sfX, \,\, \forall~ t\geq 0 \,\,\, : \,\,\,  \limsup_{R \to \infty} \|\TY U_{R}(t) \TX \f\|_{\sfY} \leq C \pppg t^{-m-1+\kappa} \nr{\f}_{\sfX},
\end{equation}
then \eqref{eq:abstract-time-decay} holds with this constant $C$. Indeed, assume for contradiction that there exist $t_0 > 0$ and $\f \in \sfX$ such that
\[
\eta = \|\TY \e^{t_0A} \TX \f\|_{\sfY} - C \pppg {t_0}^{-m-1+\kappa} \nr{\f}_{\sfX} > 0.
\]
By continuity, there exists a non-trivial interval $[t_1,t_2]$ around $t_0$ such that for all $t \in [t_1,t_2]$, using \eqref{eq:estim-UR} we have
\begin{multline*}
\liminf_{R \to \infty} \left( \| \TY \e^{tA} \TX \f\|_{\sfY} - \|\TY U_{R}(t) \TX \f\|_{\sfY} \right)\\
\geq \| \TY \e^{tA} \TX \f\|_{\sfY} - C \pppg t^{-m-1+\kappa} \nr{\f}_{\sfX} \geq \frac \eta 2.
\end{multline*}
This gives a contradiction with \eqref{eq:etA-UR-1}.

\stepp \emph{Splitting of contour integral:}
By \eqref{hyp:high-freq} there exist $\gamma > 0$ and $C_1>0$ such that for any $\la \in \C$ with $\Re\la \geq -2\gamma$ and $\abs{\Im(\la)} \geq \tau_0$ we have $\la \in \rho(A)$ and
\begin{equation} \label{eq:hyp-high-freq-2}
\nr{(A-\la)\inv}_{\cB(\scK)} \leq C_1.
\end{equation}
We consider $ \th \in C^\infty(\R;[0,1])$ supported in $[-2\tau_0,2 \tau_0]$ with $ \th = 1$ on $[-\tau_0,\tau_0]$.
Let $\eps \in (0,1]$. For $\tau \in \R$ we set
\[
\th_\eps (\tau) = -\g + (\eps +\g)  \th(\tau)  \in [-\g,\eps].
\]
In particular, $\th_\eps (\tau) = \eps$ if $\abs \tau \leq \tau_0$ and $\th_\eps(\tau) = -\gamma$ if $\abs \tau \geq 2\tau_0$. Then for $R \geq 2 \tau_0$ we consider the contour $\G_{R,\eps}$ defined by the parametrization
\[
\G_{R,\eps} : \fonc{[-R,R]}{\C}{\tau}{\th_\eps(\tau) + \ii \tau.}
\]
We also consider the contours $\G_{R,-}$ and $\G_{R,+}$ defined as the line segments joining $\mu-\ii R$ to $-\g-\ii R$ and $-\g+\ii R$ to $\mu+\ii R$, respectively (see Figure \ref{fig:contours}). Then  we have
\[
U_{R}(t) =  U_{R,-}(t) + U_{R,\eps}(t) + U_{R,+}(t), \qquad t \geq 0,
\]
where for $* \in \{-,\eps,+\}$ we have set
\[
U_{R,*}(t) = -\frac 1 {2\ii \pi} \int_{\G_{R,*}} \e^{t\la} (A-\la)\inv \diff \la.
\]
\begin{figure}
\begin{center}
\includegraphics[width = 0.3\linewidth]{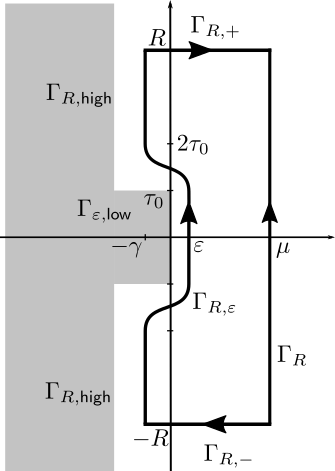}
\caption{The contours $\G_R$, $\G_{R,\pm}$, $\G_{R,\eps}$, $\G_{R,\high}$ and $\G_{\eps,\low}$ {\em (in grey a region which contains the spectrum of $A$)}.}
\label{fig:contours}
\end{center}
\end{figure}

\stepp \emph{Eliminating contribution of $\Gamma_{R,\pm}$ in \eqref{eq:estim-UR}:} Let $\f \in \sfX$ and $\psi \in \Dom(A)$. Given a fixed $\la_0 \in \rho(A)$ we have, for $\la \in \rho(A)$,
\[
(A-\la)\inv \psi = \frac 1 {\la-\la_0} \big( (A-\la)\inv (A-\la_0) - I  \big) \psi.
\]
With \eqref{eq:hyp-high-freq-2}, this implies that there exists $C_2>0$ independent of $t\geq 0$, $R > 0$ and $\psi \in \Dom(A)$ such that
\[
\nr{U_{R,\pm} (t) \psi}_\scK \leq \frac{C_2 \e^{t\mu} \left(\|A\psi\|_\scK + \|\psi\|_\scK \right)}{R}.
\]
On the other hand, for some $C_3 > 0$ independent of $\f \in \sfX$, $t\geq 0$ and $R > 0$, \eqref{eq:hyp-high-freq-2} gives that
\[
\nr{U_{R,\pm} (t) (\psi-\TX \f)}_\scK \leq C_3 \e^{t\mu} \nr{\psi- \TX \f}_{\scK} ,
\]
and so
\[
\nr{U_{R,\pm} (t) \TX \f}_\scK  \leq \frac {C_2 \e^{t\mu} \left(\|A\psi\|_\scK + \|\psi\|_\scK \right)}{R} + C_3 \e^{t\mu} \nr{\psi- \TX \f}_{\scK}.
\]
Choosing first $\psi \in \Dom(A)$ close to $\TX \f$ in $\scK$ and then $R$ large, we see that for all $t \geq 0$ and $\f \in \sfX$ we have
\begin{equation} \label{eq:lim-UR}
\nr{\TY U_{R,\pm} (t) \TX \f}_\sfY \limt R \infty 0.
\end{equation}
Now \eqref{eq:estim-UR} will follow from the estimate
\begin{equation} \label{eq:estim-UR-gamma}
\limsup_{R \to \infty} \nr{\TY U_{R,\eps}(t) \TX \f}_{\sfY} \leq C_4 \e^{\eps t} \pppg t^{-m-1+\kappa} \nr{\f}_{\sfX},
\end{equation}
where the constant $C_4>0$ is independent of $\f \in \sfX$, $t \geq 0$ and $\eps > 0$. Indeed, \eqref{eq:lim-UR} and \eqref{eq:estim-UR-gamma} imply that \eqref{eq:estim-UR-gamma} holds with $U_R(t)$ instead of $U_{R,\eps}(t)$ (the left-hand side not depending on $\eps$), so we can let $\eps \to 0$ in this estimate to arrive at \eqref{eq:estim-UR}.

\stepp \emph{Splitting of relevant contour $\Gamma_{R,\eps}$:} Define $\G_{\eps,\low}$ as the restriction of $\G_{R,\eps}$ to $\tau \in [-2\tau_0,2\tau_0]$ (which does not depend on $R \geq 2 \tau_0$), and let $\G_{R,\high}$ be the restriction to $\tau \in [-R,R] \setminus [-2\tau_0,2\tau_0]$ (which does not depend on $\eps$). Let
\begin{align*}
U_{R,\high}(t)
& =  -\frac 1 {2\ii\pi} \int_{\G_{R,\high}} \e^{t\la} (A-\la)\inv \diff \la\\
& = - \frac {\e^{-\g t}} {2\pi} \int_{2 \tau_0 \leq \abs \tau \leq R} \e^{\ii t\tau} (A-(-\g+\ii \tau))\inv \diff \tau
\end{align*}
and
\begin{align*}
U_{\eps,\low}(t)
& = -\frac 1 {2\ii \pi} \int_{\G_{\eps,\low}} \e^{t\la} (A-\la)\inv \diff \la\\
& = -\frac 1 {2 \ii \pi} \int_{\abs \tau \leq 2\tau_0} \e^{t(\th_\eps(\tau) + \ii\tau)} (\theta_\eps'(\tau) + \ii) (A-(\th_\eps(\tau)+\ii\tau))\inv \diff \tau.
\end{align*}
Then we have
\begin{equation} \label{eq:Uhigh-Ulow}
U_{R,\eps}(t) = U_{\eps,\low}(t) + U_{R,\high}(t), \qquad t \ge 0.
\end{equation}

\stepp \emph{\eqref{eq:estim-UR-gamma} for $U_{R,\high}$:} We begin with the contribution of high frequencies. By \eqref{eq:res-prop} and Plancherel theorem (see \cite[Thm.~5.1.4]{Seifert-2022-}), we have for $\f \in \sfX$
\[
\int_\R \|(A-(\mu+\ii \tau))\inv \TX \f\|_\scK^2 \diff \tau \leq \frac{M^2} {2\pi} \int_0^\infty \e^{-2t(\mu-\o)} \| \TX \f\|_\scK^2 \diff t = \frac {M^2 \nr{\TX \f}_\scK^2}{4\pi(\mu-\o)}.
\]
For $\abs \tau \geq 2 \tau_0$ we have
\begin{align*}
( A-(-\g+\ii \tau) )\inv
= \big( 1 - (\mu+\g) ( A-(-\g+\ii \tau) )\inv \big) (A-(\mu+\ii \tau) )\inv ,
\end{align*}
so, by \eqref{eq:hyp-high-freq-2},
\[
\|( A-(-\g+\ii \tau))\inv \TX \varphi\|_{\scK} \leq \big( 1 + C_1 (\mu+\gamma) \big) \|( A-(\mu+\ii \tau) )\inv \TX \varphi\|_{\scK}.
\]
After integration we get
\begin{equation} \label{eq:estim-res-L2-A}
\int_{\abs \tau \geq 2 \tau_0} \|( A-(-\g+\ii \tau) )\inv \TX \f\|_\scK^2 \diff \tau \leq \frac {M^2\big( 1 + C_1 (\mu+\gamma) \big)^2}{4\pi(\mu-\o)}\nr{\TX \f}_\scK^2.
\end{equation}
The adjoint $A^*$ of $A$ is the generator of the semigroup $\big((\e^{tA})^*\big)_{t \geq 0}$ (see \cite[Thm.~7.3.3]{Davies-2007}), which has the same growth bound as $(\e^{tA})_{t\geq 0}$. Moreover, for all $\abs \tau \geq 2\tau_0$, we have by \eqref{eq:hyp-high-freq-2}
\[
\|( A^*-(-\g-\ii \tau) )\inv \|_{\cB(\scK)} = \|( A-(-\g+\ii \tau) )\inv \|_{\cB(\scK)} \leq C_1,
\]
so for $\psi \in \sfY^*$ we similarly get
\begin{equation} \label{eq:estim-res-L2-A-star}
\int_{\abs \tau \geq 2 \tau_0} \|( A^*-(-\g-\ii \tau) )\inv \TY^* \psi\|_\scK^2 \diff \tau \leq \frac {M^2\big( 1 + C_1 (\mu+\gamma) \big)^2}{4\pi(\mu-\o)}\nr{\TY^* \psi}_\scK^2,
\end{equation}
where we use the identification $\scK^* \simeq \scK$. Now for $\f \in \sfX$ and $\psi \in \sfY^*$ we have
\begin{multline*}
(-\ii t) \innp{U_{R,\high}(t) \TX \f}{\TY^* \psi}_{\scK}\\
= - \frac {\e^{-\g t}} {2\pi} \int_{2 \tau_0 \leq \abs \tau \leq R} \left(-\partial_\tau \e^{\ii t\tau}\right) \langle(A-(-\g+\ii \tau))\inv \TX \f, \TY^*\psi \rangle_{\scK} \diff \tau .
\end{multline*}
We integrate by parts. The boundary terms can be neglected since they are exponentially decaying in $t$. More precisely, for $\abs\tau \in \{2\tau_0,R\}$ we have by \eqref{eq:hyp-high-freq-2}
\[
\abs{\frac {\e^{-\g t}} {2\pi} \e^{\ii t\tau} \langle(A-(-\g+\ii \tau))\inv \TX \f, \TY^* \psi\rangle_{\scK}} \leq \frac {\e^{-\g t}} {2\pi} C_1 \nr{\TX \f}_\scK \nr{\TY^*\psi}_\scK.
\]
Then, by the Cauchy-Schwarz inequality and \eqref{eq:estim-res-L2-A}--\eqref{eq:estim-res-L2-A-star},
\begin{eqnarray*}
\lefteqn{ \abs{\frac {\e^{-\g t}}{2\pi} \int_{2 \tau_0 \leq \abs \tau \leq R} \e^{\ii t\tau} \partial_\tau \langle (A-(-\g+\ii \tau))\inv \TX \f, \TY^* \psi\rangle_{\scK} \diff \tau}} \\
&& \qquad\leq \frac {\e^{-\g t}}{2\pi} \int_{2 \tau_0 \leq \abs \tau \leq R} \abs{ \langle(A-(-\g+\ii \tau))^{-2} \TX \f, \TY^* \psi\rangle_{\scK} } \diff \tau\\
&& \qquad \leq  \frac {\e^{-\g t}}{2\pi} \left(\int_{\abs \tau \geq 2 \tau_0}  \|(A-(-\g+\ii \tau))\inv \TX \f\|^2 \diff \tau \right)^{\frac 12}
\\
&& \qquad \qquad \qquad \quad \times  \left(\int_{\abs \tau \geq 2 \tau_0}   \|(A^*-(-\g-\ii \tau))\inv \TY^* \psi\|^2 \diff \tau \right)^{\frac 12}
\\
&& \qquad \leq  C_5 \e^{-\g t} \nr \f_{\sfX} \nr \p_{\sfY^*},
\end{eqnarray*}
for a constant $C_5>0$ independent of $t$, $R$, $\f$ and $\psi$. Finally, by the above there exists $C_6>0$ such that for all $t \geq 0$, $\f \in \sfX$ and $R > 0$ we have
\begin{equation} \label{estim:Uhigh}
\nr{\TY U_{R,\high}(t) \TX \f}_\sfY  = \sup_{0\neq \psi \in \sfY^*} \frac{|\big( \psi, \TY U_{R,\high}(t) \TX \f \big)_{\sfY^* \times \sfY}|}{\|\psi\|_{\sfY^*}}  \leq C_6 \e^{-t\g} \nr{\f}_\sfX.
\end{equation}

\stepp \emph{\eqref{eq:estim-UR-gamma} for $U_{\eps,\low}$:}
Now we estimate the contribution of low frequencies. It is enough to show \eqref{eq:estim-UR-gamma} for $t \ge 1$. Notice first that due to \eqref{hyp:high-freq}, \eqref{eq:hyp-high-freq-2} and $|\Gamma_{\eps,\low}(\tau)| \ge |\tau|$,~\eqref{hyp:low-freq} is valid (with a possibly different constant) for $\la$ running along the curve $\Gamma_{\eps,\low}$ and $\abs\la$ replaced by $\abs \tau$ on the right-hand side.

Let $\varphi \in \sfX$ and $\psi \in \sfY^*$ be arbitrary. Integrating by parts, for $j \in \N_0$, we start by deriving 
\begin{equation}
	\begin{aligned}
		& \int_{\abs \tau \leq 2\tau_0} \e^{t(\th_\eps(\tau) + \ii\tau)} (\th_\eps'(\tau) + \ii) \langle  (A-(\th_\eps(\tau)+\ii\tau))^{-1-j} \TX \varphi, \TY^* \psi \rangle_{\scK} \diff \tau  \\
		&  = \frac1t \int_{\abs \tau \leq 2\tau_0} \partial_\tau \left( \e^{t(\th_\eps(\tau) + \ii\tau)} \right) \langle  (A-(\th_\eps(\tau)+\ii\tau))^{-1-j}\TX \varphi, \TY^* \psi \rangle_{\scK} \diff \tau \\ 
		&  = \frac{j+1}{(-t)} \int_{\abs \tau \leq 2\tau_0} \e^{t(\th_\eps(\tau) + \ii\tau)} (\th_\eps'(\tau) + \ii)\langle  (A-(\th_\eps(\tau)+\ii\tau))^{-1-(j+1)}\TX \varphi, \TY^* \psi \rangle_{\scK} \diff \tau \\
		& \qquad \qquad \qquad \qquad  - \frac1{(-t)} \left[ \e^{t(\th_\eps(\tau) + \ii\tau)}  \langle  (A-(\th_\eps(\tau)+\ii\tau))^{-1-j}\TX \varphi, \TY^* \psi \rangle_{\scK} \right]_{-2\tau_0}^{2\tau_0}.
	\end{aligned}
\end{equation}
By an inductive argument, for $m \in \N_0$ we get
\begin{equation}
	\begin{aligned}
		& - 2 \ii \pi \langle U_{\eps,\low}(t) \TX \varphi, \TY^* \psi \rangle_{\scK} \\
		& = \int_{\abs \tau \leq 2\tau_0} \e^{t(\th_\eps(\tau) + \ii\tau)} (\th_\eps'(\tau) + \ii) \langle  (A-(\th_\eps(\tau)+\ii\tau))\inv \TX \varphi, \TY^* \psi \rangle_{\scK} \diff \tau  \\
		&  = \frac{m!}{(-t)^m} \int_{\abs \tau \leq 2\tau_0} \e^{t(\th_\eps(\tau) + \ii\tau)} (\th_\eps'(\tau) + \ii)\langle  (A-(\th_\eps(\tau)+\ii\tau))^{-1-m}\TX \varphi, \TY^* \psi \rangle_{\scK}  \diff \tau \\
		& \qquad \qquad \qquad - \sum_{j=0}^{m-1} \frac{j!}{(-t)^{j+1}} \left[ \e^{t(\th_\eps(\tau) + \ii\tau)}  \langle  (A-(\th_\eps(\tau)+\ii\tau))^{-1-j}\TX \varphi, \TY^* \psi \rangle_{\scK} \right]_{-2\tau_0}^{2\tau_0}.
	\end{aligned}
\end{equation}
Note that by \eqref{eq:hyp-high-freq-2} and since $\th_\eps(\pm 2 \tau_0) = -\g$, the boundary terms decay exponentially in time and can thus be neglected.

It remains to estimate the integral on the right hand side, where we split the integration at $\abs \tau = 1/t \le 1 \le \tau_0$.
For the integral around zero, we use \eqref{hyp:low-freq} with $j=m$ to obtain 
\begin{equation}
	\begin{aligned}
		\left| \int_{\abs \tau \leq \frac 1 t} \e^{t(\th_\eps(\tau) + \ii\tau)} (\th_\eps'(\tau) + \ii)\langle  (A-(\th_\eps(\tau)+\ii\tau))^{-1-m}\TX \varphi, \TY^* \psi \rangle_{\scK} \diff \tau\right| \qquad \qquad  \\
		\lesssim \e^{\eps t} \int_{\abs\tau \leq \frac 1 t}  \abs \tau^{-\kappa} \diff \tau \|\varphi \|_\sfX \|\psi\|_{\sfY^*}
		\lesssim \e^{\eps t} t^{\kappa -1} \|\varphi \|_\sfX \|\psi\|_{\sfY^*};
	\end{aligned}
\end{equation}
recall that $\theta_\eps'(\tau) = (\eps + \gamma) \theta'(\tau)$ is uniformly bounded in $\eps$. 

To estimate the remaining integral, we integrate by parts once more and obtain
\begin{equation}
	\begin{aligned}
		& \int_{\frac 1 t \leq \abs \tau \leq 2\tau_0} \e^{t(\th_\eps(\tau) + \ii\tau)} (\th_\eps'(\tau) + \ii)\langle  (A-(\th_\eps(\tau)+\ii\tau))^{-1-m}\TX \varphi, \TY^* \psi \rangle_{\scK} \diff \tau \\
		& = \frac{m+1}{-t} \int_{\frac 1 t \leq \abs \tau \leq 2\tau_0} \e^{t(\th_\eps(\tau) + \ii\tau)} (\th_\eps'(\tau) + \ii)\langle  (A-(\th_\eps(\tau)+\ii\tau))^{-1-(m+1)}\TX \varphi, \TY^* \psi \rangle_{\scK} \diff \tau \\
		& \qquad \qquad \qquad \qquad  + \frac1t \left[ \e^{t(\th_\eps(\tau) + \ii\tau)}  \langle  (A-(\th_\eps(\tau)+\ii\tau))^{-1-m}\TX \varphi, \TY^* \psi \rangle_{\scK} \right]_{-2\tau_0}^{2\tau_0} \\
		& \qquad \qquad \qquad \qquad  -  \frac1t \left[ \e^{t(\th_\eps(\tau) + \ii\tau)}  \langle  (A-(\th_\eps(\tau)+\ii\tau))^{-1-m}\TX \varphi, \TY^* \psi \rangle_{\scK} \right]_{-\frac1t}^{\frac1t}.
	\end{aligned}
\end{equation}
As before, the boundary terms at $\tau = \pm 2 \tau_0$ are decaying exponentially in time. Using that $\th_\eps(\pm 1/t) = \eps$ together with \eqref{hyp:low-freq} for $j=m$, the boundary terms at $\tau = \pm 1/t$ are dominated by $\e^{\eps t} t^{\kappa-1}\nr \f_{\sfX} \nr \p_{\sfY^*}$. 

Altogether, in the case $\kappa \neq 0$, it follows from \eqref{hyp:low-freq} with $j=m+1$ that
\begin{equation}
	\begin{aligned}
		& \left| \int_{\frac 1 t \leq \abs \tau \leq 2\tau_0} \e^{t(\th_\eps(\tau) + \ii\tau)} (\th_\eps'(\tau) + \ii)\langle  (A-(\th_\eps(\tau)+\ii\tau))^{-1-m}\TX \varphi, \TY^* \psi \rangle_{\scK} \diff \tau \right| \\
		& \qquad \qquad \ls  \e^{\eps t} t^{\kappa-1} \|\varphi \|_\sfX \|\psi\|_{\sfY^*} + {\e^{\eps t}}t^{-1} \int_{\frac 1 t \leq \abs \tau \leq 2\tau_0}   \abs{\tau}^{-\kappa-1} \diff \tau  \|\varphi \|_\sfX \|\psi\|_{\sfY^*} \\
		& \qquad \qquad  \ls \e^{\eps t} t^{\kappa-1} \|\varphi \|_\sfX \|\psi\|_{\sfY^*}.
	\end{aligned}
\end{equation}

If $\kappa = 0$ we integrate by parts once more (notice that the bound \eqref{hyp:low-freq} for $j = m+2$ is only used in this case). As before, the boundary terms at $\pm 2 \tau_0$ are exponentially decaying, while from \eqref{hyp:low-freq} with $j=m+1$ we see that the boundary terms at $\pm 1/t$ are dominated by $\e^{\eps t}t^{-1}\nr \f_{\sfX} \nr \p_{\sfY^*}$, so we can write 
\begin{equation}
	\begin{aligned}
		& \int_{\frac 1 t \leq \abs \tau \leq 2\tau_0} \e^{t(\th_\eps(\tau) + \ii\tau)} (\th_\eps'(\tau) + \ii) \langle  (A-(\th_\eps(\tau)+\ii\tau))^{-1-m}\TX \varphi, \TY^* \psi \rangle_{\scK} \diff \tau \\
		& = \frac{(m+1)(m+2)}{t^2} \int_{\frac 1 t \leq \abs \tau \leq 2\tau_0} \e^{t(\th_\eps(\tau) + \ii\tau)} (\th_\eps'(\tau) + \ii) \\
		& \qquad \qquad \qquad \qquad \qquad \qquad \qquad \times \langle  (A-(\th_\eps(\tau)+\ii\tau))^{-1-(m+2)}\TX \varphi, \TY^* \psi \rangle_{\scK} \diff \tau \\
		& \qquad \qquad \qquad \qquad  + \Oc(\e^{\eps t}t^{-1})  \|\varphi \|_\sfX \|\psi\|_{\sfY^*}.
	\end{aligned}
\end{equation}
As above, we get from \eqref{hyp:low-freq} with $j=m+2$ that
\begin{equation}
	\begin{aligned}
		& \left| \int_{\frac 1 t \leq \abs \tau \leq 2\tau_0} \e^{t(\th_\eps(\tau) + \ii\tau)} (\th_\eps'(\tau) + \ii)\langle  (A-(\th_\eps(\tau)+\ii\tau))^{-1-m}\TX \varphi, \TY^* \psi \rangle_{\scK} \diff \tau \right| \\
		& \qquad \qquad \ls  \e^{\eps t} t^{-1} \|\varphi \|_\sfX \|\psi\|_{\sfY^*} + {\e^{\eps t}}t^{-2} \int_{\frac 1 t \leq \abs \tau \leq 2\tau_0}   \abs{\tau}^{-2} \diff \tau  \|\varphi \|_\sfX \|\psi\|_{\sfY^*} \\
		& \qquad \qquad  \ls \e^{\eps t} t^{-1} \|\varphi \|_\sfX \|\psi\|_{\sfY^*}.
	\end{aligned}
\end{equation}
In any case, as in \eqref{estim:Uhigh} for high frequencies, it follows that there exists $C_7 >0$ such that for all $t\geq 0$, $\eps > 0$ and $\varphi \in \sfX$ we have
\begin{equation} \label{estim:Ulow}
\nr{\TY U_{\eps,\low}(t) \TX \varphi }_{\sfY} \leq C_7 \e^{\eps t} \pppg t^{-m-1+\kappa} \|\varphi\|_{\sfX}.
\end{equation}
With \eqref{estim:Uhigh} and \eqref{estim:Ulow} we deduce \eqref{eq:estim-UR-gamma}, which concludes the proof.
\end{proof}

\section{The case of constant damping and optimality of the decay} \label{sec:optimal}

In this section we prove optimality for the estimates of Theorem \ref{thm:decay} by considering the model case $q(x)\equiv 0$ and $a(x)\equiv 1$ on the full Euclidean space $\Omega=\Rd$. More precisely, we prove the following result (see also Remark~\ref{rem:almost.opt} below), where for $j \in \{0,1,2\}$ we have set
\[
\cc_j = \left(\sup_{s \geq 0} s^{j+1} \e^{-2s} \right)^{\frac 12} = \left( \frac {j+1} {2\e} \right)^{\frac {j+1} 2}.
\]

\begin{proposition}\label{prop:lower-bound}
Suppose $\Omega = \R^d$, $q (x) \equiv 0$ and $a (x)\equiv 1$. Let $\rho_0,\rho_1,\rho_2 : \R_+ \to \R_+$ be such that for all $F \in \KK$ and all $t > 0$ we have
\begin{equation} \label{eq:hyp-rho}
\begin{aligned}
\nr{u(t)}_{L^2} & \leq \rho_0(t)\nr F_\KK,\\
\nr{\nabla u(t)}_{L^2} & \leq \rho_1(t) \nr{F}_{\KK}, \\
\nr{\partial_t u(t)}_{L^2} & \leq \rho_2(t) \nr{F}_{\KK},
\end{aligned}
\end{equation}
where $\KK$ is as in~\eqref{KK.def}--\eqref{KK.norm} and $u(t)$ is the solution of \eqref{dwe.2ndorder}. For every $\eps > 0$ there exists $t_0 \geq 1$ such that for all $t \ge t_0$ we have
\begin{equation} \label{eq:concl-rho}
\rho_0(t) \geq \frac {\cc_0 - \eps} {\sqrt{t}}, \qquad
\rho_1(t) \geq \frac {\cc_1 - \eps }{ t}, \qquad
\rho_2(t) \geq  \frac {\cc_2 - \eps} {t^{\frac 32}}.
\end{equation}
\end{proposition}

To prove Proposition \ref{prop:lower-bound}, we first check that $u(t)$ behaves for large times like a solution of the heat equation.

\begin{lemma} \label{lem:comparison-heat}
Suppose $\Omega = \R^d$, $q (x) \equiv 0$ and $a (x)\equiv 1$. Let $F = (f,g) \in \Kc$ with $\KK$ as in~\eqref{KK.def}--\eqref{KK.norm}, and let $u(t)$ be the solution of \eqref{dwe.2ndorder}. For $t\geq 0$ we set $u_0(t) = \e^{t\Delta} (f+g)$. Then there exists $c > 0$ such that for all $t \geq 1$ we have
\begin{equation}
\begin{aligned}
	\label{eq:comp.heat}
\nr{u(t)-u_0(t)}_{L^2} & \leq c t^{-1} \nr{F}_{H^1 \oplus L^2},\\
\nr{\nabla (u(t)-u_0(t))}_{L^2} & \leq c t^{-\frac 32} \nr{F}_{H^1 \oplus L^2},\\
\nr{\partial_t (u(t)-u_0(t))}_{L^2} & \leq c t^{-2} \nr{F}_{H^1 \oplus L^2}.
\end{aligned}
\end{equation}
\end{lemma}

\begin{proof}
For $t \geq 0$ we recall that $U(t) = \e^{t\Ac} F = (u(t),\partial_t u(t))$. We  prove \eqref{eq:comp.heat} for the corresponding Fourier transformed quantities (with respect to the space variable $x$), uniformly in $F \in \KK$ and $t \ge 1$ (fixed but arbitrary). For $\xi \in \Rd$ we have
\[
\widehat U(t,\xi) = \e^{tM(\xi)} \widehat F (\xi), \qquad M(\xi) = 
\begin{pmatrix} 0 & 1 \\ -\abs \xi^2 & -1 
\end{pmatrix},
\]
as well as
\begin{equation}
	\label{eq:u0.hat}
	\wh u_0 (t,\xi) =   \e^{-t|\xi|^2} (\wh f (\xi)+ \wh g(\xi)).
\end{equation}
Moreover, we have
\begin{equation}	
	\label{eq:dt.u0.hat}
	\qquad \wh{\partial_t u}_0(t,\xi) = \partial_t \wh u_0 (t,\xi) =  -|\xi|^2 \wh u_0(t,\xi).
\end{equation}
Corresponding to the weight in the norm in $\cF ( H^1 (\Rd) \oplus L^2(\Rd))$, for $\xi \in \R^d$ and $\z = (\z_1,\z_2) \in \C^2$ we set
\[
\abs{\z}_{\C^2_\xi}^2 := (\abs \xi^2 + 1) \abs{\z_1}^2 + \abs{\z_2}^2.
\]
To prove the Fourier analogues of~\eqref{eq:comp.heat}, we split the integration in three regions:
\begin{enumerate}[\upshape (i)]
	\item high frequencies $|\xi | > 1$;
	\item intermediate frequencies $\delta \le |\xi | \le 1$, with $\delta \in(0,1/8)$ selected suitably in \eqref{eq:low.delta};
	\item small frequencies $|\xi|<\delta$.
\end{enumerate}
The main contribution comes from the part in (iii), while (i) and (ii) lead to exponentially small terms in time (both for $\wh u(t)$ and $\wh u_0(t)$ separately). Note that in (ii) and (iii), the Euclidean norm $|\cdot|_{\C^2}$ is equivalent to $|\cdot|_{\C^2_\xi}$ (with uniform constants in $\xi$).

\stepp 
\emph{Basic estimates and expansions for $\wh U(t)$:}	
The eigenvalues of $M(\xi)$ read
\[
\la_\pm(\xi) = \frac {-1 \pm \sqrt{1-4\abs \xi^2}}2, \qquad \la_+ (\xi) - \la_-(\xi) = \sqrt{1-4|\xi|^2},
\]
(with one double eigenvalue if $4 \abs \xi^2 = 1$). Their real parts satisfy
\begin{equation}
	\label{eq:re.la.pm}
	\begin{aligned}
			\Re \la_-(\xi) & \le -\frac12, \qquad & \xi & \in \Rd, \\
			\Re \la_\pm (\xi)  & = -\frac12, \qquad & |\xi| & \ge \frac14, \\
			\Re \la_+(\xi) & \le \frac{-1 + \sqrt{1-4\delta^2}}{2} := - \eta \equiv - \eta(\delta) < 0, \qquad & |\xi | & \ge \delta.
	\end{aligned}
\end{equation}
When $4 \abs \xi^2 \neq 1$, we can write
\begin{equation}
	\label{eq:eigenproj}
	\e^{tM(\xi)} = \e^{t\la_+(\xi)} \big\langle\Psi_+(\xi),\cdot \big\rangle_{\C^2} \Phi_+ (\xi) + \e^{t\la_-(\xi)} \big\langle \Psi_-(\xi), \cdot\big\rangle_{\C^2} \Phi_-(\xi),
\end{equation}
where
\[
\Phi_\pm(\xi) = \begin{pmatrix} 1 \\ \la_\pm(\xi) \end{pmatrix}, \qquad \overline{ \Psi_\pm(\xi)} = \pm \frac 1 {\la_+(\xi)-\la_-(\xi)} \begin{pmatrix} - \la_\mp(\xi) \\ 1 \end{pmatrix}.
\]
We have
\begin{equation}
	\label{eq:Phi.pm.est.high.freq}
	|\la_+(\xi) - \la_-(\xi)| \gs |\xi|, \qquad |\la_\pm(\xi)| \le |\Phi_\pm(\xi)|_{\C^2_\xi}  \lesssim \abs \xi, \qquad |\xi|  > 1,
\end{equation}
as well as 
\begin{equation}\label{eq:bound.Phi-.Psi-}
	|\la_+(\xi) - \la_-(\xi)| \gs 1, \qquad | \Psi_-(\xi)|_{\C^2}  |\Phi_-(\xi)  |_{\C^2}  \ls 1, \qquad |\xi| < \delta ,
\end{equation}
(independently of $\delta<1/8$). One moreover computes
\begin{equation}
	\label{eq:exp.la+}
	\la_+(\xi) = -\abs \xi^2 + \Oc(\abs \xi^4), \qquad \xi \to 0,
\end{equation}
and finally
\begin{equation}
	\label{eq:exp.Psi.Phi}
	\Psi_+(\xi) = \begin{pmatrix} 1 \\ 1  \end{pmatrix} + \Oc(\abs \xi^2), \qquad \Phi_+(\xi) = \begin{pmatrix} 1 \\ -\abs \xi^2 + \Oc(\abs \xi^4) \end{pmatrix}, \qquad \xi \to 0.
\end{equation}

\stepp \emph{(i) High  /  (ii) intermediate frequencies for $\wh u_0(t)$:} For every $|\xi|\ge \delta$, using~\eqref{eq:u0.hat} we estimate 
\begin{equation}
	\label{eq:bound.nabla.u0.hat}
	|\xi \wh u_0(t, \xi)| \le \frac1{\sqrt t} \sqrt{t |\xi|^2} \e^{-\frac t2 |\xi|^2}\e^{-\frac t2|\xi|^2} |\wh F(\xi)|_{\C^2_\xi}\ls \e^{-\frac {\delta^2}2 t}|\wh F(\xi)|_{\C^2_\xi}.
\end{equation}
Similarly, considering~\eqref{eq:dt.u0.hat}, one derives
\begin{equation}
	\label{eq:bound.dt.u0.hat}
	|\wh{\partial_t u}_0 (t, \xi)| \ls \e^{-\frac {\delta^2}2 t} |\wh F(\xi)|_{\C^2_\xi}, \qquad |\xi|\ge \delta.
\end{equation} 
By integration, it follows from~\eqref{eq:u0.hat}, \eqref{eq:bound.nabla.u0.hat} and \eqref{eq:bound.dt.u0.hat} that 
\begin{equation} \label{eq:u0.high}
	\begin{aligned}
		 \|\wh u_0(t)\|_{L^2(|\xi|\ge \delta)} + & \|\xi  \wh u_0(t)\|_{L^2(|\xi|\ge \delta)} + \|\wh{\partial_t u}_0(t)\|_{L^2(|\xi|\ge \delta)} \\
		& \qquad \qquad \lesssim \e^{-\frac {\delta^2}2 t} \| |\wh F (\xi)|_{\C^2_\xi}\|_{L^2(|\xi|\ge \delta)} \le \e^{-\frac {\delta^2}2 t} \| F\|_{H^1\oplus L^2}.
	\end{aligned} 
\end{equation}

\stepp \emph{(i) High frequencies for $\wh U(t)$:} 
From \eqref{eq:Phi.pm.est.high.freq}, we obtain for $|\xi |> 1$ that
\begin{equation}
	\begin{aligned}
	\big| \langle \Psi_\pm(\xi),\widehat F(\xi) \rangle_{\C^2} \big| & \lesssim \frac{1}{|\la_+(\xi) - \la_-(\xi)|} \left(|\la_\mp(\xi) | |\wh f(\xi)| + |\wh g(\xi)| \right) \\
	& \ls |\wh f(\xi)| + |\xi|^{-1} |\wh g(\xi)|  \lesssim \abs \xi\inv |\wh F(\xi)|_{\C^2_\xi}.
	\end{aligned}
\end{equation}
Combining this with \eqref{eq:eigenproj}, \eqref{eq:re.la.pm} and \eqref{eq:Phi.pm.est.high.freq}, we get
\begin{equation} \label{eq:estim-etM-high-freq}
	\begin{aligned}
		|\wh U(t)|_{\C^2_\xi} = |\e^{tM(\xi)} \wh F(\xi)|_{\C^2_\xi}
		\lesssim \e^{-\frac {t}2} |\widehat F(\xi)|_{\C^2_\xi}, \qquad |\xi| > 1.
	\end{aligned}
\end{equation}
Integrating this estimate leads to 
\begin{equation} 
	\label{eq:u.high}
		\|\wh u(t)\|_{L^2(|\xi|>1)} +  \|\xi \wh u(t) \|_{L^2(|\xi|>1)} + \|\wh{\partial_t u}(t) \|_{L^2(|\xi|>1)}  \lesssim \e^{-\frac t2} \|F\|_{H^1\oplus L^2}.
\end{equation}

\stepp \emph{(ii) Intermediate frequencies for $\wh U(t)$:}
For every fixed $\xi_0 \in \Rd$ such that $\delta \le |\xi_0| \le 1$, by computing the matrix exponential of $M(\xi_0)$ and considering \eqref{eq:re.la.pm}, there exists $C_{\xi_0}>0$ such that
\begin{equation}
	\|\e^{tM(\xi_0)}\|_{\cB(\C^2)} \leq C_{\xi_0} \e^{-\eta t} (1 + t) \ls C_{\xi_0} \e^{-\frac{\eta}{2} t};
\end{equation}
note that the factor $(1+t)$ is due to the Jordan block structure for $|\xi_0| = 1/2$.  
By continuity of $M$, there exists a neighborhood $\Vc_{\xi_0}$ of $\xi_0$ such that this estimate holds for all $\xi \in \Vc_{\xi_0}$ with $\eta/2$ replaced by $\eta/4$. Finally, by compactness, we obtain
\begin{equation} \label{eq:estim-etM-intermediate-freq}
	\|\e^{tM(\xi)}\|_{\cB(\C^2)} \ls \e^{-\frac{\eta}4 t}, \qquad \delta \le |\xi| \le 1.
\end{equation}
Similarly as in (i), upon suitable integration (and using the uniform equivalence between $|\cdot|_{\C^2}$ and $|\cdot|_{\C^2_\xi}$) this leads to 
\begin{equation} \label{eq:u.interm}
	\begin{aligned}
		\|\wh u(t)\|_{L^2(\delta \le |\xi|\le1)} + & \|\xi \wh u(t) \|_{L^2(\delta \le |\xi|\le1)} \\
		 & \qquad \quad + \|\wh{\partial_t u}(t) \|_{L^2(\delta \le |\xi|\le1)} \lesssim  \e^{-\frac \eta4 t}  \|F\|_{H^1\oplus L^2}.
	\end{aligned} 
\end{equation}

\stepp \emph{(iii) Low frequencies for $\wh u(t) - \wh u_0(t)$:}
Considering~\eqref{eq:exp.la+}, we can select $\delta$ sufficiently small such that 
\begin{equation}
	\label{eq:low.delta}
	|\la_+(\xi) + |\xi|^2 | \le \frac{|\xi|^2}{2}, \qquad |\xi| <\delta;
\end{equation}
note also that $\la_+(\xi) \le 0$. It follows that, for $|\xi |<\delta$,
\begin{align}
	\label{eq:exp.e^tla+}
	\big|\e^{t\la_+(\xi)} - \e^{-t|\xi|^2}\big| & = \left|\la_+(\xi) + |\xi|^2\right| \int_0^t \e^{-t|\xi|^2 \left(1-\frac st \frac{\la_+(\xi)+ |\xi|^2}{|\xi|^2}\right)} \diff s \\
	& \le \left|\la_+(\xi) + |\xi|^2\right| t \e^{-\frac t2 |\xi|^2} = t \e^{-\frac t2 |\xi|^2} \Oc (|\xi|^4), \qquad \xi \to 0.
\end{align}
Using \eqref{eq:exp.Psi.Phi} and \eqref{eq:u0.hat}, for $|\xi |<\delta$, we estimate 
\begin{equation}
	\label{eq:e^tla.Psi+.low}
	\begin{aligned}
		& \e^{t\la_+(\xi)}\big\langle \Psi_+(\xi), \wh F(\xi) \big\rangle_{\C^2} \\
		&\qquad \quad  = \left(\e^{t\la_+(\xi)} - \e^{-t|\xi|^2} \right)\big\langle \Psi_+(\xi), \wh F(\xi) \big\rangle_{\C^2} +  \e^{-t |\xi|^2}  \big\langle \Psi_+(\xi), \wh F(\xi) \big\rangle_{\C^2} \\
		& \qquad \quad = \left(\e^{t\la_+(\xi)} - \e^{-t|\xi|^2} \right) \left( \wh f(\xi) + \wh g(\xi) + \Oc \big(\abs \xi^2 |\wh F(\xi)|_{\C^2}\big)  \right) \\
		& \qquad \qquad \qquad \qquad \quad + \wh u_0 (t,\xi) + \e^{-t|\xi|^2} \Oc \big(\abs \xi^2 |\wh F(\xi)|_{\C^2}\big), \qquad \xi \to 0.
	\end{aligned}
\end{equation}
From the first component of~\eqref{eq:eigenproj}, by combining 
\eqref{eq:e^tla.Psi+.low}, \eqref{eq:exp.e^tla+}, \eqref{eq:re.la.pm}, \eqref{eq:bound.Phi-.Psi-} and using $|\cdot|_{\C^2} \approx |\cdot|_{\C^2_\xi}$, we compute 
\begin{equation}
\label{eq:estim-etM-low-freq}
\begin{aligned}
	\big| \wh u(t,\xi) - \wh u_0(t,\xi) \big| & \le \left|\e^{t\la_+(\xi)} \big\langle \Psi_+(\xi), \wh F(\xi) \big\rangle_{\C^2} - \wh u_0(t,\xi) \right| \\
	& \qquad \qquad \qquad \quad  + | \e^{t\la_-(\xi)} | | \Psi_-(\xi)|_{\C^2}  |\Phi_-(\xi)  |_{\C^2} |\wh F(\xi) |_{\C^2} \\
	& \lesssim  \left( \e^{-t\abs \xi^2}  \Oc(\abs \xi^2)  + t \e^{-\frac t2 \abs \xi^2}  \BigO(\abs \xi^4) + \e^{-\frac t 2} \right) |\wh F(\xi)|_{\C^2_\xi}, \\
	& \ls \frac1t \left( t |\xi|^2 \e^{-t\abs \xi^2} + (t |\xi|^2)^2 \e^{-\frac t2 \abs \xi^2} + t \e^{-\frac t 2} \right) |\wh F(\xi)|_{\C^2_\xi}, 
\end{aligned}
\end{equation}
for all $|\xi |<\delta$. 
After integration, and using that the last bracket on the right above is uniformly bounded in $\xi$ and $t$, we have
\begin{equation} \label{eq:u-u0}
\nr{\wh u(t) - \wh u_0(t)}_{L^2(|\xi|<\delta)}  \ls \frac1t \|F\|_{H^1 \oplus L^2}.
\end{equation}
By a similar argument, \eqref{eq:estim-etM-low-freq} implies that 
\begin{equation} \label{eq:u-u0-nabla}
\nr{\xi \big(\wh u(t) - \wh u_0(t) \big)}_{L^2(|\xi|<\delta)}
 \ls \frac{1}{t^\frac32} \|F\|_{H^1 \oplus L^2}.
\end{equation}
For the remaining  estimate, using the second component of~\eqref{eq:eigenproj} and combining \eqref{eq:exp.Psi.Phi}, \eqref{eq:dt.u0.hat},  \eqref{eq:e^tla.Psi+.low}, \eqref{eq:exp.e^tla+}, \eqref{eq:re.la.pm}, \eqref{eq:bound.Phi-.Psi-} and  $|\cdot|_{\C^2} \approx |\cdot|_{\C^2_\xi}$, we obtain
\begin{equation}
	\begin{aligned}
		\big| \wh{\partial_t u}(t,\xi) - \wh{\partial_t u}_0 (t,\xi) \big| & \ls \big|\e^{t\la_+(\xi)} \big\langle \Psi_+(\xi), \wh F(\xi) \big\rangle_{\C^2} \big(-|\xi|^2 + \Oc (|\xi|^4)\big) + |\xi|^2 \wh u_0(t,\xi) \big| \\
		& \qquad \qquad + | \e^{t\la_-(\xi)} | | \Psi_-(\xi)|_{\C^2}  |\Phi_-(\xi)  |_{\C^2} |\wh F(\xi) |_{\C^2} \\
		&  \lesssim  \left( \e^{-t\abs \xi^2}  \Oc(\abs \xi^4)  + t \e^{-\frac t2 \abs \xi^2}  \BigO(\abs \xi^6) + \e^{-\frac t 2} \right)|\wh F(\xi)|_{\C^2_\xi} \\
		&   \ls \frac1{t^2} \left( (t |\xi|^2)^2 \e^{-t\abs \xi^2} + (t |\xi|^2)^3 \e^{-\frac t2 \abs \xi^2} + t^2 \e^{-\frac t 2} \right) |\wh F(\xi)|_{\C^2_\xi}, 
	\end{aligned}
\end{equation}
for all $|\xi |<\delta$. Hence, after integration,
\begin{equation} \label{eq:u-u0-dt}
	\|\wh{\partial_t u}(t) - \wh{\partial_t u}_0 (t) \|_{L^2(|\xi|<\delta)}
	\ls \frac1{t^2} \|F\|_{H^1 \oplus L^2}.
\end{equation}

\stepp \emph{Conclusion:}
The proof is completed by combining~\eqref{eq:u0.high}, \eqref{eq:u.high}, \eqref{eq:u.interm}, \eqref{eq:u-u0}, \eqref{eq:u-u0-nabla}, \eqref{eq:u-u0-dt} and unitarity of the Fourier transform.
\end{proof}

We define $\kk = H^1(\R^d) \cap \dot H\inv(\R^d)$, with norm
\[
\nr{f}_{\kk}^2 = \nr{f}_{H^1(\R^d)}^2 + \nr{f}_{\dot H\inv(\R^d)}^2 =  \int_{\R^d} \big(\abs \xi^2+ 1 + \abs \xi^{-2} \big) \big|\wh f(\xi)\big|^2 \diff \xi.
\]

\begin{lemma} \label{lem:lower-bound-heat}
Suppose $\Omega = \R^d$, $q (x)\equiv 0$ and $a (x)\equiv 1$. Let $\wt \rho_0,\wt \rho_1, \wt \rho_2 : \R_+ \to \R_+$ and $T_0>0$ be such that for all $f_0 \in \kk$ and $t \ge T_0$ we have
\begin{equation} \label{eq:hyp-rho-heat}
\begin{aligned}
\nr{\e^{t\Delta} f_0}_{L^2} & \leq \wt \rho_0(t) \nr{f_0}_{\kk},\\
\nr{\nabla \e^{t\Delta} f_0}_{L^2} & \leq \wt \rho_1(t) \nr{f_0}_{\kk},\\
\nr{\partial_t  \e^{t\Delta} f_0}_{L^2} = \nr{\Delta  \e^{t\Delta} f_0}_{L^2} & \leq \wt \rho_2(t) \nr{f_0}_{\kk}.
\end{aligned}
\end{equation}
For every $\eps>0$ there exists $t_0 \geq T_0$ such that for all $j \in \{0,1,2\}$ and $t \geq t_0$ we have
\begin{equation} \label{eq:concl-rho-heat}
\wt \rho_j(t) \geq  \frac {\cc_j - \eps}{t^{\frac {j+1} 2}}.
\end{equation}
\end{lemma}

\begin{proof}
Let $j \in \{0,1,2\}$ and, without loss of generality, consider $\eps \in \big(0,\frac {\cc_j}2 \big)$ with $\eps \le \frac4{T_0}$. Let $0 \leq \eta_1 <\eta_2 \leq 2$ be such that 
\begin{equation}
\inf_{s \in [\eta_1,\eta_2]} s^{j+1}\e^{-2s} \geq \left(\cc_j -\frac\eps2\right)^2,	
\end{equation}
(notice that the maximum value $\cc_j^2 \in (0,1)$ is reached at $\frac {j+1}2 \in (0,2)$). 
For fixed $t \ge t_0 := \frac{4}{\eps} \ge T_0$, we consider
$f_0 \in \Sc(\R^d)$ such that
\begin{equation} %\label{eq:support-xi}
\wh f_0(\xi) \neq 0 \quad \implies \quad \eta_1 \leq t \abs \xi^2 \leq \eta_2.
\end{equation}
In particular, we have $f_0 \in \kk$. Notice that if $\abs \xi^2 \leq \frac {\eta_2} t \leq \frac\eps2$, then 
%\begin{equation*} 
%\frac {c_0^2}{\abs \xi^2} - (c_0-\eps)^2 \left( \frac 1 {\abs \xi^2}  + 1 + \abs \xi^2\right)
%= \frac \eps {\abs \xi^2} \left( 2 c_0 - \eps  - (c_0 - \eps)^2 \frac {\abs \xi^2 + \abs \xi^4}{\eps} \right) \geq 0.
%\end{equation*}
%
\begin{equation}
	\frac {\left(\cc_j-\frac{\eps}{2}\right)^2}{\abs \xi^2} - (\cc_j-\eps)^2 \left( \frac 1 {\abs \xi^2}  + 1 + \abs \xi^2\right)
	\ge \frac {\eps(\cc_j-\eps)} {\abs \xi^2} \left( 1 - (\cc_j - \eps) \frac {\abs \xi^2 + \abs \xi^4}{\eps} \right) \geq 0;
\end{equation}
recall that  $\cc_j \in (0,1)$ and $\eps < \frac{\cc_j}{2}$. Then for all $\xi \in \R^d$ we have 
\begin{align}
|\xi|^{2j} \e^{-2t |\xi|^2} |\wh f_0(\xi)|^2
& = \e^{-2t\abs \xi^2} (t \abs \xi^2)^{j+1} \mathds 1_{[\eta_1,\eta_2]}(t \abs \xi^2) \frac {|\widehat{f_0}(\xi)|^2}{t^{j+1} \abs \xi^2}\\
& \geq \frac {\left(\cc_j - \frac \eps2 \right)^2 |\widehat{f_0}(\xi)|^2}{t^{j+1} \abs \xi^2}\mathds 1_{[\eta_1,\eta_2]}(t \abs \xi^2)\\
& \geq \frac {\left(\cc_j -\eps \right)^2}{t^{j+1}} \left(\frac 1 {\abs \xi^2} + 1 + \abs \xi^2 \right) |\wh f_0(\xi)|^2.
\end{align}
After integration over $\xi \in \R^d$ and by the Plancherel theorem,
\[
\wt \rho_j(t) \nr{f_0}_{\kk} \ge  \nr{(-\Delta)^{\frac j 2} \e^{t\Delta} f_0}_{L^2} \geq \frac {\cc_j - \eps}{t^{\frac {j+1} 2}} \nr{f_0}_{\kk}. \qedhere
\]
%Since this holds for any $c_0 \in \big(\frac {\cc_j}2 , \cc_j\big)$, the conclusion follows.
\end{proof}

Finally we can prove Proposition \ref{prop:lower-bound}.

\begin{proof}[Proof of Proposition \ref{prop:lower-bound}]
Let $\rho_0: \R_+ \to \R_+$ such that \eqref{eq:hyp-rho} holds and let $f \in \kk$; notice that then $F = (f,0) \in \cK$ and $\nr{F}_\KK = \nr{f}_\kk$. Let $c > 0$ be given by Lemma~\ref{lem:comparison-heat} and let $T_0\ge 1$ such that
\begin{equation}
	c t^{-1} \leq \frac{\eps}{2\sqrt t}, \qquad t \ge T_0.
\end{equation}
Then for all $t \geq T_0$ we have
\begin{align*}
\nr{\e^{t\Delta} f}_{L^2} = \|u_0(t)\|_{L^2}
\leq \rho_0(t) \nr{F}_\KK + c t^{-1} \nr{F}_{H^1 \oplus L^2}
\leq \left( \rho_0(t) + \frac {\eps} {2\sqrt{t}}  \right) \nr{f}_{\kk}.
\end{align*}
By Lemma \ref{lem:lower-bound-heat} applied with $\frac\eps2$, there exists $t_0 \ge T_0$ such that for all $t \ge t_0$
\[
\rho_0(t) + \frac \eps {2\sqrt{t}} \geq \frac {\cc_0 - \frac \eps 2} {\sqrt{t}} .
\]
The proofs for $\rho_1(t)$ and $\rho_2(t)$ are similar.
\end{proof}

\begin{remark}\label{rem:almost.opt}
		For every $\eps>0$, one can construct an initial condition which up to a factor $t^\eps$ achieves the decay rates in~\eqref{eq:concl-rho-heat}. To see this, let $\chi \in C_c^\infty(\R^d;[0,1])$ with $\chi = 1$ on $B(0,\d)$ for some $\d > 0$. Let $f_0$ be the inverse Fourier transform of $\xi \mapsto \abs \xi^{1-\frac d 2 + \eps} \chi(\xi)$. Then $f_0  \in \kk$ and for $t > 0$ we have
	\begin{align*}
		\nr{\e^{t\Delta} f_0}_{L^2}^2
		& \geq \int_{B(0,\d)} \e^{-2t\abs \xi^2} \abs \xi^{2-d+2\eps} \diff \xi
		= (2t)^{-(1+\eps)} \int_{B(0,\sqrt {2t} \d)} \e^{-\abs \eta^2} \abs \eta^{2-d+2\eps} \diff \eta,
	\end{align*}
	so, for all $t \ge t_0$ with a fixed $t_0 > 0$,
	\[
	\nr{\e^{t\Delta} f_0}_{L^2} \gtrsim t^{-\frac 1 2 - \frac \eps 2}.
	\]
	Similarly, for all $t \ge t_0$,
	\[
	\nr{\nabla \e^{t\Delta} f_0}_{L^2} \gtrsim t^{-1  - \frac \eps 2}, \qquad \nr{\partial_t \e^{t\Delta} f_0}_{L^2} \gtrsim t^{-\frac 32  - \frac \eps 2}.
	\]
\end{remark}

\section{Remarks on non-uniformly positive damping}
\label{sec:non.up.damping}

We have stated Proposition \ref{prop:high.pos} about the high frequency estimate and hence the energy estimates of Theorem \ref{thm:decay} under the assumption that the damping coefficient $a$ is uniformly positive on $\Omega$. In this section, we discuss some settings where this assumption is not satisfied.

\subsection{Cases with the geometric control condition}
\label{ssec.GCC}

If the (unbounded) damping $a$ satisfies the geometric control condition, it is expected that the resolvent norm of $\Ac$ is bounded at $\pm \ii \infty$ as in Proposition \ref{prop:high.pos}. In the one-dimensional case, this was established in \cite{arnal2026resolvent-dwe}. (The spectral property $\sigma(\cA) \cap \ii \R \subset \{0\}$, which implies \eqref{eq:Gk.res.bd.infty.C} and is not discussed in \cite{arnal2026resolvent-dwe}, follows by a simple ODE argument.)

\begin{theorem}[{\cite[Thm.~3.5]{arnal2026resolvent-dwe}}]
\label{thm:antonio}
Let $\Omega = \R$ and let $\cA$ be as in \eqref{A.action} and \eqref{A.dom}. Let $0 \leq a, q \in \CiR$ satisfy the following conditions.
\begin{enumerate} [\upshape (i)]
\item \label{itm:a.incr.unbd} $a$ is unbounded at infinity:
\begin{equation}
\label{eq:a.unbd.incr}
\lim_{x \to \pm \infty} a(x) = + \infty,
\end{equation}
\item \label{itm:a.symbolclass} $a$ has controlled derivatives: 
\begin{equation}
\label{eq:a.symb}
\forall~ n \in \N, \, \exists ~C_n>0, \,\forall ~x \in \R \, : \, |a^{(n)}(x)| \le C_n \left(1 + a(x)\right) \langle x \rangle^{- n},
\end{equation}
\item \label{itm:q.symbolclass} $q$ has controlled derivatives:
\begin{equation}
\label{eq:q.symb}
\forall ~n \in \N, \,\exists~C_n'>0,\, \forall~x \in \R \, : \, |q^{(n)}(x)| \le C'_n  \left(1 + q(x)\right) \langle x \rangle^{- n},
\end{equation}
\item \label{itm:a.gt.q} $q$ is eventually not bigger than $a$:
\begin{equation}
\label{eq:a.gt.q}
\exists~ x_0 \geq 0, \,\exists~ K > 0, \,\forall~ |x| > x_0 \, : \,  q(x) \le K a(x).
\end{equation}
\end{enumerate}
Then
\begin{equation}\label{eq:Gk.res.bd.infty}
\| (\Ac - \ii b)^{-1} \|_{\cB(\HH)} \approx 1, \qquad |b| \to \infty,
\end{equation}
and, for any $\eps>0$,
\begin{equation}\label{eq:Gk.res.bd.infty.C}
\sup_{\la \in \ov {\C_+ \setminus \eps\D}}	\| (\Ac - \la)^{-1} \|_{\cB(\HH)} < \infty.
\end{equation}
\end{theorem}

In higher dimensions with $\Omega = \Rd$, the analogous uniform resolvent bound for the damping $a(x) = |x|^\beta$, $\beta>0$, follows from a result \cite[Thm.2.1]{Leautaud-2017-26} for Schr\"odinger operators with imaginary potentials $-\Delta + \ii |x|^\beta$ in $L^2(\Rd)$.
\begin{proposition}
\label{prop:x^beta}
Let $\cA$ be as in \eqref{A.action} and \eqref{A.dom} with $\Omega = \Rd$, $a(x) = |x|^\beta$, $\beta>0$, and $q=0$. Then, for any $\eps >0$,
\begin{equation}
\label{res.A.large.x.beta}
\sup_{\la \in \ov {\C_+ \setminus \eps\D}} \|(\Ac - \la)^{-1} \|_{\cB(\HH)} < \infty.
\end{equation}
\end{proposition}
\begin{proof}[Remarks on the proof of Proposition~\ref{prop:x^beta}]
	In \cite[Thm.~2.1]{Leautaud-2017-26} it is proved in particular that
	\begin{equation}\label{Schr.LL.est}
		\|(-\Delta + \ii |x|^\beta - \mu)^{-1}\|_{\cB(L^2)} \ls \mu^{-\frac{\beta}{2(\beta+1)}}, \qquad \mu \to  +\infty.
	\end{equation}
	Consider only $b >0$. Rescaling $T_{\ii b} = -\Delta +  \ii b |x|^\beta - b^2$ with a parameter $\sigma>0$ leads to
	\begin{equation}
		\|T_{\ii b}^{-1}\|_{\cB(L^2)}  = \sigma^2 \|(-\Delta +   \ii b \sigma^{2+\beta}  |x|^\beta - \sigma^2 b^2)^{-1}\|_{\cB(L^2)} .
	\end{equation}
	Selecting $\sigma$ such that $b \sigma^{2+\beta}  =1$, employing \eqref{Schr.LL.est} and considering that $T_{\ii b}^* = T_{-\ii b}$ (see~\cite[Thm.~2.4]{Freitas-2018-264}), we arrive at
	\begin{equation}
		\|T_{-\ii b}^{-1}\|_{\cB(L^2)} = \|T_{\ii b}^{-1}\|_{\cB(L^2)} \ls b^{-1}, \qquad b \to +\infty.
	\end{equation}
	The claim then follows by Lemma~\ref{lem:T.to.A} as there are no purely imaginary eigenvalues (since $T_{\ii b} f =0$ implies $b \|a^\frac12 f\|^2=\Im \frt_{\ii b}[f] = 0$ and $a(x) \neq 0$ a.e.~in $\Rd$, cf.~Theorem~\ref{thm:A.basic}~(i)).
\end{proof}

Since Assumption~\ref{asm:a.1} is satisfied for both settings in Theorem~\ref{thm:antonio} and Proposition~\ref{prop:x^beta} (in the latter even Assumption~\ref{asm:a.2} holds), the low frequency estimate from Theorem~\ref{thm:low} applies and we arrive at the same conclusions for the decay of solutions as in Theorem~\ref{thm:decay}.
\begin{corollary}
Let $\Omega$, $a$ and $q$ be as in Theorem~\ref{thm:antonio} or Proposition~\ref{prop:x^beta}. Then the solution $u(t)$ of~\eqref{dwe.2ndorder} satisfies the decay estimates \eqref{u.energy.est}--\eqref{u.L2.est} in Theorem~\ref{thm:decay}, where in the setting of Proposition~\ref{prop:x^beta} even~\eqref{dtu.L2.est.gamma}--\eqref{u.L2.est.gamma} hold.
\end{corollary}

\subsection{Example without the geometric control condition}
\label{ssec:noGCC}

Non-uniformly positive damping coefficients which do not satisfy the geometric control condition may lead to a different spectral behavior for $\Ac$ at infinity. We discuss the following example on a two-dimensional strip, where the resolvent is no longer bounded as $\la = \ii b \to \pm \ii \infty$.

In detail, let for some $n \in \N$
\begin{equation}\label{ex.asm}
\Omega = \R \times (-1, 1) \subset \R^2, \quad q(x) \equiv 0, \quad a(x,y) = x^{2n}, \quad (x,y) \in \Omega.
\end{equation}
Notice that $a(0,y)\equiv0$, i.e.~there is one trajectory in $\Omega$ with no damping. This results in eigenvalues approaching the imaginary axis at $\pm \ii \infty$ with a rate depending on $n$ (see Proposition~\ref{prop:ex.sp} below, and also \cite[Prop.~6.3]{Freitas-2018-264}). Consequently, the resolvent of $\Ac$ is unbounded at $\pm \ii \infty$ and we show in Proposition~\ref{prop:ex.res} that the rate predicted by the eigenvalues~\eqref{ev.strip.rate} is optimal.

The proofs are based on the spectral equivalence to the Schur complement $T_\la$, separation of variables and the subsequent analysis of one-dimensional operators. In particular, for $-\partial_y^2$ subject to Dirichlet boundary conditions in $L^2(-1,1)$, we have an orthonormal basis of $L^2(-1,1)$ given by (normalized) eigenfunctions $\left\{g_j\right\}_{j \in \N}$ such that
\begin{equation}\label{y.spec}
- g_j'' = 	\zeta_j g_j, \qquad \zeta_j = \left(\frac {j \pi}2\right)^2, \qquad g_j \in H_0^1(-1,1) \cap H^2(-1,1), \qquad j \in \N.
\end{equation}
Hence, the Schur complement
\begin{equation}\label{Tla.noGCC}
T_\la = -\Delta +  \la {x^{2n}} + \la^2
\end{equation}
acting in $L^2(\Omega)$ (defined by its quadratic form) is unitarily equivalent to the orthogonal sum of a sequence of operators acting in $L^2(\R)$, namely
\begin{equation}\label{T.decomp}
\begin{aligned}
T_\la  \simeq & \, \bigoplus_{j \in \N} T_{\la,j},
\\
T_{\la,j}   := & -\partial_x^2 +  \la x^{2n} + \la^2 + \zeta_j,
\\[2mm]
\Dom(T_{\la,j})  :=  & \, H^{2}(\R) \cap \Dom(x^{2n}), \qquad \quad  j \in \N,
\end{aligned}
\end{equation}
(see Appendix~\ref{app:separation} for details).
This decomposition allows us to characterize the spectrum of $\cA$ as well as its resolvent norm on the imaginary axis.

\begin{proposition}\label{prop:ex.sp}
Let $\cA$ be as in \eqref{A.action} and \eqref{A.dom} with $\Omega$, $a$ and $q$ as in \eqref{ex.asm}. Then
\begin{equation*}
\sigma_{\rm e2}(\cA) = (-\infty, 0]
\end{equation*}
and, for every fixed $k \in \N_0$, there exists a sequence $\{\la_{k,j}\}_{j \in \N}$ with $\la_{k,j}, \ov{\la_{k,j}} \in \sigma_{\rm p} (\cA)$ having the asymptotic behavior
\begin{equation}\label{ev.strip.rate}
\la_{k,j} = \frac{\ii \pi}{2} j + \frac{\mu_k}2  \left(\frac{\pi j}2\right)^{-\frac{n}{n+1}} \e^{\ii\frac{ \pi}{2} \frac{(n+2)}{(n+1)}} +  \BigO_k\left(j^{-\frac{3n+1}{n+1}}\right), \qquad j \to \infty.
\end{equation}
Here $\left\{\mu_k\right\}_{k \in \N_0}$ denote the eigenvalues of the self-adjoint anharmonic oscillator $-\partial_x^2 + x^{2n}$ in $\Lt(\R)$ on the domain $H^2(\R)\cap \Dom (x^{2n})$.
\end{proposition}

\begin{proof}
Since \eqref{a.unbd} is satisfied, we conclude that $\sigma(\Ac) \setminus (-\infty, 0]$ consists only of eigenvalues of finite multiplicities that may at most accumulate at $(-\infty, 0]$ and are symmetric with respect to the real axis. Furthermore, arguing as in \cite[Prop.~6.3]{Freitas-2018-264} where the case $n=1$ was treated, we have $(-\infty,0] = \sigma_{e2}(\cA)$.

To study the eigenvalues of $\Ac$, we employ the spectral equivalence \eqref{spec.equiv} and the decomposition \eqref{T.decomp}, which in a straightforward way implies
\begin{equation}
\sigma (T_\la) = \sigma_{\rm p} (T_\la) = \bigcup_{j \in \N} \sigma_{\rm p} \left(T_{\la,j}\right).
\end{equation}
We thus search for $\la \in \C$ with $\Re\la \le 0$, $\Im\la > 0$ and $0 \in \spp(T_{\la,j})$ for some $j \in \N$.
Employing a standard complex scaling argument for fixed $j \in \N$ (see e.g.~\cite[Prop.~6.1]{arnal2026resolvent-dwe} for details), one can prove that $0 \in \spp(T_{\la,j})$ if and only if 
\begin{equation}\label{ev.eq}
F(\la) \equiv F_{j,k}(\la) := \la^2 + \zeta_j + \mu_k \la^{\frac1{n+1}} = 0
\end{equation}
for some $k \in \N_0$ and $\zeta_j$ as in~\eqref{y.spec} . Here and in the following, the complex $(n+1)$-th root is taken as the holomorphic branch
\begin{equation}\label{root}
(\cdot)^\frac1{n+1} \,\, : \,\, \C \setminus (-\infty, 0] \to \C, \quad r\e^{\ii\varphi} \mapsto r^\frac1{n+1} \e^{\ii \frac\varphi {n+1}}, \quad \f \in (-\pi,\pi).
\end{equation}
We define $F$ as a holomorphic function on the domain
$
\{z \in \C \,: \, \arg z \in (\pi/2 - \eps, \pi) \}
$
with some $\eps >0$, which contains the range of sought solutions.

We fix $k \in \N_0$. To obtain the existence of the solutions to~\eqref{ev.eq} and their asymptotic behavior as $j \to \infty$, we guess the first terms of the expansion and apply the Rouch\'e theorem suitably. From now on, asymptotic relations are understood as $j \to \infty$ which will be considered large enough if needed. Assume that $\la$ is a solution of \eqref{ev.eq} for some large $j$. Taking powers in~\eqref{ev.eq}, it follows that
\begin{equation}
(\la^2 + \zeta_j)^{n+1} + (-1)^n \la \mu_k^{n+1} = 0.
\end{equation}
Since $\abs{\zeta_j} \approx j^2$, a solution $\la$ is necessarily of order $j$ and then, more precisely,
\begin{equation}\label{la.approx.j}
|\la| \approx j, \qquad r_j:=\la^2 + \zeta_j = \BigO\left(j^{\frac1{n+1}}\right).
\end{equation}
From this we further derive
\begin{equation*}
\la^2 = -\zeta_j + r_j = -\zeta_j \left(1 + \BigO\left(j^{-\frac{2n+1}{n+1}}\right)\right),
\end{equation*}
which by the requirement $\Im \la >0$ then implies
\begin{equation*}
\la =  \ii\sqrt{\zeta_j} + \BigO\left(j^{-\frac{n}{n+1}}\right).
\end{equation*}
Using this relation, we finally arrive at
\begin{equation}
\begin{aligned}
\la - \ii \sqrt{\zeta_j} = \frac{\la^2 + \zeta_j}{\la + \ii \sqrt{\zeta_j}} & = 
- \frac{ \mu _k \la^\frac{1}{n+1}}{2 \ii\sqrt{\zeta_j} + \BigO\left(j^{-\frac{n}{n+1}}\right)} \\
& = -  \frac{\mu _k}2 \frac{\left(\ii \sqrt{\zeta_j}\right)^\frac{1}{n+1}}{\ii \sqrt {\zeta_j}} \frac{\left(1 + \BigO \left(j^{-\frac{2n+1}{n+1}} \right)\right)}{\left(1 + \BigO \left(j^{-\frac{2n+1}{n+1}}\right)\right)} \\
& = -\frac{\mu _k}2 \left(\ii \sqrt{\zeta_j} \right)^{-\frac{n}{n+1}} \left( 1 + \BigO \left( j^{-\frac{2n+1}{n+1}} \right)\right).
\end{aligned}
\end{equation}
From the last identity we can easily read off that the first two terms in the expansion should be
\begin{equation}
\la_0 := \frac{\ii \pi}{2} j + \frac{\mu_k}2 \left(\frac{\pi j}2\right)^{-\frac{n}{n+1}}\e^{\ii \frac{\pi}{2} \frac{n+2}{n+1}} \neq 0.
\end{equation}

Considering a small parameter $\alpha \in \C$, we are looking for zeros of the function $f(\alpha) := F (\la_0 + \alpha)$. Note that  $f$ is holomorphic for $|\alpha| \leq 1$ if $j$ is large enough. A straightforward calculation yields
\begin{equation}\label{f.alph}
\begin{aligned}
f(\alpha) = 2 \alpha \la_0 + \alpha^2 & - \left(\frac{\mu_k}2 \right)^2 \left( \frac{\pi j}2 \right)^{-\frac{2n}{n+1}} \e^{\ii \pi \frac{1}{n+1}}  \\
& \qquad \qquad - \mu_k \left(\frac{ \pi j}2 \right)^{\frac1{n+1}} \e^{\ii \frac{\pi}{2} \frac{1}{n+1}} + \mu_k  (\la_0 + \alpha)^{\frac1{n+1}}.
\end{aligned}
\end{equation}
Using the formula $(z_1z_2)^{1/(n+1)} = z_1^{1/(n+1)} z_2^{1/(n+1)}$ (which is justified with the convention~\eqref{root} when $z_1$ is purely imaginary and $z_2$ is in a small neighborhood of $1$), we obtain
\[
 (\la_0 + \alpha)^{\frac1{n+1}}  = \left( \frac{\ii \pi j}{2}  \right)^{\frac1{n+1}} \left( 1 + x \right)^\frac1{n+1}, \qquad x = \frac{\mu_k}2 \left(\frac{\pi j}2 \right)^{- \frac{2n+1}{n+1}} \e^{\ii \frac\pi2 \frac1{n+1}} - \frac{2 \ii\alpha }{\pi j}.
\]
The expansion $(1+x)^\frac{1}{n+1} = 1 + \frac x {n+1} + \Oc(\abs x^2)$ then gives
\begin{align*}
 (\la_0 + \alpha)^{\frac1{n+1}}
 & = \left( \frac {\pi j}{2} \right)^{\frac 1 {n+1}}  \e^{\ii 
 	\frac\pi2 \frac {1}{n+1}} + \frac {\mu_k}{2(n+1)} \left( \frac {\pi j}{2} \right)^{-\frac {2n}{n+1}} \e^{{\ii\pi}\frac 1{n+1}} \\
& \quad - \frac {\ii \a}{n+1} \left( \frac {\pi j}{2} \right)^{-\frac n {n+1}} \e^{\ii 
	\frac\pi2 \frac {1}{n+1}} + \Oc_k \big(j^{-\frac {2n+1}{n+1}}\big),
\end{align*}
where the remainder depends on $k$ but is uniform in $\abs \a \leq 1$.
Putting together the above, one calculates
\begin{equation}	\label{eq:f.exp}
\begin{aligned}
f(\alpha)  = 2\alpha \la_0 + \alpha^2 & + \frac{\mu_k}{2} \e^{\ii \pi \frac{1}{n+1}} \left( \frac{1}{n+1} -  \frac{\mu_k}{2}\right) \left(\frac{\pi j}{2}\right)^{-\frac{2n}{n+1}}  \\
& - \frac{\ii \mu_k}{n+1} \e^{\ii \frac{\pi}2 \frac{1}{n+1}} \alpha \left(\frac{\pi j}{2}\right)^{-\frac{n}{n+1}} + \BigO_k \left(j^{-\frac{2n+1}{n+1}}\right).
\end{aligned}
\end{equation}

We apply the Rouch\'e theorem (see e.g.~\cite[Thm.~V.3.8]{Conway-1978-11}) to $f$ and $g: \alpha \mapsto 2\alpha \la_0$. To this end, we bound their difference by the sum of their moduli on a circle $|\alpha| = C j^{-(3n+1)/(n+1)}$ (matching the sought remainder) for some $C \equiv C_k>0$ to be selected suitably. For such $\alpha$, considering~\eqref{eq:f.exp} we have
\begin{equation}
	\label{eq:f-g.exp}
	\begin{aligned}
		|f(\alpha)-g(\alpha)| & = \frac{\mu_k}{2} \left|\frac{1}{n+1} - \frac{\mu_k}{2}\right| \left(\frac{\pi j}{2}\right)^{-\frac{2n}{n+1}} \left(1 + \Oc_k (j^{-\frac1{n+1}})\right), \\
		|g(\alpha)| & =  \pi C  j^{-\frac{2n}{n+1}} \left(1 + \Oc_k (j^{-\frac{2n+1}{n+1}})\right) .
	\end{aligned}
\end{equation}
Choosing first 
\begin{equation}
	C = \frac{\mu_k}{\pi} \left|\frac{1}{n+1} - \frac{\mu_k}{2}\right| \left(\frac{\pi}{2}\right)^{-\frac{2n}{n+1}},
\end{equation}
(depending only on $k$), we can find $j_0$ such that, for every $j \ge j_0$ we have $|\alpha|\le 1$ and \eqref{eq:f.exp} applies uniformly in $\alpha$. It then follows from \eqref{eq:f-g.exp} that, possibly increasing $j_0$, we have
\begin{equation}
	|f(\alpha) - g(\alpha) | < |g(\alpha)|, \qquad |\alpha| = C j^{-\frac{3n+1}{n+1}}.
\end{equation} 
By the Rouch\'e theorem, we conclude that, for every $j \ge j_0$, $f$ has exactly one zero $\alpha_{k,j}$ with $|\alpha_{k,j}| < C j^{-\frac{3n+1}{n+1}}$, and in turn $F$ has exactly one zero $\la_{k,j}$ with
\begin{equation}
|\la_{k,j} - \la_0| \le C j^{-\frac{3n+1}{n+1}},
\end{equation}
which is precisely the claim~\eqref{ev.strip.rate}.
\end{proof}

\begin{proposition}\label{prop:ex.res}
Let $\cA$ be as in \eqref{A.action} and \eqref{A.dom} with $\Omega$, $a$ and $q$ as in \eqref{ex.asm}. Then
\begin{equation}\label{ex.res.G}
\|(\Ac-\ii b)^{-1}\|_{\cB(\cH)} \ls |b|^{\frac{n}{n+1}}, \qquad |b| \to \infty.
\end{equation}
(Notice that this rate is optimal due to Proposition~\ref{prop:ex.sp}).
\end{proposition}
\begin{proof}
We show below that
\begin{equation}
\| T_{\ii b}^{-1}\|_{\cB(L^2(\Omega))} \ls |b|^{-\frac1{n+1}}, \qquad b \in \R \setminus \{0\},
\end{equation}
and the claim then follows by Lemma~\ref{lem:T.to.A}.
We consider only  $b > 0$, the case $b < 0$ follows by using the fact that $T_{\la}^* = T_{\overline\la}$ (see \cite[Thm.~2.4]{Freitas-2018-264}).
To this end, we use that as a consequence of \eqref{T.decomp} (see Appendix~\ref{app:separation} for details), we have
\begin{equation}
\label{Tla.res.orth}
\| T_{\la}^{-1}\|_{\cB(L^2(\Omega))}  = \sup_{j \in \N} \| T_{\la,j}^{-1}\|_{\cB(L^2(\R))} .
\end{equation}

Let $j \in \N$ be fixed. For $\sigma > 0$ (to be chosen suitably below) we define the family of unitary operators 
\begin{equation}
\left(U_{\sigma} u\right)(x) := \sigma^{\frac12} u(\sigma x), \qquad u \in \Lt(\R), \qquad x \in \R.
\end{equation}
Then  we have
\begin{equation}
\left(U_{\sigma} T_{\ii b,j} U_{\sigma}^{-1} u\right)(x) = - \frac1{\sigma^2} u''(x) + \left(\ii b \sigma^{2n} x^{2n} - b^2 +\zeta_j\right) u(x), \qquad x \in \R,
\end{equation}
for any $u \in U_{\sigma}\left(\Dom(T_{\ii b,j})\right)$. Setting
$\sigma := b^{-\frac1{2(n+1)}}$,
we define the operator family
\begin{equation}
S_{b,j} := \sigma^2 U_{\sigma} T_{\ii b,j} U_{\sigma}^{-1} = -\partial_x^2 + \ii x^{2n} + \sigma^2(\zeta_j - b^2).
\end{equation}
The operator
\begin{equation}
H_n = -\partial_x^2 + \ii x^{2n}, \qquad \Dom(H_n) = H^{2}(\R) \cap \Dom\left(x^{2n}\right),
\end{equation}
is known to be m-accretive and it follows from \cite[Eq.~(7.2)]{arnal2023resolvent} that there exist constants $a_0$, $K_0 > 0$ such that
\begin{equation}
\| (H_n - a)^{-1} \|_{\cB(L^2(\R))} \le K_0 a^{-\frac{n}{2n+1}}, \qquad a \geq a_0.
\end{equation}
Moreover,
\begin{equation}
\Num(H_n) \subset \left\{ \la \in \C\, : \, \Re \la \ge0, \, \Im \la \ge 0 \right\}
\end{equation}
and $H_n$ has no real eigenvalues (by complex scaling, the eigenvalues of $H_n$ are complex-rotated eigenvalues of the self-adjoint operator $-\partial_x^2 + x^{2n}$ in $L^2(\R)$). Hence, there exists $K_1>0$ such that
\begin{equation}
\| (H_n - a)^{-1} \|_{\cB(L^2(\R))} \le K_1, \quad a \in \R.
\end{equation}
From this, we conclude that, for all $j \in \N$ and $b>0$, also $\| S_{b,j}^{-1} \|_{\cB(L^2(\R))} \le K_1$. Returning to $T_{\ii b,j}$, we arrive at
\begin{equation*}
\| T_{\ii b,j}^{-1} \|_{\cB(L^2(\R))} \le K_1 b^{-\frac1{n+1}}, \qquad b > 0, \quad j \in \N. \qedhere
\end{equation*}
\end{proof}

To deduce time decay from the previous resolvent estimates, we employ the following semigroup result allowing for the growth of the resolvent norm at $\pm \ii \infty$.

\begin{theorem}[{\cite[Thm.~8.4]{batty2016fine}}]\label{thm:bct.2}
Let $(\e^{tA})_{t \ge 0}$ be a bounded $C_0$-semigroup on a Hilbert space $\mathscr H$ with generator $A$. Assume that $\sigma(A) \cap \ii\R = \{0\}$ and that there exist $\alpha \geq1$ and $\beta >0$ such that
\begin{equation}\label{res.2.sing}
\| (A - \ii b)^{-1} \|_{\cB(\mathscr H)} =
\begin{cases}
\BigO\left(|b|^{-\alpha}\right), &|b| \to 0,
\\[1mm]
\BigO\left(|b|^\beta\right), &|b| \to \infty.
\end{cases}
\end{equation}
Then
\begin{equation}\label{exp.res.2.sing.ab}
\|\e^{tA} A^\alpha (A-1)^{-(\alpha+\beta)} \|_{\cB(\mathscr H)} = \BigO\left(t^{-1}\right), \qquad t \to \infty,
\end{equation}
and
\begin{equation}\label{exp.res.2.sing}
\hspace{1.1cm} \|\e^{tA} A (A-1)^{-2} \|_{\cB(\mathscr H)} = \BigO\big(t^{-\frac 1 \gamma}\big), \qquad t \to \infty,
\end{equation}
where $\gamma=\max\{\alpha,\beta\}$.

Conversely, if \eqref{exp.res.2.sing} holds for some $\gamma>0$, then \eqref{res.2.sing} holds for $\alpha = \max\{\gamma,1\}$ and $\beta = \gamma$.
\end{theorem}

In conclusion, since the singularity of the resolvent at $\pm \ii \infty$ in \eqref{ex.res.G} is milder than at zero (see Theorem~\ref{thm:low}), the energy decay of the solutions for suitable initial data remains as for previous cases. 

\begin{corollary}
Let $\Ac$ be as in \eqref{A.action} and \eqref{A.dom} with $\Omega$, $a$ and $q$ as in \eqref{ex.asm}. Then there exists $C>0$ such that, for any initial data $F =(f,g) \in \widetilde{\KK} :=\Dom(\Ac) \cap \Ran(\Ac)$ and $t \geq 0$, the solution $u (t)$ of \eqref{dwe.2ndorder} satisfies
\begin{align}
\|\partial_t u(t)\|_{L^2} + \|\nabla u(t)\|_{L^2} & \leq C \|F\|_{\widetilde{\KK}} \langle t \rangle^{-1},
\label{u.energy.est.2.sing}
\end{align}
where
\begin{equation}\label{tilde.I(v).def}
\|F\|_{\widetilde{\KK}}^2 : = \nr{F}_{\HH}^2 + \nr{\Ac F}_{\HH}^2 + \nr{\Ac\inv F}_{\HH}^2 = \|F\|_{\KK}^2 + \|\nabla g\|_{L^2}^2 + \|\Delta f - a g\|_{L^2}^2.
\end{equation}
\end{corollary}
\begin{proof}
The claim follows by Theorem~\ref{thm:bct.2}, in particular \eqref{exp.res.2.sing}, with $\mathscr{H} = \HH$ and $A=\Ac$. To this end, note that Assumption~\ref{asm:a.1} is satisfied, thus the resolvent estimate \eqref{r.A.small.b.H-H} holds. Moreover, we have also \eqref{ex.res.G} and it is easy to see from \eqref{ev.eq} that $\sigma(\cA)\cap \ii \R = \{0\}$. Finally, similarly as in \eqref{semigr}, we arrive at
\begin{align}
\|\e^{t\Ac} F\|_{\HH} &= \|\e^{t\Ac} \Ac(\Ac-1)^{-2} (\Ac-1)^{2} \Ac^{-1} F\|_{\HH}
\\
&
\ls \langle t \rangle^{-1} \|\Ac F - 2 F + \Ac^{-1} F\|_{\HH}
\ls \langle t \rangle^{-1} \|F\|_{\widetilde \KK}. \qedhere
\end{align}
\end{proof}

\section{Comparison with the result of Ikehata-Takeda}
\label{sec:comp.IT}

As said in the introduction, Theorem \ref{thm:decay} generalizes a previous result by Ikehata and Takeda, based on an approximation of the possibly unbounded damping by a sequence of bounded dampings and a modified Morawetz multiplier method. Their precise result is the following.

\begin{theorem}[{\cite[Thm.~1.2]{ikehata2020uniform}}]
\label{thm:IT}
Let $\Omega=\Rd$ with $d \geq 3$, let $q=0$ and let $a \in C(\Rd)$ with $a(x)\geq a_0 >0$ for all $x \in \Rd$. If the initial data
\begin{equation}\label{v.IT.1}
(f,g) \in \left(H^1(\Rd) \cap L^1(\Rd)\right) \times \left(L^2(\Rd) \cap L^1(\Rd) \right)
\end{equation}
further satisfy
\begin{equation}\label{v.IT.2}
a f \in L^1(\Rd) \cap L^2(\Rd),
\end{equation}
then there exists a unique weak solution
\begin{equation}\label{weak.sol.space}
u \in L^\infty\left((0,\infty); H^1(\Rd)\right) \cap W^{1,\infty}\left((0,\infty); L^2(\Rd)\right)
\end{equation}
to
\eqref{dwe.2ndorder} (in the sense of \eqref{weak.sol.def} below) satisfying
\begin{equation}\label{IT.estimate}
\|u(t)\|_{L^2} \leq C I_{00} \langle t \rangle^{- \frac 12},
\qquad
\|\partial_t u(t)\|_{L^2} + \|\nabla u(t)\|_{L^2} \leq C I_{00} \langle t \rangle^{-1},
\end{equation}
where $C>0$ and
\begin{equation}\label{I00.def}
I_{00}^2 = \|f\|_{L^2}^2 + \|\nabla f\|_{L^2}^2 + \|a f\|_{L^1}^2 + \|a f\|_{L^2}^2 + \|g\|_{L^1}^2 + \|g\|_{L^2}^2.
\end{equation}
\end{theorem}

Recall that our main Theorem~\ref{thm:decay} concerns a solution of~\eqref{dwe.2ndorder} constructed by means of the semigroup.
Ikehata and Takeda, on the other hand, study weak solutions of~\eqref{dwe.2ndorder} in~\cite{ikehata2020uniform} (with $\Omega = \Rd$), i.e.~functions $u : [0,\infty) \times \Omega \to \C$ such that, for any $\phi \in C_c^\infty([0,\infty) \times \Omega)$, we have the following identity, where in particular all the integrals have to be finite:
\begin{equation}\label{weak.sol.def}
\begin{aligned}
& \int_0^\infty \int_\Omega u(t,x) \big(\partial_{tt}\phi(t,x) - a (x) \partial_t\phi(t,x) -(\Delta - q(x)) \phi(t,x)\big)  \; \dd x \, \dd t =
\\
&
\qquad   \int_\Omega g(x) \phi(0,x) \, \dd x - \int_\Omega f(x) \partial_t\phi(0,x) \, \dd x
+ \int_\Omega  a(x) f(x) \phi(0,x) \, \dd x.
\end{aligned}
\end{equation}
This identity arises by formally transporting all (time and space) derivatives to the test function $\phi$. It turns out that, for the initial conditions  we consider, every mild solution has the time-space regularity~\eqref{weak.sol.space} and is a weak solution of~\eqref{dwe.2ndorder}.

\begin{lemma}
Let the assumptions of Theorem~\ref{thm:decay} hold. For $F \in \KK$, let $U(t) = (u(t),\partial_t u(t))= \e^{t\Ac}F = \e^{t\AK} F$ be the unique mild solution of~\eqref{Cauchy} in $\HH$ and $\KK$. Then $u(t)$ satisfies~\eqref{weak.sol.space} (with $\Rd$ replaced by $\Omega$) and is a weak solution of \eqref{dwe.2ndorder} according to~\eqref{weak.sol.def}.
\end{lemma}

\begin{proof}
Fix $F=(f,g) \in \cK$. By Theorem \ref{thm:decay},  we have
\begin{equation}\label{est.weak.sol.space}
\|u(t)\|_{H^1} + \|\partial_t u(t)\|_{L^2}  \ls \langle t \rangle^{-\frac12} \, \|F\|_{\KK} \ls 1, \qquad t \ge 0,
\end{equation}
and thus~\eqref{weak.sol.space}.  

It remains to prove~\eqref{weak.sol.def}.
	Note first that by the properties of a mild solution, $U : [0,\infty) \to \cK$ is continuous. Hence $u :[0,\infty) \to \cD_{\frt}$ is continuous by Proposition \ref{prop:Ran.A}.
Setting
\begin{equation}
	w(t) := \int_0^t u(s) \, \dd s , \qquad t \ge 0,
\end{equation}
it follows that 
\begin{equation} \label{der-g}
	w \in C^1 ([0,\infty);\cD_{\frt}), \qquad  \partial_t w (t) = u(t),
\end{equation}
and from the second component of \eqref{Cauchy.int} we get
	\begin{equation}\label{u1.mild}
			\partial_t u(t) = (\Delta-q) w(t) - a (u(t) - f) + g, \qquad t \geq 0. \\
\end{equation}

Consider $\phi \in C_c^\infty ([0,\infty) \times \Omega;\R)$ (the identity for $\phi \in C_c^\infty ([0,\infty)\times \Omega)$ follows by considering $\Re \phi$ and $\Im \phi$ separately) and let $T>0$ and $\Sigma \subset \Rd$ be open and bounded with  $\overline \Sigma \subset \Omega$ such that $\supp \phi \subset [0,T) \times \Sigma$.
By \eqref{der-u-f2} and \eqref{u1.mild}, we have
\begin{equation}\label{mild1}
	\begin{aligned}
		& \int_0^T \int_\Sigma u (t, x) \partial_{tt}\phi (t, x) \, \dd x \, \dd t = \int_0^T \innp{u(t)}{\partial_{tt}\phi(t)}_{L^2} \diff t \\
		& \qquad = -\int_0^T \innp{\partial_t u(t)}{\partial_{t}\phi(t)}_{L^2} \diff t - \innp{u(0)}{\partial_t \phi(0)}_{L^2}\\
		& \qquad =  - \int_0^T \langle (\Delta-q) w(t) -  a (u(t) - f) + g, \partial_t \phi (t) \rangle_{L^2} \, \dd t
		- \langle f ,\partial_t \phi (0) \rangle_{L^2}.
	\end{aligned}
\end{equation}
From $a^\frac12\in L^2_\loc(\Omega)$ and the continuity of $u : [0,\infty) \to \cD_{\frt}$, we conclude 
\begin{equation} \label{au}
\sup_{t \in [0,T]} \|au(t)\|_{L^1(\Sigma)} < \infty, \qquad af = au(0) \in L^1(\Sigma).
\end{equation}
By \eqref{der-g} and \eqref{au}, for fixed $t \in [0,T]$ we can write 
\begin{equation}\label{mild2}
\begin{aligned}
	& \big\langle (\Delta-q) w(t) -  a (u(t) - f), \partial_t \phi (t) \big\rangle_{L^2}  \\[2mm]
	&  \qquad= \big((\Delta-q) w(t) -  a (u(t) - f), \partial_t \phi (t) \big)_{\Dc_t^* \times \Dc_t}   \\
	&  \qquad= \big( (\Delta-q) w(t), \partial_t \phi (t) \big)_{\cW^* \times \cW}  -  \int_\Sigma a (x) u(t, x) \partial_t \phi (t, x) \, \dd x \\
	& \qquad \qquad \qquad \quad  + \int_\Sigma a (x)  f(x)\partial_t \phi (t, x) \, \dd x.
\end{aligned}
\end{equation}

Using~\eqref{au}, we can split the time integration in the last line of \eqref{mild1} according to the summands in \eqref{mild2}; recall the identity~\eqref{u1.mild} and that $\partial_t u \in C([0,\infty);L^2(\Omega))$ by the properties of our mild solution. For the first summand, by~\eqref{CD.def}, the Green Formula, integrating by parts in time and using \eqref{der-g}, we obtain
\begin{equation}\label{mild3}
\begin{aligned}
& \int_0^T \big( (\Delta-q) w(t), \partial_t \phi (t) \big)_{\cW^* \times \cW} \, \dd t \\
& \qquad \qquad \qquad = - \int_0^T \big(\langle \nabla w(t) , \nabla \partial_t \phi(t) \rangle_{L^2} +  \langle w(t), q \partial_t \phi(t) \rangle_{L^2} \big) \diff t \\
& \qquad \qquad \qquad  = \int_0^T \langle w(t) , \partial_t ( (\Delta- q) \phi(t) ) \rangle_{L^2} \diff t \\
& \qquad \qquad \qquad =  - \int_0^T \langle u(t) ,  (\Delta- q) \phi(t) \rangle_{L^2} \diff t .
\end{aligned}
\end{equation}
Finally, since $af + g \in L^1(\Sigma)$ (see \eqref{au}), applying dominated convergence to exchange the integral and time derivative leads to
\begin{equation}
	\begin{aligned}
			& \int_0^T \int_\Sigma \big(  a(x) f (x) + g(x) \big) \partial_t \phi (t, x) \, \dd x\, \dd t  \\ 
			& \qquad \qquad \qquad \qquad = \int_0^T \frac{\dd}{\diff t} \int_\Sigma \big(  a(x) f (x) + g(x) \big) \phi (t, x) \, \dd x\, \dd t \\
			& \qquad \qquad \qquad \qquad = - \int_\Sigma \big(  a(x) f (x) + g(x) \big) \phi (0,x)  \, \dd x .
	\end{aligned}
\end{equation}
The proof is completed by putting together the above equation with~\eqref{mild1}, a version of \eqref{mild2} where the summands are integrated over $t \in [0,T]$, and~\eqref{mild3}.
\end{proof}

Employing the above lemma, it follows that our decay result contains Theorem~\ref{thm:IT}. To prove this, the Sobolev inequality is used. This indicates the origin of the difficulty in treating low dimensions $d=1,2$ by the method in \cite{ikehata2020uniform}.

\begin{proposition}
Let the assumptions of Theorem~\ref{thm:IT} hold and let $F =(f,g)$ be as in~\eqref{v.IT.1} and~\eqref{v.IT.2}. Then $F \in \KK$ with $\|F\|_{\KK} \ls I_{00}$ and Theorem~\ref{thm:decay} applies. The solution $u(t)$ in Theorem~\ref{thm:decay} coincides with the unique weak solution of~\eqref{dwe.2ndorder} in~\eqref{weak.sol.space}. In particular,~\eqref{u.energy.est},~\eqref{u.partialt.est}~and~\eqref{u.L2.est} imply~\eqref{IT.estimate}.
\end{proposition}

\begin{proof}
Clearly $\Omega = \R^d$ with $d \ge 3$, $q(x)=0$ and $0 < a_0 \le a (x) \in C(\R^d)$ satisfy the basic assumptions of Theorem~\ref{thm:decay}. For $F = (f,g)$ as in \eqref{v.IT.1} and \eqref{v.IT.2}, we have $F \in \cD_\frt \oplus L^2(\R^d)$; notice that $f \in \Dom(a) \subset \Dom\big(a^\frac12\big)$. Moreover, from $d \ge 3$ and a straightforward application of H\"older's inequality it follows that
\begin{equation}
h:=af + g \in L^1(\R^d) \cap L^2(\R^d) \subset L^{\frac{2d}{d+2}}(\R^d),
\qquad \|h\|_{L^{\frac{2d}{d+2}}} \le \|h\|^{\frac2d}_{L^1} \|h\|_{L^2}^{\frac{d-2}{d}}.
\end{equation}
Combining H\"older's inequality and the Sobolev inequality~\cite[Thm.~4.31]{Adams-2003} with the above, we further estimate
\begin{equation}
\begin{aligned}
\left| \int_{\R^d} h(x) \overline{\varphi(x)} \, \dd x \right| &  \le \|h\|_{L^{\frac{2d}{d+2}}} \| \varphi\|_{L^{\frac{2d}{d-2}}} \\
& \ls \|h\|_{L^{\frac{2d}{d+2}}} \| \nabla \varphi\|_{L^2}
\ls \|h\|^{\frac2d}_{L^1} \|h\|_{L^2}^{\frac{d-2}{d}} \|\varphi \|_{\cW},
\end{aligned}
\end{equation}
for all $\varphi \in C_c^\infty(\R^d)$. This implies that $h \in \cW^*$ and by Young's inequality
\begin{equation}
\|h\|_{\cW^*}	\ls  \|h\|^{\frac2d}_{L^1} \|h\|_{L^2}^{\frac{d-2}{d}} \ls \|h\|_{L^1} + \|h\|_{L^2}.
\end{equation}
In particular, considering~\eqref{KK.def}, we have $F \in \KK$. Finally, returning to $h=af + g$ and employing the triangle inequality, we arrive at $\|F\|_{\KK} \ls I_{00}$ (see \eqref{KK.norm} and \eqref{I00.def}).
\end{proof}

\appendix

\section{Description of the space \texorpdfstring{$\Wc$}{W}}
\label{app:W0}

The space $\Wc$ was defined in the introduction as the Hilbert space completion of $\CcOm$ w.r.t.~the polar form of
\begin{equation}
\|f \|^2_{\Wc} = \|\nabla f \|_{L^2}^2 + \|q^\frac12 f \|_{L^2}^2.
\end{equation}
Thus, an element of $\Wc$ is an equivalence class of Cauchy sequences in $\CcOm$ for the norm $\nr{\cdot}_{\Wc}$. If $\{f_n\}_n$ is a representative of this equivalence class, then $\{\nabla f_n\}_n$ and $\{q^{\frac 12} f_n\}_n$ are Cauchy sequences in $L^2(\Omega)$, and we define $\nabla f$ and $q^{\frac 12} f$ as their respective limits. These definitions do not depend on the choice of the representative $\{f_n\}_n$. Then we can define $(-\Delta+q) \in \Bc(\Wc,\Wc^*)$ by
\[
((-\Delta+q)f, g)_{\Wc^* \times \Wc} = \innp{\nabla f}{\nabla g}_{L^2} + \langle q^{\frac 12} f,q^{\frac 12} g\rangle_{L^2}, \qquad f,g \in \Wc.
\]

First recall that if \eqref{Om.q.Poincare} holds,
then $\Wc = H_0^1(\Omega) \cap \Dom (q^\frac12)$, while in general $\Wc$ need not be a subspace of $L^2(\Omega)$. 
Our goal is to identify every element of $\Wc$ with a function $f \in L^2_\loc(\Omega)$, so that $\nabla f$ and $q^{\frac 12} f$ can be understood in the sense of distributions. However, if $q=0$ a.e.~on a connected component of $\Omega$, then $f$ is only determined up to a constant thereon. We show that $\cW$ is included in a suitable quotient space.
To this end, $\Cc_q$ shall denote the space of locally constant functions $w$ (i.e.~constant on every connected component of $\Omega$) such that $q w = 0$ a.e. on $\Omega$.

\begin{proposition}\label{prop:D_q}
$\Wc$ is a closed subspace of
\begin{equation}
\dot \cD_q := \cD_q \, \big/ \,  \Cc_q, \qquad \cD_q := \left\{ f \in L^1_{\rm loc}(\Omega) \, :\, \nabla f \in L^2(\Omega)^d, \, \, q^\frac12 f \in L^2 (\Omega) \right\}.
\end{equation}
\end{proposition}
The action of $\Delta - q$ is then well-defined on $\dot \cD_q$ (and thus on $\Wc$) in the usual distributional sense. To prove the proposition, we show that $\dot \cD_q$ is a Hilbert space and that the completion $\Wc$ can be embedded in it. We will use the following standard result.

\begin{lemma} \label{lem:Uj}
Let $\omega \neq \emptyset$ be an open connected subset of $\R^d$. There exists a non-decreasing (for inclusion) sequence $\{U_j\}_{j \in \N}$ of open and connected Lipschitz subsets of $\omega$ such that $\bigcup_{j \in \N} U_j = \omega$.
\end{lemma}

\begin{proof}
For $x \in \R^d$ and $j \in \N$, we set $\mathfrak c_x^j = [x_1,x_1 + 2^{-j}] \times \dots \times [x_d,x_d + 2^{-j}]$. Fix $x_0 \in \omega$ and define
\[
\mathfrak C_j = \left(\bigcup_{x \in \mathfrak X_j} \mathfrak c_x^j \right)^{\circ}, \qquad \mathfrak X_j = \set{x \in 2^{-j} \Z^d \, : \, \mathfrak c_x^j \subset \omega \cap B(x_0,j)}.
\]
For $j \ge j_0$, where $j_0 \in \N$ is large enough (such that $x_0 \in \mathfrak C_j$), we denote by $U_j$ the connected component of $\mathfrak C_j$ which contains $x_0$. Note that $U_j$ is open and Lipschitz and that $U_j \subset U_{j+1}$ for all $j \ge j_0$.

Consider $x \in \omega$ arbitrary. Since $\omega$ is open and connected, there is $\g \in C([0,1];\omega)$ with $\g(0) = x_0$ and $\g(1) = x$. There exists $j \ge j_0$ such that
\begin{equation}
	\overline{B(\g(t),2^{-j} \sqrt d)} \subset \omega \cap B(x_0,j), \qquad t \in [0,1].
\end{equation}
We set
\[
\mathfrak C_{j,\g} = \left(\bigcup_{y \in \mathfrak X_{j,\g}} \mathfrak c_y^j \right)^{\circ} , \qquad \mathfrak X_{j,\gamma} = \set{y \in \mathfrak X_j \, : \, \mathfrak c_y^j \cap \g([0,1]) \neq \emptyset}.
\]
Then $\g([0,1]) \subset \mathfrak C_{j,\g} \subset \mathfrak C_j$. This implies that $x$ is in the same connected component of $\mathfrak C_j$ as $x_0$, so $x \in U_j$. Finally, we have $\bigcup_{j \ge j_0} U_j = \omega$ and the lemma is proved.
\end{proof}

 The next lemma generalizes~\cite[Lem.~II.6.2]{Galdi-2011}.

\begin{lemma}
The space $\dot \cD_q$ is Hilbert when equipped with the polar form of
\begin{equation}
\| [f]_{\sim} \|_{\dot \cD_q} := \|f\|_{\Wc}.
\end{equation}
\end{lemma}

\begin{proof}
From the properties of $\|\cdot\|_\cW$, it follows that $\dot \cD_q$ is an inner product space.
Note therefore that if $\|f \|_{\Wc} = 0$ then in particular $\nabla f = 0$, so $f$ is locally constant.
We also have $q^\frac12 f =0$, so $qf=0$ and hence $f \in \cC_q$. Positive definiteness in $\dot \Dc_q$ is implied.

For the completeness, consider a Cauchy sequence $\{[f_n]_\sim\}_n \subset \dot D_q$, represented by a sequence $\{f_n\}_n \subset D_q$ being Cauchy w.r.t.~$\|\cdot\|_\cW$. To prove the claim, one needs to find a limit
\begin{equation}\label{eq:conv}
f \in \cD_q, \qquad \lim_{n\to \infty} \|f_n-f\|_{\Wc} \to 0.
\end{equation}
By the completeness of $L^2$-spaces, there exist $g\in L^2(\Omega)^d$ and $h\in L^2(\Omega)$ such that 
\begin{equation}\label{eq:complete}
\lim_{n\to \infty} \| \nabla f_n - g\|_{L^2} = 0, \qquad   \lim_{n\to \infty} \| q^\frac12 f_n - h\|_{L^2} =0,
\end{equation}
(implying the respective convergence in $\cD'(\Omega)^d$ and $\cD'(\Omega)$).

Let $\omega$ be a connected component of $\Omega$ and $\{U_j\}_j$ the sequence given by Lemma \ref{lem:Uj}. For every $j \in \N$, since $\left\{\nabla f_n|_{U_j}\right\}_n$ is Cauchy in $L^2(U_j)$ and $U_j$ is bounded, Lipschitz and connected, it follows from the Poincar\'e--Wirtinger inequality and a completeness argument that there exists $v_j \in L^2(U_j)$ satisfying
\begin{equation}\label{eq:lim.Poincare}
\lim_{n\to\infty} \| (f_n - \overline{f_n}^j) - v_j \|_{L^2(U_j)} = 0, \qquad \overline{f_n}^j = \frac 1 {\abs{U_j}} \int_{U_j} f_n(x) \, \dd x.
\end{equation}
As this implies convergence in $\Dc'(U_j)$, we get $\nabla v_j = g\vert_{U_j}$ in the sense of distributions.
Using that $q \in  L^1_{\rm loc}(\Omega)$, we also have
\begin{equation}\label{eq:lim.q1}
\lim_{n\to \infty} q^\frac12 (f_n - \overline{f_n}^j) = q^\frac12 v_j \quad \text{in} \quad \cD'(U_j).
\end{equation}

Assume that $q$ is not a.e.~zero on $\o$; w.l.o.g.~we can assume that it is not a.e.~zero on $U_1$ (hence on any $U_j$, $j \in \N$). From~\eqref{eq:complete} and~\eqref{eq:lim.q1}, it follows that the sequence of scalars $\{\overline{f_n}^j\}_n$ converges to some $c_j \in \C$ satisfying $c_j q^\frac12 = h - q^\frac12 v_j$ on $U_j$. Setting $\f_j = v_j + c_j$ we have on $U_j$
\[
\nabla \f_j = g \quad \text{and} \quad  q^{\frac 12} \f_j = h.
\]
Let $j,k \in \N$ with $j\leq k$. Since $\nabla(\f_j - \f_k) = 0$ on $U_j$, there exists $c \in \C$ such that $\f_j - \f_k\vert_{U_j} = c$. On the other hand, $q^{\frac 12} (\f_j-\f_k)=0$ on $U_j$, so $c = 0$. Then $\f_k$ coincides with $\f_j$ on $U_j$ and we can define a function $f_\o$ by $f_\o \vert_{U_j} = \f_j$, $j \in \N$.

Now assume that $q = 0$ on $\omega$. In this case we set $c_1 = 0$. Then, as above, we see by induction in $j \in \N$ that there exists $c_{j+1} \in \C$ such that $v_{j+1}\vert_{U_j} + c_{j+1} = v_j + c_j$. Again we set $\f_j = v_j + c_j$ and we can consider the unique function $f_\o$ which agrees with $\f_j$ on $U_j$ for all $j \in \N$.

In any case, we have a function $f_\o \in L^2_\loc(\omega)$ such that $\nabla f_n  \to \nabla f_\o$ and $q^{\frac 12} f_n \to q^{\frac 12} f_\o$ in $L^2_\loc(\omega)$. We proceed similarly on all connected components of $\Omega$. This defines a function $f \in L^2_\loc(\Omega)$ such that $\nabla f_n  \to \nabla f$ and $q^{\frac 12} f_n \to q^{\frac 12} f$ in $L^2_\loc(\Omega)$. By \eqref{eq:complete}, we deduce that $\nabla f = g$ and $q^{\frac 12} f = h$ belong to $L^2(\Omega)$, so $f \in \Dc_q$ and $[f_n]_\sim \to [f]_\sim$ in $\dot \Dc_q$ as $n \to \infty$.
\end{proof}

\begin{proof}[Proof of Proposition~\ref{prop:D_q}]
The space $\CcOm$ can be understood as a subspace of $\dot \cD_q$ by means of the linear injection
\begin{equation}
\left\{ \begin{aligned}
\CcOm & & \to & \quad \dot \cD_q \\
f & & \mapsto & \quad [f]_{\sim}.
\end{aligned}\right.
\end{equation}
The completion $\Wc$ can be then identified with the closure of this subspace in $\dot \cD_q$.
\end{proof}

\section{Separation of variables}
\label{app:separation}

\begin{lemma}
\label{lem:ext.by.inv}
Let $A$ and $B$ be two linear operators in a Hilbert space $\HH$ such that $A \subset B$. Assume further that $A^{-1}$ and $B^{-1}$ exist and are everywhere defined on $\HH$. Then $A = B$.
\end{lemma}
\begin{proof}
Let $x \in \Dom(B)$ and set $y := B x$. Then $x' := A^{-1} y \in \Dom(A)$ and $A x' = y = B x$. Since we also have $A x' = B x'$ (because $B$ extends $A$ by assumption) and $B$ is injective, it follows that $x' = x$, i.e.~$x \in \Dom(A)$.
\end{proof}
\begin{lemma}
\label{lem:B.orthog.decomp}
Let $\HH, \HH_1, \HH_2$ be separable Hilbert spaces and $U:\HH_1 \otimes \HH_2 \to \HH$ a unitary operator. Let $B$ be a linear operator in $\HH$ and assume further that:
\begin{enumerate}[\upshape (i)]
\item \label{lem:B.i} there exist linear operators $B_1, B_2$ in $\HH_1, \HH_2$, respectively, such that
\begin{equation}\label{B.orthg.decomp}
B_1 \otimes I_{\HH_2} + I_{\HH_1} \otimes B_2 \subset U^{-1}BU,
\end{equation}
where the left hand side is defined on the linear span of simple tensors  in $\Dom(B_1)\otimes\Dom(B_2)$;
\item \label{lem:B.ii} there exists an orthonormal basis $\{e_j\}_{j \in \N} \subset \HH_2$ of eigenvectors of $B_2$ with corresponding eigenvalues $\{\zeta_j\}_{j \in \N}$;
\item \label{lem:B.iii} there exists $\la_0 \in \rho(B)$ such that
\begin{equation}\label{B1j.res.unif.ubound}
\la_0 \in \underset{j \in \N}{\bigcap} \rho(B_{1} + \zeta_j), \qquad \underset{j \in \N}{\sup} \| (B_{1}+ \zeta_j- \la_0)^{-1}\|_{\cB(\HH_1)} < \infty.
\end{equation}
\end{enumerate}
Consider the linear operator
\begin{equation*}
\begin{aligned}
\Dom(A) &:= \bigg\{f = \sum_{j \in \N} f_j \otimes e_j \in \HH_1 \otimes \HH_2: \; f_j \in \Dom(B_{1}), \; j \in \N,\\
&\hspace{6cm}  \sum_{j \in \N} \| (B_{1}+ \zeta_j) f_j \|_{\HH_1}^2 < \infty \bigg\},\\
A f &:= \sum_{j \in \N} \left( (B_{1}+ \zeta_j) f_j\right) \otimes e_j,
\end{aligned}
\end{equation*}
acting in $\HH_1 \otimes \HH_2$. Then $A = U^{-1} BU$ and
\begin{equation}\label{B.res.norm.orthg}
\| (B - \la_0)^{-1} \|_{\cB(\HH)} = \underset{j \in \N}{\sup} \| (B_{1} + \zeta_j - \la_0)^{-1} \|_{\cB(\HH_1)}.
\end{equation}
Moreover, in this case, for any $\la \in \C$ one has $\la \in \rho(B)$ if and only if $\la$ satisfies \eqref{B1j.res.unif.ubound} (with $\la_0=\la$) and then formula \eqref{B.res.norm.orthg} holds (with $\la_0=\la$).

\end{lemma}
\begin{proof}
It is evident from the definition of $A$ that the finite sums
\begin{equation}
f_N := \sum_{j=1}^N f_j \otimes e_j, \qquad f_j \in \Dom (B_1), \qquad j=1, \dotsc, N, \qquad N \in \N,
\end{equation}
form a core of $A$. The inclusion $A\subset U^{-1}BU$ then readily follows since $U^{-1}BU$ is closed, the $f_N$ lie in its domain (they are in the domain of the tensor in~\eqref{B.orthg.decomp}) and
\begin{equation}
Af_N = \sum_{j =1}^N \left( (B_{1}+ \zeta_j) f_j\right) \otimes e_j = \sum_{j=1}^N \left( B_{1} f_j\right) \otimes e_j + \sum_{j =1}^N f_j\otimes (\zeta_j e_j) = U^{-1}BU f_N.
\end{equation}

By assumption, $\{e_j\}_{j \in \N}$ is an orthonormal basis in $\HH_2$ and hence we can write any $g \in \HH_1 \otimes \HH_2$ as
\begin{equation}
g = \sum_{j \in \N} g_j \otimes e_j, \qquad g_j \in \HH_1,  \qquad \| g \|_{\HH_1 \otimes \HH_2}^2 = \sum_{j \in \N} \| g_j \|_{\HH_1}^2,
\end{equation}
(see e.g.~\cite[Prop.~II.4.2]{Reed1} and note that the $g_j$ are uniquely determined). Let us define the operator
\begin{equation*}
C_{\la_0} g := \sum_{j \in \N} ((B_{1} + \zeta_j - \la_0)^{-1} g_j) \otimes e_j, \qquad g \in \HH_1 \otimes \HH_2.
\end{equation*}
It is clear from \eqref{B1j.res.unif.ubound} that $C_{\la_0}$ is well-defined and bounded on $\HH_1 \otimes \HH_2$. Moreover, we have $C_{\la_0} g \in \Dom(A)$ for every $g \in \HH_1 \otimes \HH_2$ and
\begin{equation*}
(A - \la_0) C_{\la_0} g = \sum_{j \in \N} g_j \otimes e_j = g.
\end{equation*}
Similarly, for $g \in \Dom(A)$ we have
\begin{equation}
C_{\la_0} (A - \la_0) g = \sum_{j \in \N} g_j \otimes e_j = g.
\end{equation}
It follows that $\la_0 \in \rho(A)$ with $(A - \la_0)^{-1} = C_{\la_0}$. By Lemma~\ref{lem:ext.by.inv}, we conclude $A = U^{-1}B U$ and hence
\begin{equation}
(B - \la_0)^{-1} = U(A - \la_0)^{-1}U^{-1} =  UC_{\la_0}U^{-1}.
\end{equation}
The equality \eqref{B.res.norm.orthg} then follows since for any $g \in \HH_1 \otimes \HH_2$ we have
\begin{equation}
\| C_{\la_0} g \|_{\HH_1 \otimes \HH_2}^2 = \sum_{j \in \N} \| (B_{1}  + \zeta_j - \la_0)^{-1} g_j \|_{\HH_1}^2.
\end{equation}

After having established the unitary equivalence of $A$ and $B$, the final claim follows by repeating the above argument with $\la_0 = \la$.
\end{proof}

To justify \eqref{T.decomp} for a fixed $\la \in \C\setminus(-\infty,0]$, we apply Lemma~\ref{lem:B.orthog.decomp}. We begin by noting that
\begin{equation*}
U: \left\{ \begin{aligned}
\, \Lt(\R) \otimes \Lt(-1, 1) &\to \LtOm\\
f \otimes g \qquad \,\, \, \quad &\mapsto f(x) g(y)
\end{aligned} \right.
\end{equation*}
determines a unique unitary operator from $\HH_1 \otimes \HH_2 := \Lt(\R) \otimes \Lt(-1, 1)$ onto $\HH := \LtOm$ (see e.g.~\cite[Thm.~II.10~(a)]{Reed1}). Take $B := T_{\la}$ as in \eqref{Tla.noGCC} and
\begin{align}
\label{B1.def}B_1 &:= -\partial_x^2 + \la x^{2n} + \la^2,  & \Dom(B_1) & := H^2(\R) \cap \Dom(x^{2n}),\\
\label{B2.def}B_2 &:= -\partial_y^2,  & \Dom(B_2) & := H^2(-1,1) \cap H_0^1(-1,1).
\end{align}
One can then verify the inclusion in Lemma~\ref{lem:B.orthog.decomp} \ref{lem:B.i}. Indeed, for simple tensors $f \otimes g$ with $f \in \Dom(B_1)$ and $g \in \Dom (B_2)$, one can prove that
\begin{equation}
U(f \otimes g) \in \Dom (T_\la) = H^2(\Omega) \cap H_0^1(\Omega) \cap \Dom \left(x^{2n}\right)
\end{equation}
and moreover that
\begin{equation}
\begin{aligned}
T_\la U (f \otimes g) & = \left( - f'' + \la x^{2n}f + \la^2 f\right)(x)  g (y) + f(x) (-g'')(y) \\
& = U \left((B_1f) \otimes g + f \otimes (B_2g) \right).
\end{aligned}
\end{equation}
This extends to the linear span of simple tensors and assumption \ref{lem:B.i} follows. We also recall that $B_2$ has a complete orthonormal set of eigenvectors
\begin{equation}\label{B2.ef}
e_j(y) = c_j \left\{
\begin{aligned}
&\cos\left(\frac{\pi}{2} j y\right), \; &j = 1, 3, 5, \dots,\\
&\sin\left(\frac{\pi}{2} j y\right), \; &j = 2, 4, 6, \dots,
\end{aligned}
\right.
\end{equation}
with corresponding eigenvalues
\begin{equation}\label{B2.ev}
\zeta_j = \left(\frac{j \pi}{2} \right)^2, \qquad j \in \N.
\end{equation}
Finally, to verify the assumption \ref{lem:B.iii}, consider first the case $\Im \la \ge 0$. By simple geometric considerations, one sees that
\begin{equation}
\Num (T_\la) \cup \bigcup_{j \in \N} \Num (B_1+ \zeta_j) \subset \left\{ z \in \C : \Im z \ge \Im (\la^2) \right\}.
\end{equation}
It follows that all $\mu \in \C$ with $\Im \mu < \Im (\la^2)$ are in the resolvent set of $T_\la$ and $B_1+ \zeta_j$ for all $j \in \N$. Moreover, for such $\mu$ one has the resolvent bound
\begin{equation}
\| (B_1 + \zeta_j - \mu)^{-1}\|_{\cB(L^2(\R))} \le \frac1{\Im (\la^2) - \Im \mu}.
\end{equation}
Hence the assumption \ref{lem:B.iii} holds (uniformly) on any half plane $\Im \mu \le \Im (\la^2) - \eps$ with $\eps >0$. An analogous claim is true in the case $\Im \la <0$.

From the application of Lemma~\ref{lem:B.orthog.decomp} it follows that $T_\la$ is unitarily equivalent to an infinite direct sum of operators
\begin{equation}
U^{-1}T_\la U = \bigoplus_{j \in \N} \, (B_1 + \zeta_j) \otimes I_{\lspan \{e_j\}}
\end{equation}
acting in the direct sum of spaces
\begin{equation}
\bigoplus_{j \in \N} L^2 (\R) \otimes \lspan \{e_j\} = 	L^2(\R) \otimes L^2(-1,1).
\end{equation}
From this, however,~\eqref{T.decomp} follows immediately since
\begin{equation}
L^2 (\R) \otimes \lspan \{e_j\} \simeq L^2(\R), \qquad (B_1 + \zeta_j) \otimes I_{\lspan \{e_j\}} \simeq B_1 + \zeta_j = T_{\la,j}.
\end{equation}
In the above, for linear operators $T_1$ and $T_2$ acting in Hilbert spaces $\HH_1$ and $\HH_2$, respectively, we write $\HH_1 \simeq \HH_2$  if there exists a unitary operator $V : \HH_1 \to \HH_2$, as well as $T_1 \simeq T_2$ if moreover $VT_1 V^{-1} = T_2$. Note that the resolvent formula~\eqref{B.res.norm.orthg} now translates into~\eqref{Tla.res.orth}.

\bibliography{references}
\bibliographystyle{acm}

\end{document}